\documentclass[1 [leqno,12pt]{amsart}
\usepackage{amssymb, amsmath,latexsym,amsfonts,amsbsy, amsthm,mathtools,graphicx,CJKutf8,CJKnumb,CJKulem,color}
\usepackage{float}
\usepackage{hyperref}

\makeatletter
\newif\ifmsbmloaded@
\def\loadmsbm{\msbmloaded@true
  \font\tenmsb=msbm10 scaled 1\@ptsize00
  \font\sevenmsb=msbm7 scaled 1\@ptsize00
  \font\fivemsb=msbm5 scaled 1\@ptsize00
  \alloc@8\fam\chardef\sixt@@n\msbfam
  \textfont\msbfam=\tenmsb
  \scriptfont\msbfam=\sevenmsb
  \scriptscriptfont\msbfam=\fivemsb
  }
\@addtoreset{equation}{section}

\makeatother \loadmsbm

\def\R{\mathbb R}

\def\C{\mathbb C}

\def\T{{\mathbb T}}
\def\ir{\mathrm{i}}

\def\no{\noindent}

\def\f#1#2{\frac{#1}{#2}}
\def\f{\frac}

\def\pa{\partial}
\def\p{\partial}

\def\om{\omega}
\def\Om{\Omega}

\def\na{\nabla}
\def\la{\lambda}

\def\ph{\varphi}

\def\cL{{\mathcal L}}

\def\rd{{\mathrm d}}

\newcommand{\beq}{\begin{equation}}
\newcommand{\eeq}{\end{equation}}
\newcommand{\ben}{\begin{eqnarray}}
\newcommand{\een}{\end{eqnarray}}
\newcommand{\beno}{\begin{eqnarray*}}
\newcommand{\eeno}{\end{eqnarray*}}
\allowdisplaybreaks

\newtheorem{Theorem}{Theorem}[section]
\newtheorem{Lemma}[Theorem]{Lemma}

\newtheorem{remark}{Remark}[section]
\newtheorem{Proposition}{Proposition}[section]

\begin{document}

\title[Linear stability of pipe Poiseuille flow ]
{Linear stability of pipe Poiseuille flow at high Reynolds number regime}

\author{Qi Chen}
\address{School of Mathematical Science, Peking University, 100871, Beijing, P. R. China}
\email{chenqi940224@gmail.com}

\author{Dongyi Wei}
\address{BICMR, Peking University, 100871, Beijing, P. R. China}
\email{jnwdyi@163.com}

\author{Zhifei Zhang}
\address{School of Mathematical Science, Peking University, 100871, Beijing, P. R. China}
\email{zfzhang@math.pku.edu.cn}

\date{\today}

\maketitle

\begin{abstract}
In this paper, we prove the linear stability of the pipe Poiseuille flow for general perturbations at high Reynolds number regime.
This is a long-standing  problem since the experiments of Reynolds in 1883. Our work lays a foundation for the theoretical  analysis 
of hydrodynamic stability of pipe flow, which is one of the oldest yet unsolved problems of fundamental fluid dynamics. 
\end{abstract}

\tableofcontents

\section{Introduction}

Since the experiments of Reynolds in 1883, many physicists and applied mathematicians have devoted enormous efforts  to understand the transition mechanism from laminar to turbulent flow via theoretical analysis,  experiments and numerical simulations \cite{Rey, DR, Yag, Gro, TTR, Ker}. The stability of pipe Poiseuille flow and its transition to turbulence is of special interest, because it is one of the simplest mathematical idealizations of Reynolds experiments.  All theoretical and numerical work indicates that this flow is linearly stable for any Reynolds number \cite{DR}, yet it could exhibit transition to turbulence in practice when Reynolds number exceeds a certain critical number.  Another important laminar flow is the plane Couette flow, which is also  linearly  stable for any Reynolds number but also becomes turbulent in practice. This is a well-known paradox of hydrodynamic stability theory.  A resolution of this paradox is a long-standing unsolved problem in fluid mechanics. 

This kind of transition called subcritical transition or by-pass transition in physical literature is very different from those due to the existence of growing modes(or unstable eigenvalues). There are many attempts to resolve this paradox \cite{Cha}.  One important attempt is focused on the study of transient growth effect due to the nonnormality of the linearized operator, which was initiated by Trefethen at al \cite{TTR}. For shear flows, the linearized operator of the Navier-Stokes equations is highly non-normal at high Reynolds   number.  As a result,  the solutions of the linear stability problem could exhibit an algebraic growth before they begin to decay exponentially, which is so-called transient growth. This in turn may  lead to the instability and transition to turbulence via complex nonlinear interactions \cite{Zik, Beg, BB, TTS, Sch-JFM, SB, Hen}. Moreover, experimental and numerical evidence suggest that the transition is extremely sensitive to the size and structure of the perturbations \cite{RSB, Mes, Zik}.  An important open question proposed by Trefethen et al \cite{TTR} is to study the {\bf transition threshold problem}, which is concerned with how much disturbance will lead to the instability of the flow and the dependence of disturbance on the Reynolds number. The main motivation behind seeking the threshold is that insights into the transition process can be uncovered. The mathematical version of transition threshold problem formulated in \cite{BGM-AM, BGM-BAMS}  can be stated as follows.

{\it
Given a norm $\|\cdot\|_X$, find a $\beta=\beta(X)$ so that
\beno
&&\|u_0\|_X\le Re^{-\beta}\Longrightarrow  {stability},\\
&&\|u_0\|_X\gg Re^{-\beta}\Longrightarrow  {instability}.
\eeno
}
The exponent $\beta$ is referred to as the transition threshold. {\bf  It was conjectured by Trefethen et al. in \cite{TTR}  that $\beta\le 1$ for plane Couette flow and pipe Poiseuille flow}(see also \cite{MT}). Direct numerical simulations  by Lundbladh, Henningson and  Reddy \cite{LHR} showed 
that $\beta$ is between $1$ and $\f 74$ for the plane Couette and Poiseuille flows. Formal asymptotic analysis  by Chapman 
\cite{Cha} showed that $\beta=1$ for plane Couette flow and $\beta=\f32$ for plane Poiseuille flow. For pipe Poiseuille flow,
experimental and numerical results seem to indicate that $\beta$ is between 1 and $\f32$ \cite{DM, TCH, Mes, HJM, PM, MM}.

In our recent works \cite{WZ, CWZ}, we confirm the transition threshold conjecture proposed by Trefethen et al \cite{TTR} for the 3D Couette flow.   An important indication of our result is that the transient growth may be the key mechanism leading to the instability of the flow.  The work \cite{CWZ} is  based on two important ingredients: linear stability of the Couette flow at high Reynolds number and  pseudospectral bound of the linearized operator(or resolvent estimates). The linear stability of Couette flow was proved by Romanov in a very beautiful paper \cite{Rom}. In a joint work \cite{CLWZ} with Li,  we established the resolvent estimates of the linearized operator around the Couette flow, which encoporate three important linear effects: inviscid damping, enhanced dissipation and boundary layer. 

Let us refer to \cite{BGM-MAMS-1, BGM-MAMS-2, BGM-AM, BWV, LWZ, GMM, GGN, MZ, IMM, WZ-SCM, AH, CDE, GNR, CEW} and references therein for some recent important progress on the stability of shear flows such as Couette flow, Poiseuille flow and Kolmogorov flow.

To understand the transition to turbulence for the pipe Posieuille flow by using the method of theoretical analysis, as we mentioned before, a key step is to study the linear stability of this flow. Although it is widely believed that this flow is linearly stable for any Reynolds number, there  remains a challenging  mathematical problem for a rigorous proof of linear stability. 
In a recent work \cite{GG}, Gong and Guo considered the linear stability for the axisymmetric perturbations without swirl. \smallskip

The goal of this paper is to prove the linear stability of pipe Posieuille flow for general 3D perturbations at high Reynolds number regime. Let $\Omega=\big\{x=(x_1,x_2,z): r=\sqrt{x_1^2+x_2^2}<1,\ z\in \T_L\big\}$ be a pipe, where $\T_L$ is a torus of period $L_z$. 
We study the 3D incompressible Navier-Stokes equations at high Reynolds number regime in a pipe $\Omega$:
\begin{align}
\left\{
\begin{aligned}
&\partial_t v-\nu\Delta v+v\cdot\nabla v+\na P=0,\\
&\nabla\cdot v=0,\\
&v|_{r=1}=0,\quad v(0,x)=v_0(x).
\end{aligned}
\right.\label{eq:NS}
\end{align}
where $v(t,x)=(v^1,v^2,v^3)$ is the velocity, $P(t,x)$ is the pressure, and $\nu\sim Re^{-1}$ is the viscosity coefficient.

The  pipe Poiseuille flow ${u}^{(p)}=(0,0,1-r^2)$ is a steady solution of \eqref{eq:NS}.
Let $u=v-u^{(p)}$ be the perturbation of the velocity and $V=1-r^2$. Then it holds that 
\begin{align*}
\left\{
\begin{aligned}
&\partial_t u-\nu\Delta u+V\partial_z u+(0,0,u\cdot\nabla V)+u\cdot\nabla u+\na P=0,\\
&\nabla\cdot u=0,\\
&u|_{r=1}=0,\quad u(0,x)=u_0(x).
\end{aligned}
\right.
\end{align*}
Thus, the linearized Navier-Stokes equations around $u^{(p)}$ take the form
\begin{align}\label{eq:LNS-ev}
\left\{
\begin{aligned}
&\partial_t u-\nu\Delta u+V\partial_z u+(0,0,u\cdot\nabla V)+\na P=0,\\
&\nabla\cdot u=0,\\
&u|_{r=1}=0,\quad u(0,x)=u_0(x).
\end{aligned}
\right.
\end{align}

To state our result, we first  introduce some notations. 
We denote by $L^2_{\sigma}(\Omega)$ and ${H}_{\sigma}^{k}(\Omega)$ the closure of the set of vectors, which are smooth and solenoidal in $ \Omega$ and vanish on $ \partial\Omega,$ in the topology, respectively, of $L^2(\Omega)$ and Sobolev space $H^{k}(\Omega)$.
We introduce the operator $\mathcal{L}_\nu: D(\mathcal{L}_\nu)=H^2(\Om)\cap H^1_\sigma(\Omega) \to L^2_\sigma(\Omega)$ defined by
\beno
\mathcal{L}_\nu{u}=\nu \mathbb{P}\Delta{u}-\mathbb{P}\big(V\partial_z {u}+(0,0,{u}\cdot\nabla V)\big),
\eeno
where $\mathbb{P}$ is the Leray-Helmholtz projection from $L^2(\Om)$ to $L^2_\sigma(\Om)$. Then $\mathcal{L}_\nu$ is a closed linear operator in $L^2_{\sigma}(\Omega)$.  We denote by $m(\nu)$ the upper bound of the real part of points of the spectrum of $\mathcal{L}_\nu$. Moreover, $\mathbb{P}\Delta $ is a dissipative self-adjoint operator with compact resolvent, and $\mathcal{L}_\nu$ is a relatively compact perturbation of $\nu\mathbb{P}\Delta $. In particular, the spectrum of $\mathcal{L}_\nu$ is always discrete. Then we have
\begin{align}\label{def:mnu}
&m(\nu)=\sup\big\{\mathbf{Re}(s): s{u}=\mathcal{L}_\nu{u},\ 0\neq {u}\in D(\mathcal{L}_\nu)\big\}.
\end{align}
Let $\mathcal{H}_0=\big\{f\in L^2(\Omega): \partial_zf=0\big\}$ and $\mathcal{H}=\big\{f\in L^2(\Omega): P_0f=0\big\}$, where $P_0f=\f 1 {L_z}\int_{T_L}f(x,y,z)dz$. Then $L^2(\Omega)=\mathcal{H}_0\oplus\mathcal{H}$. Since $\mathcal{L}_\nu$ is invariant in $\mathcal{H}_0$, we have 
\beno
m(\nu)=\max\big(m_0(\nu),m_1(\nu)\big), 
\eeno
where
\begin{align*}
&m_0(\nu)=\sup\big\{\mathbf{Re}(s):s{u}=\mathcal{L}_\nu{u},\ 0\neq {u}\in D(\mathcal{L}_\nu)\cap\mathcal{H}_0\big\},\\
&m_1(\nu)=\sup\big\{\mathbf{Re}(s):s{u}=\mathcal{L}_\nu u,\ 0\neq {u}\in D(\mathcal{L}_\nu)\cap\mathcal{H}\big\}.
\end{align*}

Now our main result is stated as follows.

 \begin{Theorem}\label{thm:stability}
There exist constants $c\in(0,1)$ independent of $\nu$ and $c_0\in (0,1)$ independent of $\nu, L_z$ so that 
if $0<\nu(L_z+1)\leq c_0$,  then we have $m(\nu)\leq -c\nu$.
\end{Theorem}

 Our task is to show that 
\beno
m_0(\nu)\leq -c\nu,\quad m_1(\nu)\leq -c\nu.
\eeno
In fact, for $m_1(\nu)$, we can prove a better upper bound
\beno
m_1(\nu)\leq -c|\nu/L_z|^{\f12}.
\eeno
This will lead to a decay rate like $e^{-c\nu^\f12 t}$ of the semigroup $e^{t\mathcal{L}_\nu}$,
which is much faster than the decay rate $e^{-c\nu t}$ of the heat semigroup. This is the so-called  enhanced dissipation, which is due to mixing mechanism induced by the pipe Poiseuille flow. The same effect has been observed  for the Couette flow and Kolmogorov flow \cite{CLWZ, WZ-SCM, IMM}.  

The enhanced dissipation and inviscid damping effects play an important role for nonlinear stability of the Couette flow. 
Therefore, we believe that  Our work lays a foundation for the theoretical  analysis 
of hydrodynamic stability of pipe Poiseuille flow. In near future, we will try to solve {\bf Transition threshold conjecture} 
for pipe Poiseuille flow.
\smallskip

Let us conclude the introduction by some notations. \smallskip 

\begin{itemize} 

\item For $\la\in \C$, we denote $\la_r=\mathbf{Re} \la$ and $\la_i=\mathbf{Im} \la$;

\item We denote by $C>0, c_i\in (0,1)(i=1,2,\cdots)$ constants independent of $\nu, l, n, \la$, which may be different from line to line;

\item $\widehat{\Delta}=\partial_r^2+\dfrac{1}{r}\partial_r-\dfrac{n^2}{r^2}-l^2$;

\item  $\widehat{\Delta}_1=\widehat{\Delta}-\dfrac{2rl^2}{n^2+r^2l^2}\partial_r$ and $\widehat{\Delta}_{(1)}=\dfrac{1}{r}\partial_r(r\partial_r)-\dfrac{1+r^2l^2}{r^2}$;

\item We denote $\langle f, g\rangle=\int_0^1f(r)\overline{g}(r)r\rd r$ and $\langle f, g\rangle_s=\int_0^sf(r)\overline{g}(r)r\rd r$;

\item We denote the norm $\|f\|_{L^p}^p=\int_0^1|f(r)|^pr\rd  r$;

\item We denote the norm $\|f\|_1^2=\|\partial_r f\|_{L^2}^2+n^2\|f/r\|_{L^2}^2+l^2\|f\|_{L^2}^2$;

\item $E=\displaystyle\int_0^1\left(\dfrac{r|\partial_rW|^2}{n^2+r^2l^2}+\dfrac{|W|^2}{r}\right)\mathrm{d}r$;

\item $A(s)=|l(\lambda-s^2)/\nu|^{\f12}+|ls/\nu|^{\f13}+|l|+|n/s|, \ A=A(1)$.

\end{itemize}

\section{A key formulation of the linearized system}

To estimate $m(\nu)$,  we  consider the linearized resolvent system 
\begin{align}\label{eq:LNS-res}
\left\{
\begin{aligned}
&s{u}-\nu\Delta {u}+V\partial_z {u}+(0,0,{u}\cdot\nabla V)+\na P=0,\\
&\nabla\cdot {u}=0,\quad {u}|_{r=1}=0.
\end{aligned}
\right.
\end{align}

Although the linearized system \eqref{eq:LNS-res} looks simple, it is highly difficult to solve  due to the unknown pressure and the term $u\cdot \na V$. The most critical innovation of this paper is the introduction of a new formulation, 
which plays the key role for linear stability analysis. To our knowledge, this formulation is completely new,  and  might  shed some light 
on theoretic and numerical analysis of hydrodynamic stability of the pipe Poiseuille flow in near future.

We introduce the cylindrical coordinate with $(x_1,x_2,x_3)=(r\cos\theta,r\sin\theta,z)$ and 
\beno
e_r=\left(\frac{x_1}{r}, \frac{x_2}{r},0\right), e_{\theta}=\left(-\frac{x_2}{r},\frac{x_1}{r},0\right), e_z=(0,0,1).
\eeno
We denote 
\beno
\widehat{\Delta}=\partial_r^2+\dfrac{1}{r}\partial_r-\dfrac{n^2}{r^2}-l^2,\quad \widehat{\Delta}_1=\widehat{\Delta}-\dfrac{2rl^2}{n^2+r^2l^2}\partial_r.
\eeno

Throughout this section, we assume that $u$ is a smooth solution of  \eqref{eq:LNS-res} with $P_0u=0$.

\subsection{Formulation in the axisymmetric case} 
In order to inspire how to introduce a good formulation in general case, we first consider the axisymmetric fluid. In this case, the velocity $v$ takes  the form
\beno
v=v_r(t,r,z)e_r+v_{\theta}(t,r,z)e_{\theta}+v_z(t,r,z)e_z.
\eeno
We denote 
\begin{align*}
   &\omega_r=-\partial_z(v_{\theta}),\quad \omega_{\theta}=\partial_z(v_r)-\partial_r(v_z),\quad \omega_z=\dfrac{1}{r}\partial_r(rv_{\theta}).
\end{align*}
Let $\Omega=\dfrac{\omega_{\theta}}{r}$ and  $J=\dfrac{\omega_r}{r}$. Then $(\Omega, J)$ satisfies 
  \begin{align*}
     \left\{\begin{aligned}
     &\partial_t\Omega+v\cdot\nabla \Omega-\nu\Big(\Delta+\dfrac{2}{r}\partial_r\Big)\Omega+2\dfrac{v_\theta}{r}J=0,\\
     &\partial_tJ+v\cdot\nabla J-\nu\Big(\Delta+\dfrac{2}{r}\partial_r\Big)J-\big(\omega_r\partial_r+\omega_z\partial_z\big)\dfrac{v_r}{r}=0.
     \end{aligned}\right.
  \end{align*}
For the pipe Poiseuille flow  ${u}^{(p)}=(0,0,1-r^2)$,  we have 
 \beno
 &&u^{(p)}_r=0,\quad u^{(p)}_{\theta}=0,\quad u^{(p)}_z=1-r^2=V,\\
 &&\omega^{(p)}_r=\omega^{(p)}_z=0,\quad \omega^{(p)}_{\theta}=2r,\quad \Omega^{(p)}=2,\quad J^{(p)}=0.
 \eeno
  Let $({u},\widetilde{\Omega},\widetilde{J})=\big(v-{u}^{(p)},\Omega-\Omega^{(p)},J-J^{(p)}\big)$. 
  Then there holds 
  \begin{align*}
      \left\{\begin{aligned}
     &\partial_t\widetilde{\Omega}+V\partial_z \widetilde{\Omega}-\nu\Big(\Delta+\dfrac{2}{r}\partial_r\Big)\widetilde{\Omega} =-2\dfrac{\widetilde{u}_{\theta}}{r}\widetilde{J}-\widetilde{u}\cdot\nabla\widetilde{\Omega},\\
     &\partial_t\widetilde{J}+V\partial_z\widetilde{J} -\nu\Big(\Delta+\dfrac{2}{r}\partial_r\Big)\widetilde{J}= -\widetilde{u}\cdot\nabla\widetilde{J}-(\widetilde{\omega}_r\partial_r+\widetilde{\omega}_z\partial_z)\dfrac{\widetilde{u}_r}{r}.     
     \end{aligned}\right.
  \end{align*}
  Here $\Delta=\dfrac{1}{r}\partial_r(r\partial_r)+\partial_z^2$. Thus, the linearized  resolvent system takes the form
   \begin{align*}
      \left\{\begin{aligned}
     &s{\Omega}+V\partial_z {\Omega}-\nu\Big(\Delta+\dfrac{2}{r}\partial_r\Big){\Omega} =0,\\
     &s{J}+V\partial_z{J} -\nu\Big(\Delta+\dfrac{2}{r}\partial_r\Big){J}=0,     
     \end{aligned}\right.
  \end{align*}
where ${\Omega}=\dfrac{\partial_z{u}_r-\partial_r{u}_z}{r},\ \widetilde{J}=\dfrac{-\partial_z{u}_{\theta}}{r}$. Taking Fourier transform in $z$ variable and letting $\lambda=s/(\ir l)+1$, we obtain 
  \begin{align*}
      \left\{\begin{aligned}\ir 
     &-\nu\widehat{\Delta}_{1}^{*}\widehat{\Omega} +\ir l(\lambda-r^2)\widehat{\Omega}=0,\\
     &-\nu\widehat{\Delta}_{1}^{*}\widehat{J}+\ir l(\lambda-r^2)\widehat{J}=0.
     \end{aligned}\right.
  \end{align*}
 Here $\widehat{\Delta}_{1}^{*}=\dfrac{1}{r}\partial_r(r\partial_r)+\dfrac{2}{r}\partial_r-l^2$ is the dual operator of $\widehat{\Delta}_{1}$ in $L^2((0,1);r\mathrm{d}r)$. 
 
 Let $\widehat{W}=r\widehat{u}_r$. As $\partial_r(r\widehat{u}_r)+\ir l(r\widehat{u}_z)=0$, we have $\widehat{u}_z=\ir \partial_r\widehat{W}/(rl)$, which gives 
 \begin{align*}
    \widehat{\Omega}&=\ir \big(l^2\widehat{W}+\ir lr\partial_r\widehat{u}_z\big)/(r^2l)= \ir \big(l^2\widehat{W}-r\partial_r(\partial_r\widehat{W}/r)\big)/(r^2l)\\
    &=-\dfrac{\ir }{r^2l}\Big(\partial_r^2\widehat{W}-\dfrac{\partial_r\widehat{W}}{r}-l^2\widehat{W}\Big)= -\dfrac{\ir }{r^2l}\widehat{\Delta}_{1}\widehat{W}.
 \end{align*}
Then the system can be written as
 \begin{align*}
      \left\{\begin{aligned}
     &-\nu\widehat{\Delta}_{1}^{*}\widehat{\Omega} +\ir l(\lambda-r^2)\widehat{\Omega}=0,\\
     &-\nu\widehat{\Delta}_{1}^{*}\widehat{J}+\ir l(\lambda-r^2)\widehat{J}=0,\\
     &\widehat{\Omega}=-\dfrac{\ir }{r^2l}\widehat{\Delta}_{1}\widehat{W},\quad \widehat{W}|_{r=1}=\partial_r\widehat{W}|_{r=1}=0,\ \widehat{J}|_{r=1}=0.
     \end{aligned}\right.
  \end{align*}
We denote 
  \beno
  \widehat{\Delta}_{(1)}=\dfrac{1}{r}\partial_r(r\partial_r)-\dfrac{1+r^2l^2}{r^2},
  \eeno
   which is just $\widehat{\Delta}$ in the case of $n=1$. We find that
  \begin{align*}
     \widehat{\Delta}_{(1)}(rf)&=\partial_r^2(rf)+\frac{1}{r}\partial_r(rf) -l^2rf-\dfrac{f}{r}\\&= r\partial_r^2f+2\partial_rf+\frac{f}{r}+\partial_rf -l^2rf-\dfrac{f}{r}\\
     &=r\bigg(\frac{1}{r}\partial_r(r\partial_rf)+\frac{2}{r}\partial_rf -l^2f\bigg)=r\widehat{\Delta}_1^{*}f,
  \end{align*}
 and
   \begin{align*}
     \widehat{\Delta}_{(1)}(f/r)&=\partial_r^2(f/r)+\frac{1}{r}\partial_r(f/r) -l^2f/r-f/r^3\\&= \partial_r^2f/r-2\partial_rf/r^2+2f/r^3-f/r^3+\partial_rf/r^2 -l^2f/r-f/r^3\\
     &=\bigg(\partial_r^2f-\partial_rf/r-l^2f\bigg)/r= \bigg(\frac{1}{r}\partial_r(r\partial_rf)-\frac{2}{r}\partial_rf -l^2f\bigg)/r\\
     &=\big(\widehat{\Delta}_1f\big)/r.
  \end{align*}
 We conclude that
 \beno
 \widehat{\Delta}_{(1)}(rf)=r\widehat{\Delta}_{1}^{*}f,\quad \widehat{\Delta}_{(1)}\big(f/r\big)=\big(\widehat{\Delta}_1f\big)/r.
 \eeno
  Let $\Omega_1=\ir rl\widehat{\Omega}$, $W_2=\widehat{W}/r$, $J_1=r\widehat{J}$. Then it holds that
  \begin{align}\label{eq:LNS-WU-as}
  \left\{\begin{aligned}
     &-\nu\widehat{\Delta}_{(1)}\Omega_1 +\ir l(\lambda-r^2)\Omega_1=0,\\
     &-\nu\widehat{\Delta}_{(1)}J_1+\ir l(\lambda-r^2)J_1=0,\\
     &\Omega_1=\widehat{\Delta}_{(1)}W_2,\ W_2|_{r=1}=\partial_rW_2|_{r=1}=0,\ J_1|_{r=1}=0.
  \end{aligned}\right.
  \end{align}
 Moreover, we have $\Om_1=W_2=J_1=0$ at $r=0$.

\subsection{Formulation in general case}

Let us introduce
\ben\label{def:WU}
W=(x_1,x_2,0)\cdot{u},\quad U=(x_1,x_2,0)\cdot\Delta {u}.
\een
Taking the divergence on both sides of \eqref{eq:LNS-res}, we find that
\begin{align}\label{eq:P}
&\Delta P=-\nabla V\cdot\partial_z {u}-\partial_3({u}\cdot\nabla V)=-2\partial_z({u}\cdot\nabla V)=4\partial_zW.
\end{align}
Here we used  ${u}\cdot\nabla V=-2W. $ 
It is easy to see from \eqref{eq:LNS-res}  that 
\begin{align}\label{eq:F-W}
&s W-\nu U+V\partial_z W+(x_1,x_2,0)\cdot\na P=0,\\
&s U-\nu (x_1,x_2,0)\cdot\Delta^2{u}+(x_1,x_2,0)\cdot\big(\Delta(V\partial_zu)+\na \Delta P\big)=0.
\end{align}

Thanks to $ \nabla V=-2(x_1,x_2,0), \Delta V=-4$ and \eqref{eq:P}, we find that 
\begin{align*}
&(x_1,x_2,0)\cdot\big(\Delta(V\partial_z{u})+\na \Delta p\big)\\
&=(x_1,x_2,0)\cdot\big(V\partial_z \Delta{u}+2(\nabla V\cdot\nabla)\partial_z{u}+\Delta V\partial_z{u}+4\na \partial_zW\big)\\
&=V(x_1,x_2,0)\cdot\partial_z\Delta{u}+2(\nabla V\cdot\nabla) ((x_1,x_2,0)\cdot\partial_z{u})-2(\nabla V\cdot\partial_z{u})\\&\quad+\Delta V(x_1,x_2,0)\cdot\partial_z{u}+4(x_1,x_2,0)\cdot\na \partial_zW\\&=V \partial_zU+2(\nabla V\cdot\nabla) (\partial_zW)-2(\nabla V\cdot\partial_z{u})\\&\quad-4(x_1,x_2,0)\cdot\partial_z{u}+4(x_1,x_2,0)\cdot\na \partial_zW=V \partial_zU,
\end{align*}
and 
\begin{align*}
&\Delta W=(x_1,x_2,0)\cdot\Delta{u}+2(\partial_1,\partial_2,0)\cdot{u}=U+2(\partial_1,\partial_2,0)\cdot{u},\\
&\Delta U-(x_1,x_2,0)\cdot\Delta^2{u}=2(\partial_1,\partial_2,0)\cdot\Delta {u}=
2\Delta((\partial_1,\partial_2,0)\cdot {u})=\Delta\big(\Delta W-U\big).
\end{align*}
Then we obtain
\begin{align}\label{eq:F-U}
&s U-\nu \big(2\Delta U-\Delta^2 W\big)+V\partial_3U=0.
\end{align}

Taking the Fourier transform in $(\theta,z)$, i.e.,
\begin{align*}
&\varphi (x)=\sum_{n,l\in\mathbb{Z}}\widehat{\varphi} (r,n,l)\mathrm{e}^{\ir (n\theta+lz)}\quad  \text{for}\,\, \varphi\in\{W,U,P\},
\end{align*}
we deduce from \eqref{eq:P}, \eqref{eq:F-W}  and \eqref{eq:F-U} that 
\begin{align*}
&\widehat{\Delta} \widehat{P}=4\ir l\widehat{W},\\ 
&s \widehat{W}-\nu \widehat{U}+\ir lV\widehat{W}+r\partial_r \widehat{P}=0,\\
&s \widehat{U}-\nu\big(2\widehat{\Delta} \widehat{U}-\widehat{\Delta}^2 \widehat{W}\big)+\ir lV \widehat{U}=0,
\end{align*}
here we used $(x_1,x_2,0)\cdot\na P=r\partial_rP. $ 

Our next goal is to eliminate the pressure $\widehat{P}. $ Direct calculations show that 
\begin{align*}
&\partial_r\big(r\partial_r\widehat{P}\big)-\dfrac{n^2+r^2l^2}{r}\widehat{P}=r\widehat{\Delta} \widehat{P}=4r\ir l\widehat{W}\\
&\Longrightarrow\partial_r\Big(\dfrac{r\partial_r\big(r\partial_r\widehat{P}\big)}{n^2+r^2l^2}\Big)-
\partial_r\widehat{P}
=\partial_r\Big(\frac{4r^2\ir l\widehat{W}}{n^2+r^2l^2}\Big),
\end{align*}
which along with the fact that  $\widehat{\Delta}_1f=\dfrac{n^2+r^2l^2}{r}\left(\partial_r\dfrac{r\partial_rf}{n^2+r^2l^2}-\dfrac{f}{r}\right),$ 
shows that 
\beno
 \widehat{\Delta}_1\big(r\partial_r\widehat{P}\big)=\dfrac{n^2+r^2l^2}{r}\partial_r\Big(\dfrac{4r^2\ir l\widehat{W}}{n^2+r^2l^2}\Big)=
4r\ir l\partial_r\widehat{W}+\dfrac{8\ir n^2l\widehat{W}}{n^2+r^2l^2}.
\eeno
On the other hand, we have
 \begin{align*}
 \ir l\widehat{\Delta}_1\big(V\widehat{W}\big)=&\ir l\big(V\widehat{\Delta}_1\widehat{W}
+2\partial_rV\partial_r\widehat{W}\big)
+\ir l\dfrac{n^2+r^2l^2}{r}\left(\partial_r\dfrac{r\partial_rV}{n^2+r^2l^2}\right)\widehat{W}\\
=&\ir l\big(V\widehat{\Delta}_1\widehat{W}-4r\partial_r\widehat{W}\big)
-\dfrac{4\ir ln^2}{n^2+r^2l^2}\widehat{W}.
\end{align*}
This gives
\begin{align*}
&\ir l\widehat{\Delta}_1(V\widehat{W})+\widehat{\Delta}_1\big(r\partial_r \widehat{p}\big)=\ir lV\widehat{\Delta}_1\widehat{W}+\dfrac{4\ir n^2l\widehat{W}}{n^2+r^2l^2}.
\end{align*}
Taking $ \widehat{\Delta}_1$ to the equation of $\widehat{W},$ we obtain
\begin{align*}
s \widehat{\Delta}_1\widehat{W}-\nu \widehat{\Delta}_1\widehat{U}+\ir lV\widehat{\Delta}_1\widehat{W}+\dfrac{4\ir n^2l\widehat{W}}{n^2+r^2l^2}
=0.
\end{align*}

To proceed, we use the following important observation.

\begin{Lemma} 
It holds that 
\beno
\widehat{\Delta}^2=\widehat{\Delta}_1^*\widehat{\Delta}_1,
\eeno
 where $\widehat{\Delta}_1^* $ is the dual operator of $\widehat{\Delta}_1 $ in $L^2((0,1);r\mathrm{d}r)$ given by
\beno
\widehat{\Delta}_1^*f=
\widehat{\Delta}f+\dfrac{1}{r}\partial_r\dfrac{2r^2l^2f}{n^2+r^2l^2}.
\eeno
\end{Lemma}

\begin{proof}
For any $f, g$ smooth functions with compact support, we have
\begin{align*}
   &\Big\langle\widehat{\Delta}f,\f{2rl^2}{n^2+r^2l^2}\partial_rg\Big\rangle =\Big\langle\f{1}{r}\partial_r(r\partial_rf)-\f{n^2+r^2l^2}{r^2}f,\f{2rl^2}{n^2+r^2l^2}\partial_rg\Big\rangle\\ &=-\Big\langle\partial_rf,\partial_r\big[\f{2rl^2}{n^2+r^2l^2}\partial_rg\big]\Big\rangle -\Big\langle\f{2l^2}{r}f,\partial_rg\Big\rangle\\
   &=-\Big\langle\partial_rf,\f{2l^2}{n^2+r^2l^2}\partial_r(r\partial_rg)\Big\rangle +\Big\langle\f{2rl^2}{n^2+r^2l^2}\partial_rf,\f{2rl^2}{n^2+r^2l^2}\partial_rg\Big\rangle + \Big\langle\partial_r f,\f{2l^2}{r}g\Big\rangle\\
   &=-\Big\langle\f{2rl^2}{n^2+r^2l^2}\partial_rf,\f{1}{r}\partial_r(r\partial_rg)- \f{n^2+r^2l^2}{r^2}g\Big\rangle +\Big\langle\f{2rl^2}{n^2+r^2l^2}\partial_rf,\f{2rl^2}{n^2+r^2l^2}\partial_rg\Big\rangle\\
   &=-\Big\langle\f{2rl^2}{n^2+r^2l^2}\partial_rf,\widehat{\Delta}g\Big\rangle + \Big\langle\f{2rl^2}{n^2+r^2l^2}\partial_rf,\f{2rl^2}{n^2+r^2l^2}\partial_rg\Big\rangle,
\end{align*}
which gives
\begin{align*}
   &\langle\widehat{\Delta}_1f,\widehat{\Delta}_1g\rangle- \langle\widehat{\Delta}f,\widehat{\Delta}g\rangle
   =-\Big\langle\widehat{\Delta}f,\f{2rl^2}{n^2+r^2l^2}\partial_rg\Big\rangle\\&\quad-\Big\langle\f{2rl^2}{n^2+r^2l^2}\partial_rf,\widehat{\Delta}g\Big\rangle + \Big\langle\f{2rl^2}{n^2+r^2l^2}\partial_rf,\f{2rl^2}{n^2+r^2l^2}\partial_rg\Big\rangle=0.
\end{align*}
This implies that $\widehat{\Delta}^2=\widehat{\Delta}_1^*\widehat{\Delta}_1$.
\end{proof}
\smallskip

We denote
\beno
\lambda=s/(\ir l)+1,\quad \widehat{W}_1=\widehat{\Delta}_1\widehat{W}.
\eeno
Then we have 
\beno
s+\ir lV=\ir l(\lambda-1+V)=\ir l(\lambda-r^2),\quad \widehat{\Delta}^2\widehat{W}=\widehat{\Delta}_1^*\widehat{\Delta}_1\widehat{W}=\widehat{\Delta}_1^*\widehat{W}_1.
\eeno

In summary, we derive the following coupled system of $\big(\widehat{U}, \widehat{W}\big)$:
\begin{align*}
\left\{\begin{aligned}
&-\nu \big(2\widehat{\Delta} \widehat{U}-\widehat{\Delta}_1^*\widehat{W}_1\big)+\ir l(\lambda-r^2)\widehat{U}=0,\\&-\nu \widehat{\Delta}_1\widehat{U}+\ir l(\lambda-r^2)\widehat{W}_1+\dfrac{4\ir n^2l\widehat{W}}{n^2+r^2l^2}
=0,\\
&\widehat{W}_1=\widehat{\Delta}_1\widehat{W}.
\end{aligned}\right. 
\end{align*}

Finally, let us derive the boundary conditions of $\big(\widehat{U}, \widehat{W}, \widehat{W}_1\big)$. 
For this, we write 
\beno
u=u_re_r+u_{\theta}e_{\theta}+u_ze_z.
\eeno
 Due to ${u}|_{r=1}=0, $ we have 
\beno
u_r|_{r=1}=u_{\theta}|_{r=1}=u_z|_{r=1}=0. 
\eeno 
Thanks to $r\nabla\cdot {u}=\partial_r (ru_r)+\partial_{\theta} u_{\theta}+\partial_z (ru_z)=0, $ we have $\partial_r (ru_r)|_{r=1}=0$.
Recalling that  $W=ru_r$, we show that 
\beno
W|_{r=1}=\partial_r W|_{r=1}=0. 
\eeno
Thanks to $\Delta W=U+2(\partial_1,\partial_2,0)\cdot{u}=U-2\partial_zu_z,\ \partial_zu_z|_{r=1}=0,$ we have $\Delta W|_{r=1}=U|_{r=1}. $ Thus,
\beno
{\widehat{W}}|_{r=1}=\partial_r \widehat{W}|_{r=1}=0,\quad \Delta \widehat{W}|_{r=1}=\widehat{U}|_{r=1}. 
\eeno
Since $\widehat{W}_1=\widehat{\Delta}_1\widehat{W}=\widehat{\Delta}\widehat{W}-\dfrac{2rl^2}{n^2+r^2l^2}\partial_r\widehat{W}, $ we have $\widehat{W}_1|_{r=1}=\widehat{\Delta}\widehat{W}|_{r=1} .$ Thus, 
\beno
{\widehat{W}}|_{r=1}=\partial_r \widehat{W}|_{r=1}=0,\quad \widehat{W}_1|_{r=1}=\widehat{U}|_{r=1} .
\eeno

For the convenience of notations, dropping the hat of $\widehat{W}, \widehat{U}, \cdots$, we derive
 \begin{align}\label{eq:LNS-WU}
\left\{\begin{aligned}
&-\nu\big(2\widehat{\Delta}{U}-\widehat{\Delta}_1^*{W}_1\big)+\ir l(\lambda-r^2){U}=0,\\
&-\nu \widehat{\Delta}_1{U}+\ir l(\lambda-r^2){W}_1+\dfrac{4\ir n^2l{W}}{n^2+r^2l^2}
=0,\\
&W_1=\widehat{\Delta}_1{W},\quad {{W}}|_{r=1}=\partial_r{W}|_{r=1}=0,\quad {W}_1|_{r=1}={U}|_{r=1}.
\end{aligned}\right. 
\end{align}
Moreover, we have $U=W=0$ at $r=0$.\smallskip

When $n=0$, we have $\widehat{W}_1=\widehat{U}. $ Indeed, thanks to 
$$
\Delta W=U+2(\partial_1,\partial_2,0)\cdot{u}=U+2r^{-1}\big(\partial_r (ru_r)+\partial_{\theta} u_{\theta}\big),\quad W=ru_r,
$$ 
we deduce that 
\beno
{\widehat{\Delta}} \widehat{W}=\widehat{U}+2r^{-1}\big(\partial_r\widehat{ W}+\ir n \widehat{u}_{\theta}\big).
\eeno
When $n=0$, we have ${\widehat{\Delta}} \widehat{W}=\widehat{U}+2r^{-1}\partial_r\widehat{ W},$ thus, $\widehat{U}={\widehat{\Delta}} \widehat{W}-2r^{-1}\partial_r\widehat{ W}={\widehat{\Delta}} \widehat{W}-\dfrac{2rl^2}{n^2+r^2l^2}\partial_r\widehat{ W}={\widehat{\Delta}_1} \widehat{W}.$ 
Then it holds that
\beno
-\nu \widehat{\Delta}_1\widehat{W}_1+\ir l(\lambda-r^2)\widehat{W}_1=0.
\eeno
 Let $\Omega_1=\widehat{W}_1/r$, $W_2=\widehat{W}/r$. Thanks to $\widehat{\Delta}_{(1)}\big(\cdot/r\big)=\widehat{\Delta}_{1}/r$, 
we find that
  \begin{align}\label{eq:LNS-WU-zero}
  \left\{\begin{aligned}
     &-\nu\widehat{\Delta}_{(1)}\Omega_1 +\ir l(\lambda-r^2)\Omega_1=0,\\    
     &\Omega_1=\widehat{\Delta}_{(1)}W_2,\quad W_2|_{r=1}=\partial_rW_2|_{r=1}=0,
  \end{aligned}\right.
  \end{align}
 which is similar to \eqref{eq:LNS-WU-as}.
 
 In the sequel, we always specify the following Dirichlet boundary condition for $U, W, \Om_1, W_2$ at $r=0$:
 \beno
 U(0)=W(0)=\Om_1(0)=W_2(0)=0.
 \eeno

\section{Key ideas and  sketch of  the proof}

Main goal of this paper is to prove the spectral bound of the linearized operator $\cL_\nu$:  $m(\nu)\le -c\nu$.
The problem is reduced to show that the linearized Navier-Stokes equation  $su=\cL_\nu u$  has only trivial solution when $\mathbf{Re}(s)\ge -c\nu$ for some $c>0$.

When $u\in \mathcal{H}_0$,  we have
\begin{align*}
\left\{
\begin{aligned}
&s{u}-\nu\Delta {u}+(0,0,{u}\cdot\nabla V)+\na P=0,\\
&\nabla\cdot {u}=0,\quad {u}|_{r=1}=0.
\end{aligned}
\right.
\end{align*}
By taking the inner product with $u$ to the above equation, it is easy to show that $m(\nu)\le -c\nu$.
The main challenge is to show that $m(\nu)\le -c\nu$ when $u\in \mathcal{H}$. In this case, we need to study the following linearized Navier-Stokes equations 
\begin{align*}
\left\{
\begin{aligned}
&s{u}-\nu\Delta {u}+V\partial_z {u}+(0,0,{u}\cdot\nabla V)+\na P=0,\\
&\nabla\cdot {u}=0,\quad {u}|_{r=1}=0.
\end{aligned}
\right.
\end{align*} 
It is difficult to solve this system directly.  Our strategy is as follows.

\subsection{Introduction of key formulation}

For the axisymmetric flow(i.e, axisymmetric perturbation),  we introduce
\beno
&&\Om=\f {\om_\theta} r,\quad W=ru_r,\quad J=\f {\om_r} r,\\
&&\Omega_1=\ir rl\widehat{\Omega},\quad W_2=\widehat{W}/r,\quad J_1=r\widehat{J}. 
\eeno
Then we find that 
\begin{align}\label{eq:LNS-WU-as-3}
  \left\{\begin{aligned}
     &-\nu\widehat{\Delta}_{(1)}\Omega_1 +\ir l(\lambda-r^2)\Omega_1=0,\\
     &-\nu\widehat{\Delta}_{(1)}J_1+\ir l(\lambda-r^2)J_1=0,\\
     &\Omega_1=\widehat{\Delta}_{(1)}W_2,\ W_2|_{r=1}=\partial_rW_2|_{r=1}=0,\ J_1|_{r=1}=0.
  \end{aligned}\right.
  \end{align}
For this new system, the equation of $J_1$ is relatively easy to solve due to the vanishing  Dirichlet boundary condition.
The main trouble is to solve the equation of $\Om_1$. In fact, this equation is  similar to the linearized equation around the Couette flow,  if we view 
$\Om_1$ and $W_2$ as the vorticity and stream function respectively. Thus, it could be solved by using the resolvent estimate method 
developed in our work \cite{CLWZ}.

Motivated by the formulation in the axisymmetric case,  for general 3D perturbations,  we introduce
\beno
W=(x_1,x_2,0)\cdot{u},\quad U=(x_1,x_2,0)\cdot\Delta {u}.
\eeno
The most critical innovation of this paper is to find the following key formulation:
\begin{align}\label{eq:LNS-WU-3}
\left\{\begin{aligned}
&-\nu \big(2\widehat{\Delta} \widehat{U}-\widehat{\Delta}_1^*\widehat{W}_1\big)+\ir l(\lambda-r^2)\widehat{U}=0,\\&-\nu \widehat{\Delta}_1\widehat{U}+\ir l(\lambda-r^2)\widehat{W}_1+\dfrac{4\ir n^2l\widehat{W}}{n^2+r^2l^2}
=0,\\
&\widehat{W}_1=\widehat{\Delta}_1\widehat{W}, \quad {\widehat{W}}|_{r=1}=\partial_r\widehat{W}|_{r=1}=0,\quad \widehat{W}_1|_{r=1}=\widehat{U}|_{r=1}. 
\end{aligned}\right. 
\end{align}
Moreover, when $n=0$, we have $\widehat{U}=\widehat{W}_1$, and thus,
\beno
-\nu \widehat{\Delta}_1\widehat{W}_1+\ir l(\lambda-r^2)\widehat{W}_1=0.
\eeno
Let $(\Omega_1,W_2)=\big(\widehat{W}_1/r, \widehat{W}/r\big)$, which satisfies  
 \begin{align}\label{eq:LNS-WU-zero-3}
  \left\{\begin{aligned}
     &-\nu\widehat{\Delta}_{(1)}\Omega_1 +\ir l(\lambda-r^2)\Omega_1=0,\\    
     &\Omega_1=\widehat{\Delta}_{(1)}W_2,\quad W_2|_{r=1}=\partial_rW_2|_{r=1}=0.
  \end{aligned}\right.
  \end{align}
  
In the sequel, we will drop the hat of $\widehat{W}, \widehat{U}$ for convenience.  

\subsection{Solving the linearized system with artificial boundary condition}

To solve \eqref{eq:LNS-WU-3},  the most key step is to solve the inhomogeneous linearized system 
with artificial boundary condition:
\begin{align}\label{eq:LNS-WU-ar-3}
\left\{\begin{array}{l}-\nu\big(2\widehat{\Delta} {U}-\widehat{\Delta}_1^*{W}_1\big)+\ir l(\lambda-r^2){U}={F}_1,\\-\nu \widehat{\Delta}_1{U}+\ir l(\lambda-r^2){W}_1+\dfrac{4\ir n^2l{W}}{n^2+r^2l^2}
={F}_2,\\
{W}_1=\widehat{\Delta}_1{W},\quad {W}|_{r=1}={W}_1|_{r=1}={U}|_{r=1}=0.
\end{array}\right.
\end{align} 
For this system, we will establish the following resolvent  estimates from $L^2$ to $L^2$:  if $l\lambda_i\leq c_1|\nu n l|^{\f12}$, then it holds that 
\begin{align*}
&\big(|\nu nl|^{\f12}+|\nu\lambda l^2|^{\f13}+|l\lambda_i|\big)\|(W_1,U)\|_{L^2}\leq C\|(F_1,F_2)\|_{L^2},
\\&\|\partial_rW\|_{L^2}\leq C(|n|+|l|)^{\f52}n^{-2}|l|^{-\f12}\big(|\nu nl|^{\f12}+|\nu\lambda_r l^2|^{\f13}\big)^{-\f12}\|(F_1,F_2)\|_{L^2},\\
&|\partial_rW(1)|^2\leq  C(|n|+|l|)^{\f52}n^{-2}|l|^{-\f12}(|\nu nl|^{\f12}+|\nu\lambda l^2|^{\f13})^{-\f32}\|(F_1,F_2)\|_{L^2}^2.
\end{align*}
which are the core estimates of this paper.  The proof of these estimates is  split into four cases.\smallskip

$\bullet$ {\bf Case of $\la<0$}. This is the easiest case, since the following energy identity
\begin{align}
&\mathbf{Im}\big(\langle U,F_1\rangle-\langle F_2,W_1\rangle\big)=2\nu\mathbf{Im}\langle U,\widehat{\Delta}_1^* {W}_1\rangle+l\big(\| rU\|_{L^2}^2-\lambda\|U\|_{L^2}^2\big)\nonumber\\&\qquad+l\big(\| rW_1\|_{L^2}^2-\lambda\| W_1\|_{L^2}^2\big)+4n^2l\int_0^1\left(\dfrac{r|\partial_rW|^2}{n^2+r^2l^2}+\dfrac{|W|^2}{r}\right)dr\label{eq:EI}
\end{align}
provides a good control of $(W_1,U)$ when $\la<0$.\smallskip

$\bullet$ {\bf Case of $\la>1$}.  The key ingredient of this case is the following inequality:
  \begin{align*}
     \left\|\sqrt{\lambda-r^2}W_1\right\|_{L^2}^2\geq 
     4n^2(|n|+1)\int_{0}^{1}\left(\dfrac{r|\partial_rW|^2}{n^2+r^2l^2}+\dfrac{|W|^2}{r}\right)\mathrm{d}r,
  \end{align*}
which is a consequence of Lemma \ref{lem:E-key}.  
Then \eqref{eq:EI} also provides a good control of $(W_1,U)$ similar to the case of $\la<0$.\smallskip

$\bullet$ {\bf Case of $\la\in (0,1]$}. This is the most difficult case.  Let $r_0=\la^\f12$ and  $\delta=(\nu/|lr_0|)^{\f13}$. The case of $\delta/r_0\ge \min(c_2,1/|n|)$ or $|\delta l|\ge c_2$ is relatively simple. Therefore, we focus on the case of  $\delta/r_0\le \min(c_2,1/|n|)$ and  $|\delta l|\le c_2$. The proof consists of many steps.

{\bf Step 1.} Based on the formulation
\begin{align*}
&-\nu \widehat{\Delta} {U}+\ir l(\lambda-r^2){U}=F_1+\nu \widehat{\Delta} {U}_1+\nu (\widehat{\Delta}-\widehat{\Delta}_1^*) {W}_1:=F_{1,1},
\end{align*}
where $U_1=U-W_1$,  we can establish the following resolvent estimates
\begin{align*}
    &\nu\|\widehat{\Delta} {U}\|_{L^{2}}+|l|\|(\lambda-r^2) {U}\|_{L^{2}}+|\nu/\delta|\|U\|_{1}+|lr_0\delta|\|U\|_{L^2}+|l\delta|\|rU\|_{L^2}\leq C\|F_{1,1}\|_{L^2}.
  \end{align*}
  
 {\bf Step 2.} Based on Step 1 and  the formulation
\begin{align*}
&-\nu \widehat{\Delta} {U}_1+\ir l(\lambda-r^2){U}_1={F}_{1,2}+\dfrac{4\ir n^2l{W}}{n^2+r^2l^2},
\end{align*}
where $F_{1,2}=F_1-F_2+\nu (\widehat{\Delta}-\widehat{\Delta}_1) {U}+\nu (\widehat{\Delta}-\widehat{\Delta}_1^*) {W}_1$, we can show that 
\begin{align*}
  &\nu^2(\|\widehat{\Delta} {U}\|_{L^2}^2+\|\widehat{\Delta} {U}_1\|_{L^2}^2)+\|F_{1,1}\|_{L^2}^2\\ &\leq  C\Big(\nu l\|rU_1\|_{L^2}\| {U}_1\|_{1}+\| (F_1,F_2)\|_{L^2}^2+ c_2^2|\nu/\delta|^2\|(U_1,W_1)\|_{1}^2+\nu l\|(W_1,U_1)\|_{L^2}^2\Big).
\end{align*}

{\bf Step 3.} Based on the formulation 
\begin{align}\label{eq:W-euler}
(\lambda-r^2){W}_1+\dfrac{4n^2{W}}{n^2+r^2l^2}=({F}_2+\nu \widehat{\Delta}_1{U})/(\ir l):=F_{2,2},
\end{align}
we can establish  various estimates of $W$ and $W_1$ such as
 \begin{align*}
     r_{0}\|W_1\|_{L^2}+\|rW_1\|_{L^2}&\leq C\Big(\delta^{-1}\|F_{2,2}\|_{L^2}+ r_0\delta\|W_1\|_{1}
     +n^2\delta^{-\f12}r_0^{\f12}(n+r_0l)^{-\f32}E^{\f12}\Big).
  \end{align*} 
  
{\bf Step 4.} Key coercive estimate:
\begin{align*}
  &-\left\langle \dfrac{{W}}{n^2+r^2l^2},{W}_1\right\rangle-\left\langle \dfrac{\chi_0{W}}{n^2+r^2l^2}, \dfrac{4n^2{W}}{n^2+r^2l^2}\right\rangle
\geq \frac{E}{3}-
\dfrac{C\|{W}\|_{L^{2}(I)}^2}{r_0\delta(|n|+r_0|l|)},
\end{align*}
where $\chi_0(r)=\rho(r)/(\lambda-r^2)$ with $\rho(r)=1$ for $|r-r_0|> \delta$ and  $\rho(r)=0$ for $|r-r_0|< \delta$.\smallskip

{\bf Step 5.}  Based on the estimates in Step 3 and Step 4 and formulation \eqref{eq:W-euler}, we can show that 
\begin{align*}
&r_0\delta\|W_1\|_{L^2}+\delta\|rW_1\|_{L^2}\leq  C\big({r_0\delta^2\|W_1\|_{1}+\|F_{2,2}\|_{L^{2}}}\big),\\
&E\leq
Cn^{-4}(r_0\delta)^{-1}(|n|+r_0|l|)^3\big(r_0^2\delta^4\|W_1\|_{1}^2+\|F_{2,2}\|_{L^{2}}^2\big).
\end{align*}

{\bf Step 6.} We denote
\beno
G=\|F_{1,1}\|_{L^2}+lr_0\delta^2\|W_1\|_{1}+\|F_2\|_{L^2}.
\eeno
Summing up the estimates established in Step 1-Step 5,  we can show that 
 \begin{align*}
& E^{\f12}\leq Cn^{-2}|l|^{-1}(r_0\delta)^{-\f12}(|n|+r_0|l|)^{\f32}G,\\
& |l|r_0\delta\|(W_1,U_1)\|_{L^2}+|l|\delta\|rU_1\|_{L^2}\leq CG,\\
&G\leq C\|(F_1,F_2)\|_{L^2}.
\end{align*}

$\bullet$ {\bf Case of $\la\in \C$}. Under the assumption $l\lambda_i\leq c_1|\nu n l|^{\f12}$, we have
\begin{align*}
   &|l\lambda_i|\|(W_1,U)\|_{L^2}\leq \|(F_1,F_2)\|_{L^2}+c_1|\nu nl|^{\f12}\|(W_1,U)\|_{L^2}.
\end{align*}
Thus, the term $l\la_iU$ and $l\la_iW$ could be viewed as a perturbation term when $c_1$ is enough small, since we have proved that 
for $\la\in \R$, 
\begin{align*}
&\big(|\nu nl|^{\f12}+|\nu\lambda l^2|^{\f13}\big)\|(W_1,U)\|_{L^2}\leq C\|(F_1,F_2)\|_{L^2}.
\end{align*}

\subsection{Solving the linearized system with nonvanishing boundary condition} 

When ${W_1}|_{r=1}={U}|_{r=1}\neq 0$,  we need to solve the following homogeneous linearized system with nonvanishing boundary condition:
\begin{align}\label{eq:LNS-hom-3}
\left\{\begin{array}{l}
-\nu (2\widehat{\Delta} {U}-\widehat{\Delta}_1^*W)+\ir l(\lambda-r^2){U}=0,\\
-\nu \widehat{\Delta}_1{U}+\ir l(\lambda-r^2){W}+\dfrac{4\ir n^2l{W}}{n^2+r^2l^2}
=0,\\ 
W_1=\widehat{\Delta}_1{W},\quad {W}|_{r=1}=0,\quad {W_1}|_{r=1}={U}|_{r=1}=1.
\end{array}\right.
\end{align}
The key point is to show that $\pa_rW(1)\neq 0$. 
To this end,  we first introduce an approximate elliptic system
\begin{align}\label{eq:WU-app-3}
\left\{
\begin{aligned}
&-\nu \widehat{\Delta}_1{U}+\ir l(\lambda-r^2)U=0,\quad U(1)=1,\\
&\widehat{\Delta}_1{W}=U,\quad W(1)=0.
\end{aligned}\right.
\end{align}
We establish the following key estimates: if $l\lambda_i\leq c_5|\nu n l|^{\f12}$ and $0<\nu<c_5\min(|l|,1)$, then we have
\beno
  |\partial_r{W}(1)|& \geq C^{-1}A^{-1},\quad \|U\|^2_{L^2(0,s)}\le Cs|U(s)|^2/A(s),\quad s\in (0,1].
\eeno
Let $(U_a,W_a)$ solve \eqref{eq:WU-app-3} and 
\beno
U_e=U-U_a,\quad W_{1,e}=W_1-U_a,\quad W_e=W-W_a.
\eeno
Then $(U_e, W_{e})$ satisfies 
\begin{align*}
\left\{\begin{array}{l}-\nu (2\widehat{\Delta} {U_e}-\widehat{\Delta}_1^*{W_{1,e}})+\ir l(\lambda-r^2){U_e}=\nu(2\widehat{\Delta}-\widehat{\Delta}_1^*-\widehat{\Delta}_1)U_a,\\-\nu \widehat{\Delta}_1{U_e}+\ir l(\lambda-r^2){W_{1,e}}+\dfrac{4\ir n^2l{W_e}}{n^2+r^2l^2}
=-\dfrac{4\ir n^2l{W_a}}{n^2+r^2l^2},\\  W_{1,e}=\widehat{\Delta}_1{W_e},\ {W_e}|_{r=1}=0,\ {W_{1,e}}|_{r=1}={U_e}|_{r=1}=0.
\end{array}\right.
\end{align*}
Using the resolvent estimates for the system \eqref{eq:LNS-WU-ar-3}, we can show that
\begin{align*}
&|\partial_rW_e(1)|\leq  C\big(\nu^{\f38}+|\nu/l|^{\f{1}{12}}\big)A^{-1},
  \end{align*}
and then
\begin{align*}
|\partial_r{W}(1)|\geq |\partial_r{W_a}(1)|-|\partial_rW_e(1)| \geq \big(C^{-1}-C(\nu^{\f38}+|\nu/l|^{\f{1}{12}})\big)A^{-1},
\end{align*}
which shows $|\pa_rW(1)|\ge cA^{-1}$  by taking $\nu$ small enough. \smallskip

Finally, it seems  highly nontrivial to solve the approximate elliptic system \eqref{eq:WU-app-3}, although it looks very simple. 
We solve this system by introducing the approximate elliptic equation with constant coefficient 
\begin{align*}
-\nu \widehat{\Delta}_1{U}+\ir l(\la-1)U=0,\quad U(0)=0,\quad U(1)=1,
\end{align*}
and the Airy approximation in the case when $l\lambda_i\ge \max\big(|l(\lambda-1)|/2,\nu (n^2+l^2)/3\big)$.

\section{The inhomogeneous linearized system}

In this section, we consider the inhomogeneous linearized system with artificial boundary conditions:
\begin{align}\label{eq:LNS-WU-a}
\left\{\begin{array}{l}-\nu\big(2\widehat{\Delta} {U}-\widehat{\Delta}_1^*{W}_1\big)+\ir l(\lambda-r^2){U}={F}_1,\\-\nu \widehat{\Delta}_1{U}+\ir l(\lambda-r^2){W}_1+\dfrac{4\ir n^2l{W}}{n^2+r^2l^2}
={F}_2,\\
{W}_1=\widehat{\Delta}_1{W},\quad {W}|_{r=1}={W}_1|_{r=1}={U}|_{r=1}=0.
\end{array}\right.
\end{align} 

In the sequel, we always assume that $l\neq 0$ and $n\neq 0$. \smallskip

We will first establish the resolvent estimates of the system \eqref{eq:LNS-WU-a} for  $\la\in \R$.

\begin{Proposition}\label{prop:res-real}
Let $ \lambda\in\R$ and $(W,U)$ be a solution of \eqref{eq:LNS-WU-a}. Then there holds that 
\begin{align*}
&\big(|\nu nl|^{\f12}+|\nu\lambda l^2|^{\f13}\big)\|(W_1,U)\|_{L^2}\leq C\|(F_1,F_2)\|_{L^2},\\
&\|\partial_rW\|_{L^2}\leq C(|n|+|l|)^{\f52}n^{-2}|l|^{-\f12}\big(|\nu nl|^{\f12}+|\nu\lambda l^2|^{\f13}\big)^{-\f12}\|(F_1,F_2)\|_{L^2}.
\end{align*}
\end{Proposition}

For $\la\in \C$, if $l\lambda_i\leq c_1|\nu n l|^{\f12}$, 
we can establish similar resolvent estimates for the system \eqref{eq:LNS-WU-a}.

\begin{Proposition}\label{prop:res-com}
Let $\la\in \C$ and $(W,U)$ be a solution of \eqref{eq:LNS-WU-a}. There exists a constant $c_1$ such that if $l\lambda_i\leq c_1|\nu n l|^{\f12}$, then we have
\begin{align*}
&\big(|\nu nl|^{\f12}+|\nu\lambda l^2|^{\f13}+|l\lambda_i|\big)\|(W_1,U)\|_{L^2}\leq C\|(F_1,F_2)\|_{L^2},
\\&\|\partial_rW\|_{L^2}\leq C(|n|+|l|)^{\f52}n^{-2}|l|^{-\f12}\big(|\nu nl|^{\f12}+|\nu\lambda_r l^2|^{\f13}\big)^{-\f12}\|(F_1,F_2)\|_{L^2},\\
&|\partial_rW(1)|^2\leq  C(|n|+|l|)^{\f52}n^{-2}|l|^{-\f12}(|\nu nl|^{\f12}+|\nu\lambda l^2|^{\f13})^{-\f32}\|(F_1,F_2)\|_{L^2}^2.
\end{align*}
\end{Proposition}

The following energy identities will be constantly used. 

\begin{Lemma}\label{lem:energy} 
Let $\la\in \R$. It holds that
\begin{align*}
&\mathbf{Re}\big(\langle U,F_1\rangle+\langle F_2,W_1\rangle\big)=2\nu\big(\|\partial_r U\|_{L^2}^2+n^2\|U/r\|_{L^2}^2+l^2\|U\|_{L^2}^2\big),\nonumber\\
&\mathbf{Im}\big(\langle U,F_1\rangle-\langle F_2,W_1\rangle\big)=2\nu\mathbf{Im}\langle U,\widehat{\Delta}_1^* {W}_1\rangle+l\big(\| rU\|_{L^2}^2-\lambda\|U\|_{L^2}^2\big)\nonumber\\&\qquad+l\big(\| rW_1\|_{L^2}^2-\lambda\| W_1\|_{L^2}^2\big)+4n^2l\int_0^1\left(\dfrac{r|\partial_rW|^2}{n^2+r^2l^2}+\dfrac{|W|^2}{r}\right)dr. 
\end{align*}
In particular, we have
\begin{align}\label{eq:U-H1}
  2\nu\|U\|_{1}^2\leq \|(U,W_{1})\|_{L^2}\|(F_1,F_2)\|_{L^2}.
\end{align}
\end{Lemma}

\begin{proof}
Two energy identities follow from \eqref{eq:LNS-WU-a} and  the following facts that 
\begin{align*}
&-\langle U,\widehat{\Delta} {U}\rangle=\|\partial_r U\|_{L^2}^2+n^2\|U/r\|_{L^2}^2+l^2\|U\|_{L^2}^2,\\
& \langle U,\widehat{\Delta}_1^* {W}_1\rangle=\langle \widehat{\Delta}_1U, {W}_1\rangle,\\
&\mathbf{Re}\langle U,\ir l(\lambda-r^2){U}\rangle=\mathbf{Re}\langle \ir l(\lambda-r^2){W}_1,W_1\rangle=0,
\end{align*}
and
\begin{align*}\left\langle \dfrac{{W}}{n^2+r^2l^2}, {W}_1\right\rangle=&\left\langle \dfrac{{W}}{n^2+r^2l^2}, \dfrac{n^2+r^2l^2}{r}\left(\partial_r\dfrac{r\partial_rW}{n^2+r^2l^2}-\dfrac{W}{r}\right)\right\rangle\\
=&-\int_0^1\left(\dfrac{r|\partial_rW|^2}{n^2+r^2l^2}+\dfrac{|W|^2}{r}\right)dr.
\end{align*}
The inequality \eqref{eq:U-H1} follows from the first identity. 
\end{proof}\smallskip

The following fact will be also used constantly.  

\begin{Lemma}\label{lem:W1-H1}
It holds that 
\begin{align*}
   \|W_1\|_{1}^2\leq -2\mathbf{Re}\langle W_1,\widehat{\Delta}^{*}_1W_{1}\rangle.
\end{align*}
\end{Lemma}
\begin{proof}
Thanks to 
\begin{align*}
   &\langle W_1,\widehat{\Delta}^{*}_1W_{1}\rangle= \langle \widehat{\Delta}_1 W_1,W_{1}\rangle=-\|W_1\|_{1}^2-\left\langle \partial_rW_1,\frac{2rl^2 W_1}{n^2+r^2l^2}\right\rangle,
\end{align*}
we deduce that
\begin{align*}
   \|W_1\|_{1}^2&\leq -\mathbf{Re}\langle W_1,\widehat{\Delta}^{*}_1W_{1}\rangle+\|\partial_rW_1\|_{L^2}
   \left\|\frac{2l^2r W_1}{n^2+r^2l^2}\right\|_{L^2}\\
   &\leq -\mathbf{Re}\langle W_1,\widehat{\Delta}^{*}_1W_{1}\rangle+1/2\|\partial_rW_1\|_{L^2}^2
   +2\left\|\frac{l^2r W_1}{2nrl}\right\|_{L^2}^2\\ &\leq -\mathbf{Re}\langle W_1,\widehat{\Delta}^{*}_1W_{1}\rangle+1/2\big(\|\partial_rW_1\|_{L^2}^2+l^2\| W_1\|_{L^2}^2\big),
\end{align*}
here we used $ n^2+r^2l^2\geq 2|nrl|$. This shows that 
\begin{align*}
   \|W_1\|_{1}^2\leq -2\mathbf{Re}\langle W_1,\widehat{\Delta}^{*}_1W_{1}\rangle.
\end{align*}
\end{proof}

Let us define
\beno
&&\|f\|_1^2=\|\partial_r f\|_{L^2}^2+n^2\|f/r\|_{L^2}^2+l^2\|f\|_{L^2}^2,\\
&& E=\int_{0}^{1}\left(\dfrac{r|\partial_rW|^2}{n^2+r^2l^2}+\dfrac{|W|^2}{r}\right)\mathrm{d}r=-\left\langle\f{W}{n^2+r^2l^2},W_1 \right\rangle. 
\eeno

\subsection{Resolvent estimates when $\la\le 0$}
It follows from Lemma \ref{lem:energy} that 
\begin{align*}
   &|l|\big(\| (rW_1,rU)\|_{L^2}^2+|\lambda|\| (W_1,U)\|_{L^2}^2\big) +4n^2|l|\big(\|\partial_rW/\sqrt{n^2+r^2l^2}\|_{L^2}^2+\| W/r\|_{L^2}^2\big)\\
   &\leq \big|\langle U,F_1\rangle-\langle F_2,W_1\rangle\big|+2\nu\big|\langle U,\widehat{\Delta}_1^* {W}_1\rangle\big|,
\end{align*}
and by \eqref{eq:U-H1}, we have 
\begin{align}
   \big|\langle U,\widehat{\Delta}_1^* {W}_1\rangle\big|&= \big|\langle \widehat{\Delta}_1U, {W}_1\rangle\big|\leq \big|\langle \widehat{\Delta}U, {W}_1\rangle\big|+\left|\bigg\langle \f{2rl^2}{n^2+r^2l^2}\partial_rU, W_1\bigg\rangle\right|\nonumber\\
&\leq \|U\|_{1}\|W_1\|_{1}+\|\partial_rU\|_{L^2}\left\|\frac{2W_1}{r}\right\|_{L^2}\leq C\|U\|_{1}\|W_1\|_{1}\label{eq:la<0-est5}\\
&\leq C\nu^{-\f12}\|(W_1,U)\|_{L^2}^{\f12}\|(F_1,F_2)\|_{L^2}^{\f12}\|W_1\|_{1}.\nonumber
\end{align}
Then we can deduce that
\begin{align}\label{eq:la<0-est1}
   &\|(rW_1,rU)\|_{L^2}^2+|\lambda|\| (W_1,U)\|_{L^2}^2+n^2\| W/r\|_{L^2}^2+n^2\|\partial_rW/\sqrt{n^2+r^2l^2}\|_{L^2}^2\nonumber \\
   &\leq C |l|^{-1}\Big(\|(W_1,U)\|_{L^2}\|(F_2,F_3)\|_{L^2} +\nu^{\f12}\|(W_1,U)\|_{L^2}^{\f12}\|(F_1,F_2)\|_{L^2}^{\f12}\|W_1\|_{1}\Big).
\end{align}

Using the equation \eqref{eq:LNS-WU-a} to obtain 
\begin{align*}
  &\big\langle W_1,-\nu(2\widehat{\Delta}U-\widehat{\Delta}_1^{*} W_1)+\ir l(\lambda-r^2)U\big\rangle= \langle W_1,F_1\rangle,\\
  &\big\langle -\nu\widehat{\Delta}_1 U+\ir l(\lambda-r^2)W_1+\f{4\ir  n^2lW}{n^2+r^2l^2},U\big\rangle= \langle F_2,U\rangle,
\end{align*}
which give
\begin{align}\label{eq:la<0-est2}
   &-\nu\langle W_1,\widehat{\Delta}^{*}_1W_1\rangle+\nu\langle \widehat{\Delta}_1U,U\rangle+2\nu\langle W_1,\widehat{\Delta}U\rangle-4\ir n^2l\left\langle \f{W}{n^2+r^2l^2},U\right\rangle \\
   &=-\langle W_1,F_1\rangle-\langle F_2,U\rangle.\nonumber
\end{align}
We get by integration by parts that 
\begin{align*}
   &-\langle W_1,\widehat{\Delta}U\rangle =\langle \partial_rW_1,\partial_rU\rangle+ \left\langle W_1,\f{n^2+r^2l^2}{r^2}U \right\rangle\leq \|W_1\|_{1}\|U\|_{1},
\end{align*}
and
\begin{align*}
  \big|\langle \widehat{\Delta}_1U,U\rangle\big| &\leq \big|\langle \widehat{\Delta}U,U\rangle\big|+\bigg|\left\langle \f{2rl^2\p_rU}{n^2+r^2l^2},U\right\rangle\bigg|\leq \|U\|_{1}^2+2\|U/r\|_{L^2}\|\pa_rU\|_{L^2}\leq C\|U\|_{1}^2.
\end{align*}
For the nonlocal term, we have 
\begin{align*}
   \left|\bigg\langle\f{4n^2lW}{n^2+r^2l^2},U\bigg\rangle\right|&\leq 4l\|W/r\|_{L^2}\|rU\|_{L^2}.
\end{align*}
Then we infer from \eqref{eq:la<0-est2} that
\begin{align}\label{eq:la<0-est3}
 &-\nu\mathbf{Re}\langle W_1,\widehat{\Delta}^{*}_1W_1\rangle\nonumber\\
  &\leq \nu\big|\langle \widehat{\Delta}_1U,U\rangle\big|+2\nu\big|\langle W_1,\widehat{\Delta}U\rangle \big|+\big|\langle W_1,F_1\rangle+\langle F_2,U\rangle\big| +\left|\bigg\langle\f{4n^2lW}{n^2+r^2l^2},U\bigg\rangle\right|
 \nonumber\\&\leq C\Big(\nu\big(\|U\|_{1}^2+\|U\|_{1}\|W_1\|_{1}\big)+\|(W_1,U)\|_{L^2}\|(F_2,F_3)\|_{L^2}\nonumber\\
 &\qquad\quad+|l|\|W/r\|_{L^2}\|rU\|_{L^2}\Big),
\end{align}
from which, \eqref{eq:la<0-est1} and Lemma \ref{lem:W1-H1}, we infer that 
\begin{align*}
   \nu\|W_1\|_{1}^2\leq& C\big(\nu\|U\|_{1}^2+ \|(W_1,U)\|_{L^2}\|(F_1,F_2)\|_{L^2}+|l|\|(W/r,rU)\|_{L^2}^2\big)\\
   \leq &C\Big(\|(W_1,U)\|_{L^2}\|(F_1,F_2)\|_{L^2}+\nu^{\f12}\|(W_1,U)\|_{L^2}^{\f12}\|(F_1,F_2)\|_{L^2}^{\f12} \|W_1\|_{1} +\nu\|U\|_{1}^2\Big).
\end{align*}
Then Young's inequality gives 
\begin{align*}
   &\|W_1\|_{1}\leq C\big(\nu^{-\f12}\|(W_1,U)\|_{L^2}^{\f12}\|(F_1,F_2)\|_{L^2}^{\f12} +\|U\|_{1}\big).
\end{align*} 
This along with \eqref{eq:U-H1} gives
\begin{align*}
   \|(W_1,U)\|_{1}&\leq C\big(\nu^{-\f12} \|(W_1,U)\|_{L^2}^{\f12}\|(F_1,F_2)\|^{\f12}_{L^2}+\|U\|_{1}\big)\nonumber\\
   &\leq C\nu^{-\f12}\|(W_1,U)\|_{L^2}^{\f12}\|(F_1,F_2)\|_{L^2}^{\f12}.
\end{align*}
This along with \eqref{eq:la<0-est1} shows that 
\begin{align}
   &\|(rW_1,rU)\|_{L^2}^2+|\lambda|\| (W_1,U)\|_{L^2}^2+n^2\left\|\partial_rW/\sqrt{n^2+r^2l^2}\right\|_{L^2}^2\nonumber
   \\ &\leq C |l|^{-1}\|(W_1,U)\|_{L^2}\|(F_1,F_2)\|_{L^2},\label{eq:la<0-est4}
\end{align}
and hence,
\begin{align*}
   &|\lambda l|\| (W_1,U)\|_{L^2} \leq C \|(F_1,F_2)\|_{L^2}.
\end{align*}

On the other hand,  we get by the interpolation that
\begin{align}
 \|(W_1,U)\|_{L^2}^2\leq& C\|(W_1/r,U/r)\|_{L^2}\|(rW_1,rU)\|_{L^2}\leq C|n|^{-1}\|(W_{1},U)\|_{1}\|(rW_1,rU)\|_{L^2}\nonumber\\
 \leq &C|\nu n^2l|^{-\f12}\|(W_1,U)\|_{L^2}\|(F_1,F_2)\|_{L^2},\label{eq:la<0-est7}
\end{align}
which gives
\begin{align*}
   & |\nu n^2l|^{\f12}\|(W_1,U)\|_{L^2}\leq C\|(F_1,F_2)\|_{L^2}.
\end{align*} 
Thanks to 
\beno
|\nu \lambda l^2|^{\f13}=|\lambda l|^{\f13}|\nu l|^{\f13}\leq C(|\lambda l|+|\nu l|^{\f12})\leq C(|\lambda l|+|\nu n^2 l|^{\f12}),
\eeno
 we conclude that
\begin{align*}
 \big(|\nu nl|^{\f12}+|\nu\lambda l^2|^{\f13}\big)\|(W_1,U)\|_{L^2}\leq C(|\lambda l|+|\nu n^2 l|^{\f12})\|(W_1,U)\|_{L^2}\leq C\|(F_1,F_2)\|_{L^2},
\end{align*}
which along with \eqref{eq:la<0-est4} gives
\begin{align*}
   \left\|\partial_rW/\sqrt{n^2+r^2l^2}\right\|_{L^2}&\leq C|n^2l|^{-\f12}\|(W_1,U)\|_{L^2}^{\f12}\|(F_1,F_2)\|_{L^2}^{\f12}\\
   &\leq C|n|^{-1}|l|^{-\f12}\big(|\nu nl|^{\f12}+|\nu\lambda l^2|^{\f13}\big)^{-\f12}\|(F_1,F_2)\|_{L^2}.
\end{align*}
Then we obtain
\begin{align*}
   \|\partial_rW\|_{L^2}&\leq (|n|+|l|)\left\|\partial_rW/\sqrt{n^2+r^2l^2}\right\|_{L^2}\\& \leq C(|n|+|l|)|n|^{-1}|l|^{-\f12}(|\nu nl|^{\f12}+|\nu\lambda l^2|^{\f13})^{-\f12}\|(F_1,F_2)\|_{L^2}\\ &\leq C(|n|+|l|)^{\f52}|n|^{-2}|l|^{-\f12}(|\nu nl|^{\f12}+|\nu\lambda l^2|^{\f13})^{-\f12}\|(F_1,F_2)\|_{L^2}.
\end{align*}

This completes the proof of Proposition \ref{prop:res-real} in the case of $\la\le 0$.\smallskip

Let us point out that  \eqref{eq:la<0-est5},  \eqref{eq:la<0-est2} and \eqref{eq:la<0-est4} hold for all $\la\in \R$, and so does for 
the first line of \eqref{eq:la<0-est7}, i.e., 
\begin{align}
 \|(W_1,U)\|_{L^2}^2\leq C|n|^{-1}\|(W_{1},U)\|_{1}\|(rW_1,rU)\|_{L^2}.\label{eq:W1U-L2}
\end{align}
\subsection{Resolvent estimates when  $\la>1$}

Let us first prove the following important inequality. 

\begin{Lemma}\label{lem:E-key}
It holds that
\begin{align*}
 \int_{0}^{1}\left(\dfrac{r|\partial_rW|^2}{n^2+r^2l^2}+\dfrac{|W|^2}{r}\right)\mathrm{d}r \geq\int_{0}^{1}\dfrac{4n^2r|W|^2(|n|+1)}{(n^2+r^2l^2)^2(\lambda-r^2)}\mathrm{d}r.
\end{align*}
\end{Lemma}

\begin{proof}
  Let $\psi(r)=\dfrac{2r}{\lambda-r^2}-\dfrac{|n|}{r}$ and then
\begin{align*}
   \int_{0}^{1}\f{r|\partial_rW+\psi W|^2}{n^2+r^2l^2}\mathrm{d}r&= \int_{0}^{1}\f{r|\partial_rW|^2+r|\psi W|^2}{n^2+r^2l^2}\mathrm{d}r+\int_{0}^{1}\f{r\psi (\partial_r|W|^2)}{n^2+r^2l^2}\mathrm{d}r\\
   &=\int_{0}^{1}\f{r|\partial_rW|^2+r|\psi W|^2}{n^2+r^2l^2}\mathrm{d}r-\int_{0}^{1}\partial_r\bigg(\f{r\psi }{n^2+r^2l^2}\bigg)|W|^2\mathrm{d}r.
\end{align*}
Thanks to
\begin{align*}
   &\int_{0}^{1}\f{r|\psi W|^2}{n^2+r^2l^2}\mathrm{d}r=\int_{0}^{1}\bigg(\f{4r^3}{(\lambda-r^2)^2}- \f{4r|n|}{\lambda-r^2} +\f{n^2}{r}\bigg)\f{|W|^2}{n^2+r^2l^2}\mathrm{d}r,\\
& \partial_r\bigg(\f{r\psi(r) }{n^2+r^2l^2}\bigg) =\f{4r\lambda}{(\lambda-r^2)^2(n^2+r^2l^2)}-\f{4r^3l^2}{(\lambda-r^2)(n^2+r^2l^2)^2} +\f{2rl^2|n|}{(n^2+r^2l^2)^2},
\end{align*}
we deduce that
\begin{align*}
   & \int_{0}^{1}\f{r|\psi W|^2}{n^2+r^2l^2}\mathrm{d}r-\int_{0}^{1}\partial_r\bigg(\f{r\psi }{n^2+r^2l^2}\bigg)|W|^2\mathrm{d}r\\
   &= \int_{0}^{1} \bigg(\f{4r^3-4r\lambda}{(\lambda-r^2)^2}- \f{4r|n|}{\lambda-r^2}+\f{4r^3l^2}{(\lambda-r^2)(n^2+r^2l^2)} +\f{n^2}{r}-\f{2rl^2|n|}{n^2+r^2l^2}\bigg)\f{|W|^2}{n^2+r^2l^2}\mathrm{d}r\\
   &= \int_{0}^{1} \bigg(-\f{4r(|n|+1)}{(\lambda-r^2)}+\f{4r^3l^2}{(\lambda-r^2)(n^2+r^2l^2)} +\f{n^2}{r}-\f{2rl^2|n|}{n^2+r^2l^2}\bigg)\f{|w_1|^2}{n^2+r^2l^2}\mathrm{d}r\\
   &=  \int_{0}^{1} \bigg(-\f{4rn^2(|n|+1)+4r^3l^2|n|}{(\lambda-r^2)(n^2+r^2l^2)}+\f{n^2}{r} -\f{2rl^2|n|}{n^2+r^2l^2}\bigg)\f{|W|^2}{n^2+r^2l^2}\mathrm{d}r\\
   &\leq -\int_{0}^{1} \f{4rn^2(|n|+1)|W|^2}{(\lambda-r^2)(n^2+r^2l^2)^2}\mathrm{d}r+ \int_{0}^{1}\f{n^2|W|^2}{r(n^2+r^2l^2)}\mathrm{d}r.
\end{align*}

Summing up, we obtain
\begin{align*}
  0\leq &\int_{0}^{1}\f{r|\partial_rW+\psi W|^2}{n^2+r^2l^2}\mathrm{d}r= \int_{0}^{1}\f{r|\partial_rW|^2+r|\psi W|^2}{n^2+r^2l^2}\mathrm{d}r-\int_{0}^{1}\partial_r\bigg(\f{r\psi }{n^2+r^2l^2}\bigg)|W|^2\mathrm{d}r\\
  \leq& \int_{0}^{1}\f{r|\partial_rW|^2}{n^2+r^2l^2}\mathrm{d}r-\int_{0}^{1} \f{4rn^2(|n|+1)|W|^2}{(\lambda-r^2)(n^2+r^2l^2)^2}\mathrm{d}r+ \int_{0}^{1}\f{n^2|W|^2}{r(n^2+r^2l^2)}\mathrm{d}r,
\end{align*}
which gives
\begin{align*}&\int_0^1\left(\dfrac{r|\partial_rW|^2}{n^2+r^2l^2}+\dfrac{|W|^2}{r}\right)\mathrm{d}r\geq \int_0^{1}\dfrac{4n^2r|W|^2(|n|+1)}{(n^2+r^2l^2)^2(\lambda-r^2)}\mathrm{d}r.
\end{align*}

This proves the lemma.
\end{proof}\smallskip

Now we prove Proposition \ref{prop:res-real}.\smallskip

  Thanks to $E=-\left\langle\f{W}{n^2+r^2l^2},W_1 \right\rangle$, we  get by Lemma \ref{lem:energy} that 
  \begin{align}\label{eq:la>1-est1}
  &\left\|\sqrt{\lambda-r^2}U\right\|_{L^2}^2+\Big(\left\|\sqrt{\lambda-r^2}W\right\|_{L^2}^2 -4n^2E\Big)\nonumber\\
  &\quad\leq 2\nu|l|^{-1}|\langle U,\widehat{\Delta}_1^*W_1\rangle|+|l|^{-1}\|(U,W_1)\|_{L^2}\|(F_1,F_2)\|_{L^2}.
  \end{align}
  By Lemma \ref{lem:E-key}, we have
  \begin{align*}
     &2|n|\sqrt{|n|+1}\left\|\dfrac{(\lambda-r^2)^{-\f12}W}{n^2+r^2l^2}\right\|_{L^2}\leq E^{\f12},
  \end{align*}
  which gives
    \begin{align*}
     2E&=-2\left\langle\f{W}{n^2+r^2l^2},W_1 \right\rangle\leq 2\left\|\dfrac{(\lambda-r^2)^{-\f12}W}{n^2+r^2l^2}\right\|_{L^2} \left\|\sqrt{\lambda-r^2}W_1\right\|_{L^2}\\
     &\leq |n|^{-1}(|n|+1)^{-\f12}E^{\f12}\left\|\sqrt{\lambda-r^2}W_1\right\|_{L^2}\\
     &\leq E+4^{-1}|n|^{-2}(|n|+1)^{-1}\left\|\sqrt{\lambda-r^2}W_1\right\|_{L^2}^2.
  \end{align*}
  Hence,  we deduce that
  \begin{align}\label{eq:la>1-est2}
     \left\|\sqrt{\lambda-r^2}W_1\right\|_{L^2}^2\geq 4n^2(|n|+1)E,
  \end{align}
which along with  \eqref{eq:la>1-est1} gives
  \begin{align*}
  &\left\|\sqrt{\lambda-r^2}U\right\|_{L^2}^2+\left\|\sqrt{\lambda-r^2}W_1\right\|_{L^2}^2/2\\
  &\leq \left\|\sqrt{\lambda-r^2}U\right\|_{L^2}^2+\bigg(1-\dfrac{1}{|n|+1}\bigg)\left\|\sqrt{\lambda-r^2}W_1\right\|_{L^2}^2 \\&\leq \left\|\sqrt{\lambda-r^2}U\right\|_{L^2}^2+\Big(\left\|\sqrt{\lambda-r^2}W_1\right\|_{L^2}^2-4n^2E\Big)\\
  &\leq 2\nu|l|^{-1}\big|\langle U,\widehat{\Delta}_1^*W_1\rangle\big|+|l|^{-1}\|(U,W_1)\|_{L^2}\|(F_1,F_2)\|_{L^2}.
  \end{align*}
 We known from \eqref{eq:la<0-est5} that 
  \begin{align*}
     \big|\langle U,\widehat{\Delta}_1^*W_1\rangle\big| \leq C\|U\|_{1}\|W_1\|_{1}.
  \end{align*}
  Then we conclude that
  \begin{align}\label{eq:la>1-est3}
     \left\|\sqrt{\lambda-r^2}(U,W_1)\right\|_{L^2}^2& \leq C\big( \nu|l|^{-1}\|U\|_{1}\|W_1\|_{1}+|l|^{-1}\|(U,W_1)\|_{L^2}\|(F_1,F_2)\|_{L^2}\big).
  \end{align}
  
 Using the first equation of \eqref{eq:LNS-WU-a} to obtain
  \begin{align*}
    &\big\langle W_1,-\nu(2\widehat{\Delta}U-\widehat{\Delta}_1^{*} W_1)+\ir l(\lambda-r^2)U\big\rangle= \langle W_1,F_1\rangle.
  \end{align*}
 We infer from Lemma \ref{lem:W1-H1} that
  \begin{align*}
    \nu\|W_1\|_{1}^2&\leq 4\nu|\langle W_1,\widehat{\Delta}U\rangle|+2|l|\big|\left\langle W_1,(\lambda-r^2)U\right\rangle\big|+2\|W_1\|_{L^2}\|F_2\|_{L^2}\\
    &\leq C\big(\nu\|U\|_{1}\|W_1\|_{1}+|l|\|\sqrt{\lambda-r^2}(U,W_1)\|_{L^2}^2+\|W_1\|_{L^2}\|F_1\|_{L^2}\big),
  \end{align*}
  which along with \eqref{eq:la>1-est3} gives
  \begin{align*}
    \nu\|W_1\|_{1}^2 &\leq C\big(\nu\|U\|_{1}\|W_1\|_{1}+\|(U,W_1)\|_{L^2}\|(F_1,F_2)\|_{L^2}\big).
  \end{align*}
  Then Young's inequality and \eqref{eq:U-H1} ensure that 
  \begin{align}\label{eq:la>1-est4}
    \nu\|(U,W_1)\|_{1}^2 &\leq C\big(\nu\|U\|_{1}^2+\|(U,W_1)\|_{L^2}\|(F_1,F_2)\|_{L^2}\big)\\ \nonumber&\leq C\|(U,W_1)\|_{L^2}\|(F_1,F_2)\|_{L^2}.
  \end{align}
  This along with \eqref{eq:la>1-est3} gives
\begin{align}\label{eq:la>1-est5}
     \left\|\sqrt{\lambda-r^2}(U,W_1)\right\|_{L^2}^2& \leq C|l|^{-1}\|(U,W_1)\|_{L^2}\|(F_1,F_2)\|_{L^2}.
  \end{align}
  
   By Lemma \ref{lem:hardy-1}, \eqref{eq:la>1-est4} and \eqref{eq:la>1-est5}, we have
  \begin{align*}
     \|(U,W_1)\|_{L^2}&\leq C\lambda^{-\f13}\|(U,W_{1})\|_{1}^{\f13}\left\|\sqrt{\lambda-r^2}(U,W_1)\right\|_{L^2}^{\f23}\\
     &\leq C\nu^{-\f16}\lambda^{-\f13}|l|^{-\f13}\big( \|(U,W_1)\|_{L^2}\|(F_1,F_2)\|_{L^2}\big)^{\f12}.
  \end{align*}
  This shows that 
  \begin{align*}
  \|(U,W_1)\|_{L^2}&\leq C|\nu l^2\lambda^2|^{-\f13}\|(F_1,F_2)\|_{L^2}.
  \end{align*}
 
Thanks to  $\|(U,W_1)\|_{1}^2\geq(l^2+n^2)\|(U,W_1)\|_{L^{2}}^2 $, we get by \eqref{eq:la>1-est4} that 
\begin{align*}
   \nu(l^2+n^2)\|(U,W_1)\|_{L^{2}}\leq &C\|(F_1,F_2)\|_{L^2}.
\end{align*}
Thanks to $|\nu nl|^{\f12}=|\nu n^2|^{\f14}|\nu l^2|^{\f14}\leq C(\nu n^2+|\nu\lambda^2 l^2|^{\f13})$, we have
\begin{align*}
   &(|\nu nl|^{\f12}+|\nu\lambda l^2|^{\f13})\|(W_1,U)\|_{L^2}\leq C(\nu n^2+|\nu\lambda^2 l^2|^{\f13})\|(W_1,U)\|_{L^2}\leq C\|(F_1,F_2)\|_{L^2}.
\end{align*}

By \eqref{eq:la>1-est2} and \eqref{eq:la>1-est5}, we get
\begin{align*}
   n^2(|n|+1)E&\leq C|l|^{-1}\|(U,W)\|_{L^2}\|(F_1,F_2)\|_{L^2}\\
   &\leq C|l|^{-1}(|\nu nl|^{\f12}+|\nu\lambda l^2|^{\f13})^{-1}\|(F_2,F_3)\|_{L^2}^2.
\end{align*}
Then we have
\begin{align*}
   \|\partial_rW\|_{L^2}&\leq C(|n|+|l|)\left\|\partial_rW/\sqrt{n^2+r^2l^2}\right\|_{L^2}\leq C(|n|+|l|)E^{\f12}\\
   &\leq C(|n|+|l|)|n|^{-1}|l|^{-\f12}(|\nu nl|^{\f12}+|\nu\lambda l^2|^{\f13})^{-\f12}\|(F_1,F_2)\|_{L^2}\\
   &\leq C(|n|+|l|)^{\f52}|n|^{-2}|l|^{-\f12}(|\nu nl|^{\f12}+|\nu\lambda l^2|^{\f13})^{-\f12}\|(F_1,F_2)\|_{L^2}.
\end{align*}

This completes the proof of Proposition \ref{prop:res-real} in the case of $\la>1$.

\subsection{Resolvent estimates when $\la\in (0,1]$}
 Let $r_0=\lambda^{\f12},\ \delta=(\nu/|lr_0|)^{\f13}.$ We have $|\nu/(l\la^2)|=|\delta/r_0|^3$.
 The proof of this case is based on the following key proposition, whose proof is very technical and will be left  to section 4.5.

 \begin{Proposition}\label{prop:res-real-key}
There exists a constant $c_2\in (0,\f12)$ such that if $\lambda\in(0,1]$, $\delta/r_0\leq \min(c_2,1/|n|)$ and  $|\delta l|\leq c_2,$ then we have
 \begin{align*}
&lr_0\delta\|(W_1,U)\|_{L^2}\leq  C\| (F_1,F_2)\|_{L^2},
\\&E^{\f12}\leq Cn^{-2}|l|^{-1}(r_0\delta)^{-\f12}(|n|+r_0|l|)^{\f32}\| (F_1,F_2)\|_{L^2}.
\end{align*}
\end{Proposition}

\subsubsection{Case of $\delta/r_0\le \min(c_2,1/|n|)$ and $|\delta l|\le c_2$}

Notice that 
\beno
&&|lr_0\delta|=|\nu\lambda l^2|^{\f13},\\
&&|\nu nl|^{\f12}= |\nu\lambda l^2|^{\f13}(\nu^{\f16}|n|^{\f12}\lambda^{-\f13}|l|^{-\f16})=|lr_0\delta||n\delta/r_0|^{\f12}\leq |lr_0\delta|,
\eeno 
hence, $|\nu nl^2|^{\f12}+|\nu \lambda l^2|^{\f13}\leq 2|lr_0\delta|$.  Then it follows from Proposition \ref{prop:res-real-key} that 
 \begin{align*}
 \big(|\nu nl|^{\f12}+|\nu \lambda l^2|^{\f13}\big)\|(W_1,U)\|_{L^2}\leq 2|lr_0\delta|\|(W_1,U)\|_{L^2}\leq C\|(F_1,F_2)\|_{L^2},
\end{align*}
and 
\begin{align*}
   \|\partial_rW\|_{L^2}&\leq (|n|+|l|)E^{\f12}\leq C(|n|+|l|)n^{-2}|l|^{-1}(r_0\delta)^{-\f12}(|n|+r_0|l|)^{\f32}\| (F_1,F_2)\|_{L^2}\\
   &\leq  Cn^{-2}|l|^{-\f12}|lr_0\delta|^{-\f12}(|n|+|l|)^{\f52}\|(F_1,F_2)\|_{L^2}\\
   &\leq  Cn^{-2}|l|^{-\f12}(|n|+|l|)^{\f52}(|\nu nl|^{\f12}+|\nu \lambda l^2|^{\f13})^{-\f12}\| (F_1,F_2)\|_{L^2}.
\end{align*}

\subsubsection{Case of $\delta/r_0\geq \min(c_2,1/|n|)$}

It follows from Lemma \ref{lem:energy}, \eqref{eq:la<0-est5} and \eqref{eq:la<0-est7} that 
\begin{align*}
   &|l|\|(rU,rW_1)\|_{L^2}^2 +4n^2|l|\big(\|\partial_rW/\sqrt{n^2+r^2l^2}\|_{L^2}^2+\| W/r\|_{L^2}^2\big)\\
   &\leq \big|\langle U,F_1\rangle-\langle F_2,W_1\rangle\big|+2\nu\big|\langle U,\widehat{\Delta}_1^* {W}_1\rangle\big|+\lambda|l|\|(U,W_1)\|_{L^2}^2\\
   &\leq \|(U,W_1)\|_{L^2}\|(F_1,F_2)\|_{L^2}+C\nu\|U\|_{1}\|W_1\|_{1}+C\lambda|l/n| \|(rU,rW_1)\|_{L^2}\|(U,W_1)\|_{1},
\end{align*}
from which and Young's inequality, we infer that 
\begin{align*}
   &|l|\|(rU,rW_1)\|_{L^2}^2 +n^2|l|\big(\|\partial_rW/\sqrt{n^2+r^2l^2}\|_{L^2}^2+\| W/r\|_{L^2}^2\big)\\
   &\leq C\Big(\|(U,W_1)\|_{L^2}\|(F_1,F_2)\|_{L^2}+(\nu+\lambda^2|l|n^{-2})\|(U,W_1)\|_{1}^2\Big).
\end{align*}
Thanks to 
\beno
\lambda^2|l|n^{-2}=r_0^4|l|n^{-2}\leq \nu n^{-2}(r_0/\delta)^{3}\leq \max(c_2^{-3},|n|^3)\nu n^{-2}\leq (c_2^{-3}+|n|)\nu,
\eeno
 we get
\begin{align}\label{eq:lam-est1}
   &\|rW_1\|_{L^2}^2+n^2\|W/r\|_{L^2}^2 +n^2\left\|\partial_rW/\sqrt{n^2+r^2l^2}\right\|_{L^2}^2\nonumber\\&\leq  C|l|^{-1}\Big(\|(U,W_1)\|_{L^2}\|(F_1,F_2)\|_{L^2}+\nu|n|\|(U,W_1)\|_{1}^2\Big).
\end{align}

Recall the equality \eqref{eq:la<0-est2}, which shows 
\begin{align*}
   &-\nu\langle W_1,\widehat{\Delta}^{*}_1W_1\rangle+\nu\langle \widehat{\Delta}_1U,U\rangle+2\nu\langle W_1,\widehat{\Delta}U\rangle-4\ir n^2l\left\langle \f{W}{n^2+r^2l^2},U\right\rangle \\
   &=-\langle W_1,F_1\rangle-\langle F_2,U\rangle.\nonumber
\end{align*}
We get by integration by parts that 
\begin{align*}
   &|\langle W_1,\widehat{\Delta}U\rangle|+|\langle \widehat{\Delta}_1U,U\rangle|\leq C\big(\|W\|_1\|U\|_{1}+\|U\|_{1}^2\big),
\end{align*}
and 
\begin{align*}
   \left|\bigg\langle\f{4n^2lW}{n^2+r^2l^2}, U\bigg\rangle\right|&\leq 4|l|\|W/r\|_{L^2}\|rU\|_{L^2}.
\end{align*}
We get by Lemma \ref{lem:W1-H1} that 
\begin{align*}
 &\nu\|W_1\|_{1}^2\leq -2\nu\mathbf{Re}\langle W_1,\widehat{\Delta}^{*}_1W_1\rangle\\
  &\leq 2\nu\big|\langle \widehat{\Delta}_1U,U\rangle\big|+4\nu\big|\langle W_1,\widehat{\Delta}U \rangle \big|+2\big|\langle W_1,F_1\rangle+\langle F_2,U\rangle\big| +2\left|\bigg\langle\f{4n^2lW}{n^2+r^2l^2},U \bigg\rangle\right|
 \\&\leq C\Big(\nu\big(\|U\|_{1}^2+\|U\|_{1}\|W_1\|_{1}\big)+\|(W_1,U)\|_{L^2}\|(F_1,F_2)\|_{L^2} +|l|\|W/r\|_{L^2}\|rU\|_{L^2}\Big),
\end{align*}
from which and \eqref{eq:lam-est1}, we infer that
\begin{align*}
 \|W_1\|_{1}^2\leq& C\big(\|U\|_{1}^2+\nu^{-1}\|(W_1,U)\|_{L^2}\|(F_1,F_2)\|_{L^2} +\nu^{-1}|l|\|W/r\|_{L^2}\|rU\|_{L^2}\big)\\
 \leq& C\Big(\|U\|_{1}^2+\nu^{-1}\|(W_1,U)\|_{L^2}\|(F_1,F_2)\|_{L^2} \\&+\nu^{-1}|l|^{\f12}|n|^{-1}\big(\|(U,W_1)\|_{L^2}\|(F_1,F_2)\|_{L^2}+\nu|n|\|(U,W_1)\|_{1}^2\big)^{\f12}\|rU\|_{L^2}\Big)\\
 \leq& C\Big(\|U\|_{1}^2+\nu^{-1}\|(W_1,U)\|_{L^2}\|(F_1,F_2)\|_{L^2}+|\nu n^2/l|^{-1}\|rU\|_{L^2}^2\\&\quad+|\nu n/l|^{-\f12}\|(U,W_1)\|_{1}\|rU\|_{L^2}\Big).
\end{align*}
This along with \eqref{eq:U-H1}  shows that
\begin{align}\label{eq:lam-est2}
 \|(U,W_1)\|_{1}^2\leq& C\big(\nu^{-1}\|(W_1,U)\|_{L^2}\|(F_1,F_2)\|_{L^2} +|\nu n/l|^{-1}\|rU\|_{L^2}^2\big).
\end{align}

Now let us estimate $\|rU\|_{L^2}$. We have
  \begin{align*}
     &\langle U,F_1\rangle=\big\langle U,-\nu(2\widehat{\Delta}U-\widehat{\Delta}_{1}^{*}W_1)+ \ir l(\lambda-r^2)U\big\rangle,
  \end{align*}
which gives
  \begin{align*}
     &\mathbf{Im}\langle U,F_1\rangle=\nu\mathbf{Im}\langle U,\widehat{\Delta}_1^* {W}_1\rangle+l\big(\| rU\|_{L^2}^2-\lambda\|U\|_{L^2}^2\big).
  \end{align*}
  Then we have
  \begin{align*}
     \|rU\|_{L^2}^2&\leq C\big(\lambda\|U\|_{L^2}^2+|\nu/l|\|U\|_{1}\|W_1\|_{1}+|l|^{-1} \|U\|_{L^2}\|F_1\|_{L^2}\big)\\
     &\leq C\big(\lambda|n|^{-1}\|rU\|_{L^2}\|U\|_{1}+|\nu/l|\|U\|_{1}\|W_1\|_{1}+|l|^{-1} \|U\|_{L^2}\|F_1\|_{L^2}\big).
  \end{align*}
  Here we used $\|U\|_{L^2}^2\leq \|rU\|_{L^2}\|U/r\|_{L^2}\leq |n|^{-1}\|rU\|_{L^2}\|U\|_{1}$. Thanks to 
  \begin{align*}\lambda^2|n|^{-2}=|\nu/l||n|^{-2}(r_0/\delta)^{3}\leq |\nu/l||n|^{-2}\max(c_2^{-3},|n|^3)\leq (c_2^{-3}+|n|)|\nu/l| , \end{align*} we get by \eqref{eq:lam-est2}  that 
  \begin{align*}
     \|rU\|_{L^2}^2&\leq C\big(\lambda^2|n|^{-2}\|U\|_{1}^2+|\nu/l|\|U\|_{1}\|W_1\|_{1}+|l|^{-1} \|U\|_{L^2}\|F_1\|_{L^2}\big)\\
     &\leq C\big(|\nu n/l|\|U\|_{1}(\|W_1\|_{1}+\|U\|_{1})+|l|^{-1} \|U\|_{L^2}\|F_1\|_{L^2}\big)\\
     &\leq C|\nu n/l|\|U\|_{1}\big(\nu^{-1}\|(W_1,U)\|_{L^2}\|(F_1,F_2)\|_{L^2} +|\nu n/l|^{-1}\|rU\|_{L^2}^2\big)^{\f12}\\ &\qquad+C|l|^{-1} \|U\|_{L^2}\|F_1\|_{L^2}.
  \end{align*}
  Thus, we obtain
  \begin{align*}
     \|rU\|_{L^2}^2&\leq C\big(|n|\nu^{\f12}|l|^{-1}\|U\|_{1}\|(W_1,U)\|_{L^2}^{\f12}\|(F_1,F_2)\|_{L^2}^{\f12} +|\nu n/l|\|U\|_{1}^2+|l|^{-1}\|U\|_{L^2}\|F_1\|_{L^2}\big).
  \end{align*}
  This along with \eqref{eq:U-H1}  gives 
  \begin{align*}
     \|rU\|_{L^2}^2&\leq C|n/l|\|(U,W_1)\|_{L^2}\|(F_1,F_2)\|_{L^2}.
  \end{align*}
Then we infer from \eqref{eq:lam-est2} that
   \begin{align}\label{eq:lam-est3}
 \|(U,W_1)\|_{1}^2\leq C\nu^{-1}\|(W_1,U)\|_{L^2}\|(F_1,F_2)\|_{L^2}.
\end{align}
By \eqref{eq:lam-est1} and \eqref{eq:lam-est3}, we have
\begin{align}\label{eq:lam-est4}
  \|(rU,rW_1)\|_{L^2}^2+n^2\left\|\partial_rW/\sqrt{n^2+r^2l^2}\right\|_{L^2}^2 &\leq C|n/l|\|(W_1,U)\|_{L^2}\|(F_1,F_2)\|_{L^2}.
\end{align}

Summing up, we get by \eqref{eq:W1U-L2}  that
\begin{align*}
   \|(U,W_1)\|_{L^2}&\leq |n|^{-1}\|(U,W_1)\|_{1}\|(rU,rW_1)\|_{L^2}\leq C|\nu nl|^{-\f12}\|(W_1,U)\|_{L^2}\|(F_1,F_2)\|_{L^2},
\end{align*}
which shows that
\begin{align*}
   |\nu nl|^{\f12}\|(U,W_1)\|_{L^2}&\leq C\|(F_1,F_2)\|_{L^2}.
\end{align*}

As $\delta/r_0\geq \min(c_2,1/|n|),$ we have $|\nu/(l\lambda^2)|\geq \min(c_2,1/|n|)^3\geq c_2^3/|n|^3$, and then $|l\lambda^2/(\nu n^3)|\leq C$. Thus, $|\nu\lambda l^2|^{\f13}=|\nu nl|^{\f12}|\lambda^2l/(\nu n^3)|^{\f16}\leq C|\nu nl|^{\f12}$ and 
\begin{align*}
   &(|\nu nl|^{\f12}+|\nu\lambda l^2|^{\f13})\|(W_1,U)\|_{L^2}\leq C|\nu nl|^\f12\|(W_1,U)\|_{L^2}\leq C\|(F_1,F_2)\|_{L^2}.
\end{align*}
Then we get by \eqref{eq:lam-est4} that
\begin{align*}
   \|\partial_rW\|_{L^2}&\leq C(|n|+|l|)\left\|\partial_rW/\sqrt{n^2+r^2l^2}\right\|_{L^2} \\ &\leq C(|n|+|l|)|n|^{-1}\big(|n/l|\|(W_1,q)\|_{L^2}\|(F_1,F_2)\|_{L^2}\big)^{\f12}\\
   &\leq C(|n|+|l|)|n|^{-\f12}|l|^{-\f12}(|\nu nl|^{\f12}+|\nu\lambda l^2|^{\f13})^{-\f12}\|(F_1,F_2)\|_{L^2}\\
   &\leq C(|n|+|l|)^{\f52}|n|^{-2}|l|^{-\f12}(|\nu nl|^{\f12}+|\nu\lambda l^2|^{\f13})^{-\f12}\|(F_1,F_2)\|_{L^2}.
\end{align*}

\subsubsection{Case of $\delta/r_0\le \min(c_2,1/|n|)$ and $|\delta l|\ge c_2$}

In this case, we have $c_2^3\leq |l\delta|^3=\nu l^2/r_0$, and then $\nu l^2/c_2^3\geq r_0\geq\delta/c_2\geq|1/l|$, thus $\nu |l|^3\geq c_2^3$. Let $U_1=U-W_1$, which satisfies
  \begin{align*}
     &-\nu\widehat{\Delta}_{1}^{*}U_1+{i}l(\lambda-r^2)U_1=F_1-F_2+ \nu(2\widehat{\Delta}-\widehat{\Delta}_{1}-\widehat{\Delta}_{1}^{*})U+\dfrac{4{\ir}n^2lW}{n^2+r^2l^2}.
  \end{align*}
Then we find that
  \begin{align*}
     &-\nu\langle U_1,\widehat{\Delta}_{1}^{*}U_1\rangle-{\ir}l\langle U_1,(\lambda-r^2)U_1\rangle \\
     &=\big\langle U_1, F_1-F_2+ \nu(2\widehat{\Delta}-\widehat{\Delta}_{1}-\widehat{\Delta}_{1}^{*})U\big\rangle+\left\langle U_1, \dfrac{4{\ir}n^2lW}{n^2+r^2l^2}\right\rangle\\
     &=\big\langle U_1, F_1-F_2+ \nu(2\widehat{\Delta}-\widehat{\Delta}_{1}-\widehat{\Delta}_{1}^{*})U\big\rangle-\left\langle W_1,\dfrac{4{\ir}n^2lW}{n^2+r^2l^2}\right\rangle+\left\langle U, \dfrac{4{\ir}n^2lW}{n^2+r^2l^2}\right\rangle\\
     &=\big\langle U_1, F_1-F_2+ \nu(2\widehat{\Delta}-\widehat{\Delta}_{1}-\widehat{\Delta}_{1}^{*})U \big\rangle+4{\ir}n^2lE+\left\langle U, \dfrac{4{\ir}n^2lW}{n^2+r^2l^2}\right\rangle.
  \end{align*}
Thanks to $E\in\mathbb{R}$ and  $\|U_1\|_{1}^2\leq-2\mathbf{Re}\langle U_1,\widehat{\Delta}_1^*U_1\rangle$, we get
  \begin{align*}
     \nu\|U_1\|_{1}^2&\leq C\Big(\|U_1\|_{L^2}\|(F_1,F_2)\|_{L^2}+ \nu\|U_1\|_{L^2}\left\|(2\widehat{\Delta}-\widehat{\Delta}_{1}-\widehat{\Delta}_{1}^{*})U\right\|_{L^2}\\
     &\qquad+\|U\|_{L^2}\left\|\dfrac{4{\ir}n^2lW}{n^2+r^2l^2}\right\|_{L^2}\Big).
  \end{align*}
  
Thanks to  $\|U\|_{L^2}\leq |l|^{-1}\|U\|_{1}$ and $\|U_1\|_{L^2}\leq |l|^{-1}\|U_1\|_{1}$, we have
  \begin{align*}
    \nu|l|\|U_{1}\|_{1} &\leq C\Big( \|(F_{1},F_{2})\|_{L^2}+\nu\left\|(2\widehat{\Delta}-\widehat{\Delta}_{1}-\widehat{\Delta}_{1}^{*})U\right\|_{L^2} +(\nu|l|\|U\|_{1})^{\f12}\left\|\dfrac{4{\ir}n^2lW}{n^2+r^2l^2}\right\|_{L^2}^{\f12}\Big).
  \end{align*}
Due to  $(n^2+r^2l^2)^2\geq n^2|2nrl|$, we get
  \begin{align*}
     \left\|(2\widehat{\Delta}-\widehat{\Delta}_{1}-\widehat{\Delta}_{1}^{*})U\right\|_{L^2}& = \left\|\dfrac{2rl^2}{n^2+r^2l^2}\partial_rU-\dfrac{1}{r}\partial_r\dfrac{2r^2l^2U}{n^2+r^2l^2}\right\|_{L^2}= \left\|\dfrac{4l^2n^2}{(n^2+r^2l^2)^2}U\right\|_{L^2}\\ &\leq 2|l/n|\|U/r\|_{L^2}\leq2|l/n^2|\|U\|_{1}.
  \end{align*}
 Using the fact that
  \begin{align}\label{eq:lam-est5}
     |l|^2\|W\|_{L^2}^2&\leq\|W\|_{1}^2\leq-2\mathbf{Re}\langle W,\widehat{\Delta}_1W\rangle=-2\mathbf{Re}\langle W,W_1\rangle\leq 2\|W\|_{L^2}\|W_1\|_{L^2},
  \end{align}
we deduce that 
   \begin{align}\label{eq:lam-est6}
   \left\|\dfrac{4{\ir}n^2lW}{n^2+r^2l^2}\right\|_{L^2}&\leq 4|l|\|W\|_{L^2}\leq 2|l|^{-1}\|W_1\|_{L^2}\leq C|l|^{-2}\|W_1\|_{1}\nonumber\\
     &\leq C|l|^{-2}\|(U,U_1)\|_{1}\leq Cc_2^{-3}\nu|l|\|(U,U_1)\|_{1}.
  \end{align}   
 Here we used $c_2^3 \leq \nu|l|^3$. Summing up,  we conclude
  \begin{align*}
     \nu|l|\|U_{1}\|_{1} &\leq C\big( \|(F_{1},F_{2})\|_{L^2}+\nu|l/n^2|\|U\|_{1} +(\nu|l|\|U\|_{1})^{\f12}(\nu|l|\|(U_1,U)\|_{1})^{\f12}\big)\\
     &\leq C\big( \|(F_{1},F_{2})\|_{L^2}+\nu|l|\|U\|_{1} +(\nu|l|\|U\|_{1})^{\f12}(\nu|l|\|(U_1,U)\|_{1})^{\f12}\big),
  \end{align*}
  from which and  Young's inequality, we infer that
  \begin{align*}
     \nu|l|\|(U,U_{1})\|_{1} &\leq C\big( \|(F_{1},F_{2})\|_{L^2}+\nu|l|\|U\|_{1}\big).
  \end{align*}
Then by using \eqref{eq:U-H1} and  $\|(U,W_1)\|_{L^2}\leq C\|(U,U_1)\|_{L^2}\leq C|l|^{-1}\|(U,U_1)\|_{1}$, we get
  \begin{align*}
     \nu|l|\|(U,U_{1})\|_{1} &\leq C\big( \|(F_{1},F_{2})\|_{L^2}+\|(F_1,F_2)\|_{L^2}^{\f12}(\nu l^2\|(U,W_1)\|_{L^2})^{\f12}\big)\\
     &\leq C\big( \|(F_{1},F_{2})\|_{L^2}+\|(F_1,F_2)\|_{L^2}^{\f12}(\nu |l|\|(U,U_1)\|_{1})^{\f12}\big).
  \end{align*}
 Then by Young's inequality, we obtain 
  \begin{align}\label{eq:lam-est7}
     \nu|l|\|(U,U_{1})\|_{1}&\leq C\|(F_1,F_2)\|_{L^2}.
  \end{align}
   Using the fact that $\|(U,W_1)\|_{L^2}\leq C|l|^{-1}\|(U,U_1)\|_{1}$ again, we deduce that
  \begin{align*}
    \nu l^2\|(U,W_1)\|_{L^2} &\leq C\|(F_1,F_2)\|_{L^2}.
  \end{align*}

  As $\delta/r_0\leq \min(c_2,1/|n|)\leq |n|^{-1}$, we have $|\nu n^3/(l\lambda^2)|\leq 1$ and 
  \begin{align*}
     & |\nu nl|^{\f12}=|\nu \lambda l^2|^{\f13}|\nu n^3/(l\lambda^2)|^{\f16}\leq |\nu \lambda l^2|^{\f13},\\
  &|\nu\lambda l^2|^{\f13}=\nu l^2(\nu^{-\f23}\lambda^{\f13}|l|^{-\f43})=\nu l^2|\delta l|^{-2}\leq c_2^{-2}\nu l^2\leq C\nu l^2.
  \end{align*}
  Then we conclude that
  \begin{align*}
     &(|\nu nl|^{\f12}+|\nu\lambda l^2|^{\f13})\|(W_1,U)\|_{L^2}\leq C\nu l^2\|(W_1,U)\|_{L^2}\leq C\|(F_1,F_2)\|_{L^2}.
  \end{align*}
  
  By \eqref{eq:lam-est6} and \eqref{eq:lam-est7}, we get
  \begin{align*}
     &\|W\|_{L^2}\leq C(|n|+|l|)^{2}|n|^{-2}|l|^{-1}\left\|\dfrac{4{\ir}n^2lW}{n^2+r^2l^2}\right\|_{L^2} \leq C(|n|+|l|)^{2}|n|^{-2}|l|^{-1}\|(F_1,F_2)\|_{L^2}.
  \end{align*}
  Thanks to $\|\partial_rW\|_{L^2}^2\leq 2\mathbf{Re} \langle- \widehat{\Delta}_1W,W\rangle\leq 2\|W_1\|_{L^2}\|W\|_{L^2}$,  we obtain  \begin{align*}
     \|\partial_rW\|_{L^2}&\leq \sqrt{2}\|W\|_{L^2}^{\f12}\|W_1\|_{L^2}^{\f12}\\
     &\leq C(|n|+|l|)|n|^{-1}|l|^{-\f12}(|\nu nl|^{\f12}+|\nu\lambda l^2|^{\f13})^{-\f12}\|(F_1,F_2)\|_{L^2}\\
     & \leq C (|n|+|l|)^{\f52}|n|^{-2}|l|^{-\f12}(|\nu nl|^{\f12}+|\nu\lambda l^2|^{\f13})^{-\f12}\|(F_1,F_2)\|_{L^2}.
  \end{align*}

\subsection{Resolvent estimates when $\la\in \C$}

Throughout this subsection, we always assume that $\lambda\in\C$ and  $l\lambda_i\leq c_1|\nu n l|^{\f12}.$\smallskip

Now let us prove Proposition \ref{prop:res-com}.

 \begin{proof}
 Using the equation \eqref{eq:LNS-WU-a}, we obtain
 \begin{align*}
    &\big\langle U,-\nu (2\widehat{\Delta} {U}-\widehat{\Delta}_1^*{W}_1)+{\ir}l(\lambda-r^2){U}\big\rangle+ \left\langle -\nu \widehat{\Delta}_1{U}+{\ir}l(\lambda-r^2){W}_1+\dfrac{4{\ir}n^2l{W}}{n^2+r^2l^2},W_1 \right\rangle\\
    &=\langle U,F_1\rangle+\langle F_2,W_1\rangle.
 \end{align*}
 Thanks to 
 \beno
 &&\langle U,-(2\widehat{\Delta}U-\widehat{\Delta}_{1}^{*}W_1)\rangle -\langle\widehat{\Delta}_1U,W_1\rangle=\langle U,-2\widehat{\Delta}U\rangle=2\|U\|_{1}^2,\\
  &&\mathbf{Re}(\langle U,{\ir}l(\lambda-r^2){U}\rangle+\langle {\ir}l(\lambda-r^2){W}_1, W_1\rangle)=-l\lambda_i\|(W_1,U)\|_{L^2}^2,\\
  &&\mathbf{Re}\left\langle \dfrac{4{\ir}n^2l{W}}{n^2+r^2l^2},W_1 \right\rangle=0.
  \eeno
 we obtain 
  \begin{align*}
    &\mathbf{Re}\big(\langle U,F_1\rangle+\langle F_2,W_1\rangle\big)=2\nu\|U\|_{1}^2-l\lambda_i\|(W_1,U)\|_{L^2}^2.
 \end{align*}
If $l\lambda_i\leq 0$, then we have
\begin{align*}
  |l\lambda_i|\|(W_1,U)\|_{L^2}\leq \|(F_1,F_2)\|_{L^2}.
\end{align*}
As $l\lambda_i\leq c_1|\nu nl|^{\f12}$, we conclude that
\begin{align}\label{eq:p2-est1}
   &|l\lambda_i|\|(W_1,U)\|_{L^2}\leq \|(F_1,F_2)\|_{L^2}+c_1|\nu nl|^{\f12}\|(W_1,U)\|_{L^2}.
\end{align}

Now we rewrite the system \eqref{eq:LNS-WU-a} as
   \begin{align*}
   \left\{\begin{array}{l}-\nu (2\widehat{\Delta} {U}-\widehat{\Delta}_1^*{W}_1)+{\ir}l(\lambda_r-r^2){U}={F}_1+l\lambda_i{U},\\-\nu \widehat{\Delta}_1{U}+{\ir}l(\lambda_r-r^2){W}_1+\dfrac{4{\ir}n^2l{W}}{n^2+r^2l^2}
={F}_2+l\lambda_i{W}_1,\\{W}_1=\widehat{\Delta}_1{W},\ {W}|_{r=1}={W}_1|_{r=1}={U}|_{r=1}=0.\end{array}\right.
\end{align*}
Then by Proposition \ref{prop:res-real} and \eqref{eq:p2-est1}, we get
\begin{align*}
  \big(|\nu nl|^{\f12}+|\nu\lambda_rl^2|^{\f13}\big)\|(W_1,U)\|_{L^2}&\leq C\|({F}_1+l\lambda_i{U}, {F}_2+l\lambda_i{W_1})\|_{L^2}\\
   &\leq C\big(\|(F_1,F_2)\|_{L^2}+|l\lambda_i|\|(W_1,U_1)\|_{L^2}\big)\\
   &\leq C\big(\|(F_1,F_2)\|_{L^2}+c_1|\nu nl|^{\f12}\|(W_1,U)\|_{L^2}\big).
\end{align*}
Taking $c_1$ sufficiently small so that $Cc_1\leq1/2$, we get
\begin{align*}
   & (|\nu nl|^{\f12}+|\nu\lambda_rl^2|^{\f13})\|(W_1,U)\|_{L^2}\leq C\|(F_1,F_2)\|_{L^2},
\end{align*}
and by \eqref{eq:p2-est1}, 
\begin{align*}
   &|l\lambda_i|\|(W_1,U)\|_{L^2}\leq C\|(F_1,F_2)\|_{L^2}.
\end{align*}
Thanks to $|\nu\lambda_il^2|^{\f13}\leq |n|^{-\f13}|l\lambda_i|^{\f13}(|\nu nl|^{\f12})^{\f23}\leq C(|l\lambda_i|+|\nu nl|^{\f12})$,
we infer that
\begin{align*}
   & (|\nu nl|^{\f12}+|\nu\lambda l^2|^{\f13}+|l\lambda_i|)\|(W_1,U)\|_{L^2}\leq C\|(F_1,F_2)\|_{L^2},
\end{align*}
which gives the first inequality of the proposition.

By Proposition \ref{prop:res-real} and \eqref{eq:p2-est1} again, we get
\begin{align*}
   \|\partial_rW\|_{L^2}&\leq C(|n|+|l|)^{\f52}n^{-2}|l|^{-\f12}(|\nu nl|^{\f12}+|\nu\lambda_r l^2|^{\f13})^{-\f12}\|(F_1+l\lambda_iU,F_2+l\lambda_iW_1)\|_{L^2}\\
   &\leq C(|n|+|l|)^{\f52}n^{-2}|l|^{-\f12}(|\nu nl|^{\f12}+|\nu\lambda_r l^2|^{\f13})^{-\f12}\|(F_1,F_2)\|_{L^2},
\end{align*}
which gives the second inequality of the proposition. 

We get by Lemma \ref{lem:sob} that
\begin{align*}
  |\partial_rW(1)|^2\leq& C\|\partial_rW\|_{L^2}\|\widehat{\Delta}_1W\|_{L^2} =C\|\partial_rW\|_{L^2}\|W_1\|_{L^2}\\
   \leq &C (|n|+|l|)^{\f52}n^{-2}|l|^{-\f12}(|\nu nl|^{\f12}+|\nu\lambda_r l^2|^{\f13})^{-\f12}\\
   &\times(|\nu nl|^{\f12}+|\nu\lambda l^2|^{\f13}+|l\lambda_i|)^{-1}\|(F_1,F_2)\|_{L^2}^2.
\end{align*}
Thanks to $|\nu \lambda_il^2|^{\f12}= (|\nu l|^{\f14})\times(|\nu l|^{\f14}|l\lambda_i|^{\f12})\leq (|\nu l|^{\f14})\big(\dfrac{|\nu l|^{\f12}+|l\lambda_i|}{2}\big)$, we find that
\begin{align*}
    (|\nu nl|^{\f12}+|\nu\lambda l^2|^{\f13})^{\f32}&\leq C\big((|\nu nl|^{\f12}+|\nu\lambda_r l^2|^{\f13})^{\f32}+|\nu \lambda_il^2|^{\f12}\big)\\
   &\leq  C\big((|\nu nl|^{\f12}+|\nu\lambda_r l^2|^{\f13})^{\f32}+|\nu l|^{\f14}(|\nu l|^{\f12}+|l\lambda_i|)\big)\\
   &\leq C\big(|\nu nl|^{\f12}+|\nu\lambda_r l^2|^{\f13}\big)^{\f12}\big(|\nu nl|^{\f12}+|\nu\lambda l^2|^{\f13}+|l\lambda_i|\big).
\end{align*}
Then we conclude that
\begin{align*}
   & |\partial_rW(1)|^2\leq C(|n|+|l|)^{\f52}n^{-2}|l|^{-\f12}(|\nu nl|^{\f12}+|\nu\lambda l^2|^{\f13})^{-\f32}\|(F_1,F_2)\|_{L^2}.
\end{align*}
This shows the third inequality of the proposition.
 \end{proof}

\subsection{Proof of Proposition \ref{prop:res-real-key}} 
Let us recall that $r_0=\lambda^{\f12}$ and $\delta=(\nu/|lr_0|)^{\f13}.$ Here we always assume that $\delta/r_0\leq \min(c_2,1/|n|)$ and $|\delta l|\leq c_2$. \smallskip

Let  $U_1=U-W_1,$ which satisfies
\begin{align}\label{eq:U1}
&-\nu \widehat{\Delta} {U}_1+{\ir}l(\lambda-r^2){U}_1={F}_{1,2}+\dfrac{4{\ir}n^2l{W}}{n^2+r^2l^2},
\end{align}
where
\ben\label{def:F21}
F_{1,2}:=F_1-F_2+\nu (\widehat{\Delta}-\widehat{\Delta}_1) {U}+\nu (\widehat{\Delta}-\widehat{\Delta}_1^*) {W}_1.
\een
We rewrite the first equation of \eqref{eq:LNS-WU-a} as
\begin{align}\label{eq:U}
&-\nu \widehat{\Delta} {U}+{\ir}l(\lambda-r^2){U}=F_1+\nu \widehat{\Delta} {U}_1+\nu (\widehat{\Delta}-\widehat{\Delta}_1^*) {W}_1:=F_{1,1}.
\end{align}

The following two lemmas are devoted to the estimates of $U$ based on the formulation \eqref{eq:U1} and \eqref{eq:U}.

 \begin{Lemma}\label{lem:U-H2}
It holds that
  \begin{align*}
    &\nu\|\widehat{\Delta} {U}\|_{L^{2}}+|l|\|(\lambda-r^2) {U}\|_{L^{2}}+|\nu/\delta|\|U\|_{1}+|lr_0\delta|\|U\|_{L^2}+|l\delta|\|rU\|_{L^2}\leq C\|F_{1,1}\|_{L^2}.
  \end{align*}
\end{Lemma}

\begin{proof}
It follows from \eqref{eq:U}  that
\begin{align*}
 \nu^2\|\widehat{\Delta} {U}\|_{L^2}^2+2l\nu\mathbf{Im}\langle (\lambda-r^2){U},\widehat{\Delta} {U}\rangle+l^2\|(\lambda-r^2){U}\|_{L^2}^2=\|{F}_{1,1}\|_{L^2}^2,
\end{align*}
from which and the fact that 
\beno
\mathbf{Im}\langle (\lambda-r^2){U},\widehat{\Delta} {U}\rangle=-\mathbf{Im}\langle \partial_r[(\lambda-r^2){U}],\partial_r {U}\rangle=\mathbf{Im}\langle 2r{U},\partial_r {U}\rangle,
\eeno
we infer that
\begin{align*}
\nu^2\|\widehat{\Delta} {U}\|_{L^2}^2+l^2\|(\lambda-r^2){U}\|_{L^2}^2\leq4|l\nu|\|rU\|_{L^2}\|\partial_r {U}\|_{L^2}+\|{F}_{1,1}\|_{L^2}^2.
\end{align*}

Using the fact that for $r>0$,
\ben
r\leq(r^2+r^2_0)/(2r_0)= r_0+(-\lambda+r^2)/(2r_0)\leq r_0+|\lambda-r^2|/(2r_0),\label{eq:kp-est1}
\een
we deduce from Lemma \ref{lem:interp} that 
\begin{align*}
  \|r{U}\|_{L^2}\leq& r_0\|{U}\|_{L^2} +\|(\lambda-r^2){U}\|_{L^2}/(2r_0)\\ \leq& C\big(r_0\|(\lambda-r^2){U}\|_{L^2}\|{U}\|_{1}\big)^{\f12} +(C/r_0)\|(\lambda-r^2){U}\|_{L^2}.
\end{align*}
Then we obtain
\begin{align*}
  &\nu^2\|\widehat{\Delta} {U}\|_{L^2}^2+l^2\|(\lambda-r^2){U}\|_{L^2}^2\\ &\leq C|l\nu|\Big(\big(r_0\|(\lambda-r^2){U}\|_{L^2}\|{U}\|_{1}\big)^{\f12}+\|(\lambda-r^2){U}\|_{L^2}/r_0\Big)\|\partial_r {U}\|_{L^2}+\|{F}_{1,1}\|_{L^2}^2.
\end{align*}
Then by Young's inequality, we obtain
\begin{align*}
  &\nu^2\|\widehat{\Delta} {U}\|_{L^2}^2+l^2\|(\lambda-r^2){U}\|_{L^2}^2 \leq C\big(|r_0 l\nu^2|^{\f23}+|\nu/r_0|^2\big)\|{U}\|_{1}^2+C\|{F}_{1,1}\|_{L^2}^2.
\end{align*}
As $\delta/r_0\leq c_2\leq 1$,  $|\nu/r_0|\leq |\nu/\delta|=|r_0l\nu^2|^{\f13}$ and  then
\begin{align*}
 \nu^2\|\widehat{\Delta} {U}\|_{L^2}^2+l^2\|(\lambda-r^2){U}\|_{L^2}^2 \leq C|r_0 l\nu^2|^{\f23}\|{U}\|_{1}^2+C\|{F}_{1,1}\|_{L^2}^2.
\end{align*}

Taking the real part of the following equation
\begin{align*}
 -\nu\langle\widehat{\Delta}U,U\rangle+{\ir}l\int_{0}^{1}(\lambda-r^2)|U|^2dr=\langle {F}_{1,1},U\rangle,
\end{align*}
we obtain
\ben\label{eq:kp-est2}
\nu\|U\|_{1}^2\leq \|F_{1,1}\|_{L^2}\|U\|_{L^2}.
\een
Hence, we conclude
\begin{align}\label{eq:kp-est3}
   &\nu^2\|\widehat{\Delta} {U}\|_{L^2}^2+|\nu/\delta|^2\|U\|_{1}^2 +l^2\|(\lambda-r^2){U}\|_{L^2}^2\nonumber\\
    &\quad\leq C|lr_0\delta|\|U\|_{L^2}\|F_{1,1}\|_{L^2}+C\|{F}_{1,1}\|_{L^2}^2.
\end{align}
Here we used $|r_0 l\nu^2|^{\f23}= |\nu/\delta|^2=\nu|lr_0\delta|.$

By Lemma \ref{lem:interp} and \eqref{eq:kp-est2}, we get
\begin{align*}
  |lr_0\delta|^2\|U\|_{L^2}^2\leq& C|l\delta|^2r_0\|U\|_{1}\|(\lambda-r^2)U\|_{L^2}\\
  \leq& C|l|\delta^2r_0\nu^{-\f12}\|F_{1,1}\|_{L^2}^{\f12}\|U\|_{L^2}^{\f12} \big(|lr_0\delta|\|U\|_{L^2}\|F_{1,1}\|_{L^2}+\|{F}_{1,1}\|_{L^2}^2\big)^{\f12}\\
  \leq &C\|F_{1,1}\|_{L^2}^{\f12}\big(|lr_0\delta|\|U\|_{L^2}\big)^{\f12}\big(|lr_0\delta|\|U\|_{L^2}\|F_{1,1}\|_{L^2}+\|{F}_{1,1}\|_{L^2}^2\big)^{\f12},
\end{align*}
from which and Young's inequality, we infer that
\begin{align*}
  & |lr_0\delta|\|U\|_{L^2}\leq C\|F_{1,1}\|_{L^2}.
\end{align*}
This along with \eqref{eq:kp-est3} shows  that
\begin{align*}
    &\nu\|\widehat{\Delta} {U}\|_{L^{2}}+|l|\|(\lambda-r^2) {U}\|_{L^{2}}+|\nu/\delta|\|U\|_{1}+|lr_0\delta|\|U\|_{L^2}\leq C\|F_{1,1}\|_{L^2}.
  \end{align*}
Due to $\delta/r_0\leq c_2\leq 1$, we get by \eqref{eq:kp-est1} that 
\beno
\|r{U}\|_{L^2}\leq r_0\|{U}\|_{L^2} +(C/r_0)\|(\lambda-r^2){U}\|_{L^2} \leq C\big(r_0\|{U}\|_{L^2} +\delta^{-1}\|(\lambda-r^2)U\|_{L^2}\big),
\eeno
which gives
\beno
|l\delta|\|rU\|_{L^2}\le C\|F_{1,1}\|_{L^2}.
\eeno

This proves the lemma. 
\end{proof}

\begin{Lemma}\label{lem:UF11}
It holds that
\begin{align*}
  &\nu^2(\|\widehat{\Delta} {U}\|_{L^2}^2+\|\widehat{\Delta} {U}_1\|_{L^2}^2)+\|F_{1,1}\|_{L^2}^2\\ &\leq  C\Big(\nu l\|rU_1\|_{L^2}\| {U}_1\|_{1}+\| (F_1,F_2)\|_{L^2}^2+ c_2^2|\nu/\delta|^2\|(U_1,W_1)\|_{1}^2+\nu l\|(W_1,U_1)\|_{L^2}^2\Big).
\end{align*}
\end{Lemma}
\begin{proof}
Recall that $F_{1,2}=F_1-F_2+\nu(\widehat{\Delta}-\widehat{\Delta}_1)U+\nu(\widehat{\Delta}-\widehat{\Delta}_1^{*})W_1$.
Using the facts that 
\begin{align*}\| (\widehat{\Delta}-\widehat{\Delta}_1) {U}\|_{L^2}&=\left\|\dfrac{2rl^2}{n^2+r^2l^2}\partial_r{U}\right\|_{L^2}\leq \|(l/n)\partial_r{U}\|_{L^2}\leq (l/n)\|{U}\|_{1},\\\| (\widehat{\Delta}-\widehat{\Delta}_1^*) {W_1}\|_{L^2}&=\left\|\dfrac{1}{r}\partial_r\dfrac{2r^2l^2W_1}{n^2+r^2l^2}\right\|_{L^2}\leq \left\|\dfrac{2rl^2}{n^2+r^2l^2}\partial_rW_1\right\|_{L^2}+\left\|\dfrac{4n^2l^2W_1}{(n^2+r^2l^2)^2}\right\|_{L^2}\\
&\leq\|(l/n)\partial_rW_1\|_{L^2}+2\|(l/n)W_1/r\|_{L^2}\leq 3(l/n)\|W_1\|_{1},
\end{align*}
we infer that
\begin{align}\label{eq:kp-F12}
\| {F}_{1,2}\|_{L^2}\leq 3\big(\| (F_1,F_2)\|_{L^2}+|\nu l/n|\|(U_1,W_1)\|_{1}\big).
\end{align}

Recall that $U_1$ satisfies
\begin{align*}
   -\nu\widehat{\Delta}U_1+{\ir}l(\lambda-r^2)U_1&=F_{1,2}+\dfrac{4{\ir}n^2lW}{n^2+r^2l^2}.
\end{align*}
Taking the inner product  with $\Delta U_1$ and taking the real part, we get
\begin{align}\label{eq:kp-est7}
\nu\|\widehat{\Delta} {U}_1\|_{L^2}^2+l\mathbf{Im}\langle (\lambda-r^2){U}_1,\widehat{\Delta} {U}_1\rangle=-\mathbf{Re}\left\langle {F}_{1,2}+\dfrac{4{\ir}n^2l{W}}{n^2+r^2l^2},\widehat{\Delta} {U}_1\right\rangle.
\end{align}
It is easy to see that
\begin{align}\label{eq:kp-est8}
  -\mathbf{Im}\langle (\lambda-r^2){U}_1,\widehat{\Delta} {U}_1\rangle&=\mathbf{Im}\langle \partial_r[(\lambda-r^2){U}_1],\partial_r {U}_1\rangle \nonumber\\ &=\mathbf{Im}\langle -2r{U}_1,\partial_r {U}_1\rangle\leq2\|rU_1\|_{L^2}\|\partial_r {U}_1\|_{L^2}.
\end{align}
Using the facts that
\begin{align*}
      &\left[\widehat{\Delta},\dfrac{1}{n^2+r^2l^2}\right]=-\dfrac{4rl^2}{(n^2+r^2l^2)^2}\partial_r- \dfrac{4l^2}{(n^2+r^2l^2)^2}+\dfrac{8r^2l^4}{(n^2+r^2l^2)^3},
   \end{align*}
   and
\begin{align*}
   \|\widehat{\Delta}W\|_{L^2}^2=&\left\|\dfrac{1}{r}\pa_r(r\partial_rW)-\dfrac{n^2}{r^2}W\right\|_{L^2}^2 -2\mathbf{Re}\left\langle \dfrac{1}{r}\pa_r(r\partial_rW)-\dfrac{n^2}{r^2}W,l^2W\right\rangle+l^4\|W\|_{L^2}^2 \\
   =&\left\|\dfrac{1}{r}(r\partial_rW)-\dfrac{n^2}{r^2}W\right\|_{L^2}^2 + 2l^2\|\partial_rW\|_{L^2}^2+2n^2l^2\left\|\frac{W}{r}\right\|_{L^2}^2+l^4\|W\|_{L^2}^2,
\end{align*}
we infer that
\begin{align*}
   &n^2\bigg\|\left[\widehat{\Delta},\dfrac{1}{n^2+r^2l^2}\right]W\bigg\|_{L^2}\\
&\leq C\bigg(\left\|\dfrac{rn^2l^2}{(n^2+r^2l^2)^2}\partial_rW\right\|_{L^2} +\left\| \dfrac{n^2l^2W}{(n^2+r^2l^2)^2}\right\|_{L^2}+ \left\|\dfrac{r^2n^2l^4W}{(n^2+r^2l^2)^3}\right\|_{L^2}\bigg)\\
&\leq C\bigg(\left\|\dfrac{l}{n}\partial_rW\right\|_{L^2} +\left\| \dfrac{l^2W}{n^2}\right\|_{L^2}\bigg)\leq C|n|^{-1}\|\widehat{\Delta}W\|_{L^2}\leq C\|\widehat{\Delta}W\|_{L^2}.
\end{align*}
Thanks to 
\begin{align*}
& \|W\|_{1}^2\leq -2\mathbf{Re}\langle W,\widehat{\Delta}_1W\rangle\leq 2\|W\|_{L^2}\|\widehat{\Delta}_1W\|_{L^2}\leq 2|l|^{-1}\|W\|_{1}\|W_1\|_{L^2},
\end{align*}
we infer that $\|W\|_{1}\leq  2|l|^{-1}\|W_1\|_{L^2}$, and then 
\beno
\| (\widehat{\Delta}-\widehat{\Delta}_1) {W}\|_{L^2}\leq |l/n|\|{W}\|_{1} \leq2|n|^{-1}\|W_1\|_{L^2}\leq2\|W_1\|_{L^2},
\eeno
which gives
\beno
\|\widehat{\Delta}W\|_{L^2}\leq\| (\widehat{\Delta}-\widehat{\Delta}_1) {W}\|_{L^2}+\| \widehat{\Delta}_1 {W}\|_{L^2}
\leq 2\|W_1\|_{L^2}+\|W_1\|_{L^2}\leq 3\|W_1\|_{L^2}.
\eeno
Thus, we conclude that
\begin{align*}
  \left\|\widehat{\Delta}\dfrac{4{\ir}n^2l{W}}{n^2+r^2l^2}\right\|_{L^2}\leq & \left\|\dfrac{4{\ir}n^2l{\widehat{\Delta}W}}{n^2+r^2l^2}\right\|_{L^2} +4n^2|l|\bigg\|\left[\widehat{\Delta},\dfrac{1}{n^2+r^2l^2}\right]W\bigg\|_{L^2} \\
  \leq &C|l|\|\widehat{\Delta}W\|_{L^2}= C|l|\|W_1\|_{L^2},
\end{align*}
which gives
\begin{align}\label{eq:kp-est9}
  \left|\left\langle\dfrac{4{\ir}n^2l{W}}{n^2+r^2l^2},\widehat{\Delta} {U}_1\right\rangle\right|&=\left|\left\langle \widehat{\Delta}\dfrac{4{\ir}n^2l{W}}{n^2+r^2l^2}, {U}_1\right\rangle\right|\leq\left\|\widehat{\Delta}\dfrac{4{\ir}n^2l{W}}{n^2+r^2l^2}\right\|_{L^2}\|U_1\|_{L^2} \nonumber\\&\leq C|l|\|W_1\|_{L^2}\|U_1\|_{L^2}.
\end{align}

By \eqref{eq:kp-est7}, \eqref{eq:kp-est8} and \eqref{eq:kp-est9}, we get
\begin{align*}
  &\nu\|\widehat{\Delta} {U}_1\|_{L^2}^2\leq2|l|\|rU_1\|_{L^2}\|\partial_r {U}_1\|_{L^2}+\|{F}_{1,2}\|_{L^2}\|\widehat{\Delta} {U}_1\|_{L^2}+C|l|\|W_1\|_{L^2}\|U_1\|_{L^2}.
\end{align*}
Then Young's inequality and \eqref{eq:kp-F12}  show that
\begin{align*}
  &\nu\|\widehat{\Delta} {U}_1\|_{L^2}^2\leq C\big(|l|\|rU_1\|_{L^2}\|\partial_r {U}_1\|_{L^2}+\nu^{-1}\|{F}_{1,2}\|_{L^2}^2+|l|\|W_1\|_{L^2}\|U_1\|_{L^2})\\ &\leq C\big(|l|\|rU_1\|_{L^2}\| {U}_1\|_{1}+\nu^{-1}\| (F_1,F_2)\|_{L^2}^2+\nu |l/n|^2\|(U_1,W_1)\|_{1}^2+|l|\|(W_1,U_1)\|_{L^2}^2\big).
\end{align*}

Recall that $F_{1,1}=F_1+\nu\widehat{\Delta}U_1+\nu(\widehat{\Delta}-\widehat{\Delta}_{1}^{*})W_1,$ and then
\begin{align*}
  \|F_{1,1}\|_{L^2} &\leq \|F_{1}\|_{L^2}+\nu\|\widehat{\Delta}U_1\|_{L^2}+3\nu|l|/|n|\|W_1\|_{1}.
\end{align*}
Thus, we get by Lemma \ref{lem:U-H2} that
\begin{align*}
   &\nu^2\big(\|\widehat{\Delta} {U}\|_{L^2}^2+\|\widehat{\Delta} {U}_1\|_{L^2}^2\big)+\|F_{1,1}\|_{L^2}^2\\ &\leq \nu^2\|\widehat{\Delta} {U}_1\|_{L^2}^2+C\|F_{1,1}\|_{L^2}^2\leq C\big(\|F_{1}\|_{L^2}^2+\nu^2\|\widehat{\Delta}U_1\|_{L^2}^2+|\nu l/n|^2\|W_1\|_{1}^2\big)\\ &\leq C\big(|\nu l|\|rU_1\|_{L^2}\| {U}_1\|_{1}+\| (F_1,F_2)\|_{L^2}^2+|\nu l/n|^2\|(U_1,W_1)\|_{1}^2+|\nu l|\|(W_1,U_1)\|_{L^2}^2\big)\\ &\leq C\big(|\nu l|\|rU_1\|_{L^2}\| {U}_1\|_{1}+\| (F_1,F_2)\|_{L^2}^2+ c_2^2|\nu/\delta|^2\|(U_1,W_1)\|_{1}^2+|\nu l|\|(W_1,U_1)\|_{L^2}^2\big).
\end{align*}
Here we used  $|\delta ln^{-1}|\leq c_2n^{-1}\leq c_2$. 
\end{proof}\smallskip

The following lemmas are devoted to the estimates of $W$ and $W_1$.  We will view the viscous term as a perturbation and rewrite the second equation of \eqref{eq:LNS-WU-a} as
\begin{align}\label{eq:W}
(\lambda-r^2){W}_1+\dfrac{4n^2{W}}{n^2+r^2l^2}=({F}_2+\nu \widehat{\Delta}_1{U})/({\ir}l):=F_{2,2}.
\end{align}

\begin{Lemma}\label{lem:W-L2}
Assume that $\lambda,\delta\in(0,1],\ r_0/\delta\geq \max(|n|,2)$ and $|l|\delta\leq 1$. Let $I=[r_0-\delta,r_0+\delta]\cap(0,1]$. 
Then it holds that
  \begin{align*}
    &\left\|\dfrac{W}{\sqrt{n+rl}}\right\|_{L^{\infty}}^2+\left\|\dfrac{W}{r}\right\|_{L^{2}}^2+
    r_0\delta(n+r_0l)^3 \left\|\dfrac{(\lambda-r^2)^{-1}{W}}{n^2+r^2l^2}\right\|_{L^{2}((0,1]\setminus I)}^2\leq CE,\\
    &\|W_1\|_{L^2}^2\leq C\Big({r_0\delta\|W_1\|_{L^{\infty}(I)}^2+(r_0\delta)^{-2} \|F_{2,2}\|_{L^{2}}^2}
    +n^4(r_0\delta)^{-1}(n+r_0l)^{-3}E\Big),\\
    &r_0\|W_1\|_{L^{\infty}(I)}^2\leq C\|W_1\|_{L^2}\|W_1\|_{1},\\&\|rW_1\|_{L^2}\leq C\big(r_0\|W_1\|_{L^{2}}+r_0^{-1}\|F_{2,2}\|_{L^{2}}\big).
  \end{align*}
  In particular, we have
  \begin{align*}
     r_{0}\|W_1\|_{L^2}+\|rW_1\|_{L^2}&\leq C\Big(\delta^{-1}\|F_{2,2}\|_{L^2}+ r_0\delta\|W_1\|_{1}
     +n^2\delta^{-\f12}r_0^{\f12}(n+r_0l)^{-\f32}E^{\f12}\Big).
  \end{align*}
\end{Lemma}

\begin{proof}
It is obvious that $\left\|\f{W}{r}\right\|_{L^2}^2\leq E$. For $r\in[0,1],$ due to $W(0)=0$, we have
\begin{align*}
   \left|\f{W(r)}{\sqrt{n+rl}}\right|^2&\leq \int_{0}^{r}\partial_s\bigg(\f{|W(s)|^2}{n+sl}\bigg)\mathrm{d}s\leq \int_{0}^{1}\bigg|\partial_r\bigg(\f{|W(r)|^2}{n+rl}\bigg)\bigg|\mathrm{d}r\\
   &\leq \int_{0}^{1}\bigg(\f{2|W\partial_rW|}{n+rl}+ \f{l|W|^2}{(n+rl)^2}\bigg)\mathrm{d}r\leq \int_{0}^{1}\bigg(\f{2r^{\f12}|\partial_rW|}{n+rl}\f{|W|}{r^{\f12}}+ \f{|W|^2}{2nr}\bigg)\mathrm{d}r\\
   &\leq \int_{0}^{1}\bigg(\f{r|\partial_rW|}{(n+rl)^2}+\f{|W|^2}{r}+ \f{|W|^2}{2nr}\bigg)\mathrm{d}r\leq CE,
\end{align*}
which gives $\left\|\f{W}{\sqrt{n+rl}}\right\|_{L^\infty}^2\leq CE.$

Thanks to $(0,1]\setminus I\subseteq \{|r^2-r_0^2|\geq r_0\delta\}$. Using the change of variables $ \tau=r^2-r_0^2,\ s=r_0^2+\tau$, we obtain
\begin{align*}
   &\left\|\f{(\lambda-r^2)^{-1}}{(n^2+r^2l^2)^{\f34}}\right\|_{L^2((0,1]\setminus I)}^2=\int_{(0,1]\setminus I}\f{r(\lambda-r^2)^{-2}}{(n^2+r^2l^2)^{\f32}}\mathrm{d}r\\ &\leq \f{1}{2}\int_{|\tau|\geq r_0\delta,\tau\geq -r_0^2/2}\f{\tau^{-2}}{(n^2+(\tau+r_0^2)l^2)^{\f32}}\mathrm{d}\tau+\f{1}{2}\int_{-r_0^2\leq\tau\leq -r_0^2/2}\f{\tau^{-2}}{(n^2+(\tau+r_0^2)l^2)^{\f32}}\mathrm{d}\tau\\
  & \leq (r_0\delta)^{-1}(n^2+r^2_0l^2/2)^{-\f32}+\f{2}{r_0^4}\int_{0}^{r_0^2/2}\f{\mathrm{d}s}{(n^2+sl^2)^{\f32}}\leq C(r_0\delta)^{-1}(n+r_0l)^{-3},
\end{align*}
here we used
\begin{align*}
   &\f{2}{r_0^4}\int_{0}^{r_0^2/2}\f{\mathrm{d}s}{(n^2+sl^2)^{\f32}}=\f{1}{r_0^4l^2}\left(\f1n-\f1{\sqrt{n^2+r_0^2l^2/2}}\right)\\&=
   \f{1}{2r_0^2}\f1{n\sqrt{n^2+r_0^2l^2/2}}\f1{n+\sqrt{n^2+r_0^2l^2/2}}\leq\f{1}{2r_0^2n(n^2+r_0^2l^2/2)}\\ &\leq Cr_0^{-2}(n+r_0l)^{-2}\leq Cr_0^{-1}l(n+r_0l)^{-3}\leq Cr_0^{-1}\delta^{-1}(n+r_0l)^{-3}.
\end{align*}
Thus, we have
\begin{align*}
   \left\|\dfrac{(\lambda-r^2)^{-1}{W}}{n^2+r^2l^2}\right\|_{L^{2}((0,1]\setminus I)}^2&\leq C\left\|\dfrac{(\lambda-r^2)^{-1}}{(n^2+r^2l^2)^{\f34}}\right\|_{L^{2}((0,1]\setminus I)}^2\left\|\f{W}{\sqrt{n+rl}}\right\|_{L^\infty}^2\\
   &\leq C(r_0\delta)^{-1}(n+r_0l)^{-3}E.
\end{align*}
This proves the first inequality of the lemma.

 Let $\chi_0(r)=\rho(r)/(\lambda-r^2)$, where $\rho(r)=1$ for $|r-r_0|> \delta$, $\rho(r)=0$ for $|r-r_0|< \delta$. Then we have
\begin{align*}
\left\langle \chi_0{W_1}, F_{2,2}\right\rangle=&\left\langle \rho{W_1},{W}_1\right\rangle+\left\langle \chi_0{W_1}, \dfrac{4n^2{W}}{n^2+r^2l^2}
\right\rangle,
\end{align*}
which gives
\begin{align*}
   &\|W_1\|_{L^2}^2= \left\langle (1-\rho){W_1},{W}_1\right\rangle+ \left\langle \rho{W_1},{W}_1\right\rangle\\
   &\leq \|1-\rho\|_{L^1}\|W_1\|_{L^\infty(I)}^2+ \|\chi_0\|_{L^\infty}\|W_1\|_{L^2}\|F_{2,2}\|_{L^2}+\|W_1\|_{L^2}\left\|\f{4n^2\chi_0 W}{n^2+r^2l^2}\right\|_{L^2}\\
   &\leq \int_{\ir}r\mathrm{d}r\|W_1\|_{L^\infty(I)}^2+ (r_0\delta)^{-1}\|W_1\|_{L^2}\|F_{2,2}\|_{L^2}\\
   &\qquad+4n^2\|W_1\|_{L^2} \left\|\dfrac{(\lambda-r^2)^{-1}{W}}{n^2+r^2l^2}\right\|_{L^{2}((0,1]\setminus I)}\\
   &\leq 2r_0\delta\|W_1\|_{L^\infty(I)}^{2}+(r_0\delta)^{-1}\|W_1\|_{L^2}\|F_{2,2}\|_{L^2}+Cn^2(r_0\delta)^{-\f12}(n+r_0l)^{-\f32}\|W_1\|_{L^2}E^{\f12}.
\end{align*}
Here we used $\|\chi_0\|_{L^\infty}\leq \|(r-r_0)^{-1}(r_0+r)^{-1}\|_{L^\infty((0,1]\setminus I)}\leq (r_0\delta)^{-1}.$
Then young's inequality gives
\begin{align*}
  \|W_1\|_{L^2}^2 &\leq C\Big(r_0\delta\|W_1\|_{L^\infty(I)}^{2}+(r_0\delta)^{-2}\|F_{2,2}\|_{L^2}^2 +n^4(r_0\delta)^{-1}(n+r_0l)^{-3}E\Big),
\end{align*}
which gives the second inequality of the lemma.

For $r\in I,$ we have $r_0/r\leq r_0/(r_0-\delta)\leq 2$, and then $r_0\|W_1\|_{L^\infty(I)}^2\leq 2\|r|W_1|^2\|_{L^\infty(I)}\leq 2\|r|W_1|^2\|_{L^\infty}$, where for any $r\in I$, 
\begin{align*}
   r|W_1(r)|^2&=\int_{0}^{r}\partial_s\big(s|W_1(s)|^2\big)\mathrm{d}s\leq \int_{0}^{1}\partial_r\big(r|W_1|^2\big)\mathrm{d}r\\
   &\leq \int_{0}^{1}\big(|W_1|^2 +2r|W_1||\partial_rW_1|\big)\mathrm{d}r
   \leq \|W_1\|_{L^2}\left\|\f{W_1}{r}\right\|_{L^2}+ 2\|W_1\|_{L^2}\|\partial_rW_1\|_{L^2}\\
   &\leq C\|W_1\|_{L^2}\|W_1\|_{1}.
\end{align*}
This gives the third inequality of the lemma.

Thanks to $r^2W_1=r_0^2W_1+\f{4n^2W}{n^2+r^2l^2}-F_{2,2}$ and $\left\langle\f{W}{n^2+r^2l^2},W_1\right\rangle=-E\leq 0$, we obtain
\begin{align*}
   \|rW_1\|_{L^2}^2&=\langle r^2W_1,W_1\rangle= \left\langle r_0^2W_1+\f{4n^2W}{n^2+r^2l^2}-F_{2,2} ,W_1\right\rangle\\
   &=r_0^2\|W_1\|_{L^2}^2+4n^2\left\langle\f{W}{n^2+r^2l^2},W_1\right\rangle -\left\langle F_{2,2} ,W_1\right\rangle\\
   &\leq r_0^2\|W_1\|_{L^2}^2-4n^2E +\| F_{2,2}\|_{L^2}\|W_1\|_{L^2}\\
   &\leq r_0^2\|W_1\|_{L^2}^2+\big(\f{1}{4}r_0^2\|W_1\|_{L^2}^2+r_0^{-2}\|F_{2,2}\|_{L^2}^2\big)\\
   &\leq C\big(r_0^2\|W_1\|_{L^2}^2+r_0^{-2}\|F_{2,2}\|_{L^2}^2\big).
\end{align*}
This gives the fourth inequality. 
\end{proof}

\begin{Lemma}\label{lem:W-lower}
Under the same assumptions as in Lemma \ref{lem:W-L2}, it holds that
\begin{align*}
  &-\left\langle \dfrac{{W}}{n^2+r^2l^2},{W}_1\right\rangle-\left\langle \dfrac{\chi_0{W}}{n^2+r^2l^2}, \dfrac{4n^2{W}}{n^2+r^2l^2}\right\rangle
\geq \frac{E}{3}-
\dfrac{C\|{W}\|_{L^{2}(I)}^2}{r_0\delta(|n|+r_0|l|)},
\end{align*}
where $\chi_0(r)=\rho(r)/(\lambda-r^2)$ with $\rho(r)=1$ for $|r-r_0|> \delta$ and  $\rho(r)=0$ for $|r-r_0|< \delta$.\end{Lemma}

\begin{proof}Thanks to 
\begin{align*}
\int_{r_0-\delta/2}^{r_0}r\mathrm{d}r=(r_0-\delta/4)(\delta/2)>(r_0/2)(\delta/2),
\end{align*}
there exists a point $\tilde{r}\in (r_0-\delta/2,r_0)\subset I$ so that 
\beno
|W(\tilde{r})|\leq 2(r_0\delta)^{-\f12}\|W\|_{L^2}.
\eeno
We fix this $\tilde{r}$ and then  extend $W$ to $r\in(0,2]$ by defining $W(r)=0$ for $r\in[1,2]$. Now we decompose $W(r)=w_1(r)+w_2(r)$ for $r\in(0,2]$, where 
\beno
w_2(r)=W(\tilde{r})\chi_2(r),\quad \chi_2(r)=\xi((r-r_0)/r_1)
\eeno
with $r_1=r_0/\max(|n|,2,|l|r_0)$ and $ \xi\in C_0^{\infty}(\R)$ satisfying $\xi(s)=1$ for $|s|\leq 1/2$ and  $\xi(s)=0$ for $|s|\geq 1$, and $\xi(s)=\xi(-s)$ for $s\in \R.$ Then we have
\beno
 \delta\leq r_1\leq r_0/2,\quad \chi_2(\widetilde{r})=1,\quad w_1(\tilde{r})=0.
\eeno
It holds that 
 \begin{align}\label{eq:kp-est6}
    &-\left\langle \f{\chi_0W}{n^2+r^2l^2},\f{4n^2W}{n^2+r^2l^2}\right\rangle
    =-\left\langle \f{\chi_0w_1}{n^2+r^2l^2},\f{4n^2w_1}{n^2+r^2l^2}\right\rangle\\&\qquad-\left\langle \f{\chi_0w_2}{n^2+r^2l^2},\f{4n^2w_2}{n^2+r^2l^2}\right\rangle-2\mathbf{Re} \left\langle \f{\chi_0w_1}{n^2+r^2l^2},\f{4n^2w_2}{n^2+r^2l^2}\right\rangle.\nonumber
     \end{align}
     
First of all,  we have
\begin{align*}
   &\left\langle\f{\chi_0w_2}{n^2+r^2l^2},\f{4n^2w_2}{n^2+r^2l^2} \right\rangle\\
   &= \int_{0}^{r_0-\delta}\f{4n^2r(\lambda-r^2)^{-1}|w_2|^2}{(n^2+r^2l^2)^2}\mathrm{d}r+ \int_{r_0+\delta}^{2}\f{4n^2r(\lambda-r^2)^{-1}|w_2|^2}{(n^2+r^2l^2)^2}\mathrm{d}r\\
  & = |W(\tilde{r})|^2\int_{\delta}^{r_1}\f{4n^2(r_0-\tau)(\lambda-(r_0-\tau)^2)^{-1}\xi(-\tau/r_1)}{\big( n^2+(\tau-r_0)^2l^2\big)^2}\mathrm{d}\tau\\
   &\quad+|W(\tilde{r})|^2\int_{\delta}^{r_1}\f{4n^2(r_0+\tau) (\lambda-(r_0+\tau)^2)^{-1}\xi(\tau/r_1)}{\big( n^2+(\tau+r_0)^2l^2\big)^2}\mathrm{d}\tau.
\end{align*}
Using the fact that for $\tau\in[\delta, r_1]$,
\begin{align*}
   &\left|\f{(r_0-\tau)(\lambda-(r_0-\tau)^2)^{-1}}{\big(n^2+(\tau-r_0)^2l^2\big)^2}+ \f{(r_0+\tau)(\lambda-(r_0+\tau)^2)^{-1}}{\big(n^2+(\tau+r_0)^2l^2\big)^2}\right|\\&=\frac{1}{\tau}
   \left|\f{(r_0-\tau)(2r_0-\tau)^{-1}}{\big(n^2+(\tau-r_0)^2l^2\big)^2}-\f{(r_0+\tau)(2r_0+\tau)^{-1}}{\big(n^2+(\tau+r_0)^2l^2\big)^2}\right|\\ &\leq 2\left\|\partial_{r}\bigg[\f{r/(r+r_0)}{\big(n^2+r^2l^2\big)^2}\bigg]\right\|_{L^\infty_{r}[r_0-r_1,r_0+r_1]} \leq  \f{C}{r_0(|n|+r_0|l|)^{4}}.
\end{align*}
and $\xi(-\tau/r_1)=\xi(\tau/r_1)$, we deduce that
\begin{align}\label{eq:kp-est4}
   \bigg|\left\langle\f{\chi_0w_2}{n^2+r^2l^2},\f{4n^2w_2}{n^2+r^2l^2} \right\rangle\bigg|\leq \f{Cn^2(r_1-\delta)|W(\tilde{r})|^2}{r_0(|n|+r_0|l|)^{4}}\leq \f{C|W(\tilde{r})|^2}{(|n|+r_0|l|)^2}.
\end{align}

Let
  \begin{align*}
  &\psi(r)=\frac{2r}{\lambda-r^2}-\frac{|n|}{r}\quad \text{for}\ 0<r<\tilde{r},\quad \psi(r)=0\quad \text{for}\ r>\tilde{r}.
\end{align*}
Using $w_1(\tilde{r})=0$ and $\psi(r)=0\ \text{for}\ r\geq \tilde{r}$,  we get
\begin{align*}
   \int_{0}^{2}\f{r|\partial_rw_1+\psi w_1|^2}{n^2+r^2l^2}\mathrm{d}r&= \int_{0}^{2}\f{r|\partial_rw_1|^2+r|\psi w_1|^2}{n^2+r^2l^2}\mathrm{d}r+\int_{0}^{2}\f{r\psi (\partial_r|w_1|^2)}{n^2+r^2l^2}\mathrm{d}r\\
   &=\int_{0}^{2}\f{r|\partial_rw_1|^2+r|\psi w_1|^2}{n^2+r^2l^2}\mathrm{d}r-\int_{0}^{\tilde{r}}\partial_r\bigg(\f{r\psi }{n^2+r^2l^2}\bigg)|w_1|^2\mathrm{d}r.
\end{align*}
Thanks to
\begin{align*}
   \int_{0}^{2}\f{r|\psi w_1|^2}{n^2+r^2l^2}\mathrm{d}r&=\int_{0}^{\tilde{r}}\bigg(\f{4r^3}{(\lambda-r^2)^2}- \f{4r|n|}{\lambda-r^2} +\f{n^2}{r}\bigg)\f{|w_1|^2}{n^2+r^2l^2}\mathrm{d}r,
\end{align*}
and for $r\in(0,\tilde{r}]$ 
\begin{align*}
   \partial_r\bigg(\f{r\psi(r) }{n^2+r^2l^2}\bigg) =\f{4r\lambda}{(\lambda-r^2)^2(n^2+r^2l^2)}-\f{4r^3l^2}{(\lambda-r^2)(n^2+r^2l^2)^2} +\f{2rl^2|n|}{(n^2+r^2l^2)^2},
\end{align*}
 we conclude that
\begin{align*}
   & \int_{0}^{2}\f{r|\psi w_1|^2}{n^2+r^2l^2}dr-\int_{0}^{\tilde{r}}\partial_r\bigg(\f{r\psi }{n^2+r^2l^2}\bigg)|w_1|^2\mathrm{d}r\\
   &= \int_{0}^{\tilde{r}} \bigg(\f{4r^3-4r\lambda}{(\lambda-r^2)^2}- \f{4r|n|}{\lambda-r^2}+\f{4r^3l^2}{(\lambda-r^2)(n^2+r^2l^2)} +\f{n^2}{r}-\f{2rl^2|n|}{n^2+r^2l^2}\bigg)\f{|w_1|^2}{n^2+r^2l^2}\mathrm{d}r\\
   &= \int_{0}^{\tilde{r}} \bigg(-\f{4r(|n|+1)}{(\lambda-r^2)}+\f{4r^3l^2}{(\lambda-r^2)(n^2+r^2l^2)} +\f{n^2}{r}-\f{2rl^2|n|}{n^2+r^2l^2}\bigg)\f{|w_1|^2}{n^2+r^2l^2}\mathrm{d}r\\
   &=  \int_{0}^{\tilde{r}} \bigg(-\f{4rn^2(|n|+1)+4r^3l^2|n|}{(\lambda-r^2)(n^2+r^2l^2)}+\f{n^2}{r} -\f{2rl^2|n|}{n^2+r^2l^2}\bigg)\f{|w_1|^2}{n^2+r^2l^2}\mathrm{d}r\\
   &\leq -\int_{0}^{\tilde{r}} \f{4rn^2(|n|+1)|w_1|^2}{(\lambda-r^2)(n^2+r^2l^2)^2}\mathrm{d}r+ \int_{0}^{\tilde{r}}\f{n^2|w_1|^2}{r(n^2+r^2l^2)}\mathrm{d}r.
\end{align*}
Summing up, we obtain
\begin{align*}
  0\leq &\int_{0}^{2}\f{r|\partial_rw_1+\psi w_1|^2}{n^2+r^2l^2}\mathrm{d}r= \int_{0}^{2}\f{r|\partial_rw_1|^2+r|\psi w_1|^2}{n^2+r^2l^2}dr-\int_{0}^{\tilde{r}}\partial_r\bigg(\f{r\psi }{n^2+r^2l^2}\bigg)|w_1|^2\mathrm{d}r\\
  \leq& \int_{0}^{2}\f{r|\partial_rw_1|^2}{n^2+r^2l^2}\mathrm{d}r-\int_{0}^{\tilde{r}} \f{4rn^2(|n|+1)|w_1|^2}{(\lambda-r^2)(n^2+r^2l^2)^2}\mathrm{d}r+ \int_{0}^{\tilde{r}}\f{n^2|w_1|^2}{r(n^2+r^2l^2)}\mathrm{d}r\\
  \leq& \int_{0}^{2}\f{r|\partial_rw_1|^2}{n^2+r^2l^2}\mathrm{d}r-\int_{0}^{\tilde{r}} \f{4rn^2(|n|+1)|w_1|^2}{(\lambda-r^2)(n^2+r^2l^2)^2}\mathrm{d}r+ \int_{0}^{2}\f{|w_1|^2}{r}\mathrm{d}r,
\end{align*}
which implies 
\begin{align*}
&E_1=\int_0^2\left(\dfrac{r|\partial_rw_1|^2}{n^2+r^2l^2}+\dfrac{|w_1|^2}{r}\right)\mathrm{d}r\geq \int_0^{\tilde{r}}\dfrac{4n^2r|w_1|^2(|n|+1)}{(n^2+r^2l^2)^2(\lambda-r^2)}\mathrm{d}r.
\end{align*}
This implies that
\begin{align}\label{eq:kp-est5}
   &-\left\langle\f{\chi_0w_1}{n^2+r^2l^2},\f{4n^2w_1}{n^2+r^2l^2} \right\rangle \geq -\int_{0}^{\tilde{r}}\f{4n^2r|w_1|^2}{(n^2+r^2l^2)^2(\lambda-r^2)}\mathrm{d}r\geq-\f{E_1}{|n|+1}
\end{align}

Thanks to $0<r_0-\delta<\widetilde{r}<r_0$, if $|r-r_0|>\delta,\ r>0,$  we have
 \begin{align*}
 &|\tilde{r}-r|\leq 2|r_0-r|,\ |\tilde{r}+r|\leq |r_0+r|,
 \end{align*}
 hence, $|\tilde{r}^2-r^2|\leq 2|r_0^2-r^2|=2|\lambda-r^2|$ and then 
 \beno
 |\chi_0(r)|=|\lambda-r^2|^{-1}\leq2|\tilde{r}^2-r^2|^{-1}.
 \eeno
 If $|r-r_0|<\delta,$ then $|\chi_0(r)|=0\leq2|\tilde{r}^2-r^2|^{-1} .$ Then we get by Lemma \ref{lem:hardy-2} that
 \begin{align*}
    &\bigg|\left\langle \f{\chi_0w_1}{n^2+r^2l^2},\f{4n^2w_2}{n^2+r^2l^2}\right\rangle\big| \leq 2\left\|\f{(\tilde{r}^2-r^2)^{-1}w_1}{n^2+r^2l^2}\right\|_{L^2} \left\|\f{4n^2w_2}{n^2+r^2l^2}\right\|_{L^2}\\
    &\leq C\bigg(\f{1}{\tilde{r}^2(n^2+\tilde{r}^2l^2)}\int_{0}^{2}\f{r|\partial_rw_1|^2}{n^2+r^2l^2} +\f{|w_1|^2}{r}\mathrm{d}r\bigg)^{\f12}\left\|\f{4n^2w_2}{n^2+r^2l^2}\right\|_{L^2}\\
    &\leq \f{C(E_1)^{\f12}|W(\tilde{r})|}{r_0(|n|+r_0|l|)}\left\|\f{n^2}{n^2+r^2l^2}\right\|_{L^2[r_0-r_1,r_0+r_1]} \leq \f{C(E_1)^{\f12}\|W\|_{L^2(I)}}{r_0(|n|+r_0|l|)(r_0\delta)^{\f12}}\f{n^2(r_0r_1)^{\f12}}{n^2+r^2_0l^2}\\
   & \leq \f{E_{1}}{24} +\f{Cn^4r_0r_1\|W\|_{L^2(I)}^2}{r_0^2(|n|+r_0|l|)^6r_0\delta}\leq \f{E_{1}}{24} +\f{Cn^4\|W\|_{L^2(I)}^2}{(|n|+r_0|l|)^6r_0\delta}\leq  \f{E_{1}}{24} +\f{C\|W\|_{L^2(I)}^2}{(|n|+r_0|l|)^2r_0\delta}.
 \end{align*}
 This along with \eqref{eq:kp-est6}, \eqref{eq:kp-est4} and \eqref{eq:kp-est5} shows that 
  \begin{align*}
    -\left\langle \f{\chi_0W}{n^2+r^2l^2},\f{4n^2W}{n^2+r^2l^2}\right\rangle
     \geq& -\f{E_{1}}{|n|+1}- \f{C|W(\tilde{r})|^2}{(|n|+r_0|l|)^2} -2\bigg| \left\langle \f{\chi_0w_1}{n^2+r^2l^2},\f{4n^2w_2}{n^2+r^2l^2}\right\rangle\bigg|\\
    \geq& -\f{E_{1}}{|n|+1}- \f{C|W(\tilde{r})|^2}{(|n|+r_0|l|)^2} -\f{E_{1}}{12} -\f{C\|W\|_{L^2(I)}^2}{(|n|+r_0|l|)^2r_0\delta}\\
    \geq & -\f{E_{1}}{2}-\f{E_{1}}{12}-\f{C\|W\|_{L^2(I)}^2}{(|n|+r_0|l|)r_0\delta}\\
     \geq& -\f{7E_{1}}{12}-\f{C\|W\|_{L^2(I)}^2}{r_0\delta(|n|+r_0|l|)}.
 \end{align*}
 Using $|f-g|^2\leq \f{8}{7}|f|^2+8|g|^2$ and the choice of $r_1$, we can obtain
 \begin{align*}
   -\f{7E_1}{12}=&-\f{7}{12}\bigg(\int_{0}^{2}\left(\dfrac{r|\partial_r(W-w_2)|^2}{n^2+r^2l^2} +\dfrac{|W-w_2|^2}{r}\right)\mathrm{d}r\bigg) \\
   \geq& -\f{7}{12}\bigg( \f{8}{7}\int_{0}^{2}\left(\dfrac{r|\partial_rW|^2}{n^2+r^2l^2}+\dfrac{|W|^2}{r}\right)\mathrm{d}r +8\int_{0}^{2}\left(\dfrac{r|\partial_rw_2|^2}{n^2+r^2l^2}+\dfrac{|w_2|^2}{r}\right)\mathrm{d}r  \bigg)\\
   \geq& -\f{2E}{3}-C\int_{r_0-r_1}^{r_0+r_1}\left(\dfrac{r|\partial_rw_2|^2}{n^2+r^2l^2}+\dfrac{|w_2|^2}{r}\right)\mathrm{d}r\\
   \geq& -\f{2E}{3}-Cr_1|W(\tilde{r})|^2\left(\dfrac{r_0}{r_1^2(n^2+r^2_0l^2)} +\dfrac{1}{r_0}\right)\\
   \geq& -\f{2E}{3}-C\dfrac{\|W\|_{L^2(I)}^2}{r_0\delta(|n|+r_0|l|)}.
\end{align*}
Thus, we have
\begin{align*}
  -\left\langle \f{\chi_0W}{n^2+r^2l^2},\f{4n^2W}{n^2+r^2l^2}\right\rangle\geq -\f{2E}{3}-\dfrac{C\|W\|_{L^2(I)}^2}{r_0\delta(|n|+r_0|l|)},
\end{align*}
which gives the lemma due to  $-\left\langle \dfrac{{W}}{n^2+r^2l^2},{W}_1\right\rangle=E$.
\end{proof}

\begin{Lemma}\label{lem:E}
Under the same assumptions as in Lemma \ref{lem:W-L2}, it holds that
  \begin{align*}
  E\leq &Cn^{-4}(r_0\delta)^{-1}(|n|+r_0|l|)^3\big((r_0\delta)^3\|W_1\|_{L^{\infty}(I)}^2+\|F_{2,2}\|_{L^{2}}^2\big).
\end{align*}
\end{Lemma}

\begin{proof}
 Let $\chi_0(r)$ and $\rho$ be as in Lemma \ref{lem:W-lower} and  $I=[r_0-\delta,r_0+\delta]\cap(0,1]$. It follows from \eqref{eq:W} that  
\begin{align*}
\left\langle \dfrac{\chi_0{W}}{n^2+r^2l^2}, F_{2,2}\right\rangle=&\left\langle \dfrac{\rho{W}}{n^2+r^2l^2},{W}_1\right\rangle+\left\langle \dfrac{\chi_0{W}}{n^2+r^2l^2}, \dfrac{4n^2{W}}{n^2+r^2l^2}\right\rangle,
\end{align*}
which gives
\begin{align*}
  &-\left\langle \dfrac{{W}}{n^2+r^2l^2},{W}_1\right\rangle-\left\langle \dfrac{\chi_0{W}}{n^2+r^2l^2}, \dfrac{4n^2{W}}{n^2+r^2l^2}\right\rangle\\
  &=-\left\langle \dfrac{(1-\rho){W}}{n^2+r^2l^2},{W}_1\right\rangle-\left\langle \dfrac{\chi_0{W}}{n^2+r^2l^2}, F_{2,2}\right\rangle
\\ &\leq \left\|\dfrac{{W}}{n^2+r^2l^2}\right\|_{L^{2}(I)} \|W_1\|_{L^{2}(I)}+\left\|\dfrac{\chi_0{W}}{n^2+r^2l^2}\right\|_{L^{2}} \|F_{2,2}\|_{L^{2}},
\end{align*}
from which and Lemma \ref{lem:W-lower}, we infer that 
\begin{align*}
    \f{1}{3}E &\leq -\left\langle \dfrac{{W}}{n^2+r^2l^2},{W}_1\right\rangle-\left\langle \dfrac{\chi_0{W}}{n^2+r^2l^2}, \dfrac{4n^2{W}}{n^2+r^2l^2}\right\rangle + \f{C\|W\|_{L^2(I)}^2}{r_0\delta(|n|+r_0|l|)}\\ \nonumber
    &\leq C\left\|\f{W}{n^2+r^2l^2}\right\|_{L^2(I)}\|W_1\|_{L^2(I)}+\left\|\dfrac{\chi_0{W}}{n^2+r^2l^2}\right\|_{L^{2}} \|F_{2,2}\|_{L^{2}}+ \f{C\|W\|_{L^2(I)}^2}{r_0\delta(|n|+r_0|l|)}.
\end{align*}
By Lemma \ref{lem:W-L2}, we have
\begin{align*}
    \left\|\dfrac{\chi_0{W}}{n^2+r^2l^2}\right\|_{L^{2}}^2=\left\|\dfrac{(\lambda-r^2)^{-1}{W}}{n^2+r^2l^2}\right\|_{L^{2}((0,1]\setminus I)}^2\leq \frac{CE}{r_0\delta(|n|+r_0|l|)^{3}}.
\end{align*}
For $\tilde{r}\in I$, $|\lambda-\widetilde{r}^2 |=|r_0-\tilde{r}||r_0+\tilde{r}|\leq\delta(2r_0+\delta)\leq3r_0\delta$, and hence,
\begin{align*}
    &\left\|\dfrac{4n^2{W}}{n^2+r^2l^2}\right\|_{L^{2}(I)}
    \leq \|(\lambda-r^2)W_1\|_{L^{2}(I)}+\|F_{2,2}\|_{L^{2}}\leq 3r_0\delta\|W_1\|_{L^{2}(I)}+\|F_{2,2}\|_{L^{2}}.
  \end{align*}
 Thanks to 
 \begin{align*}
    &\left\|{n^2+r^2l^2}\right\|_{L^{\infty}(I)}
    \leq n^2+(r_0+\delta)^2l^2\leq n^2+(2r_0)^2l^2\leq C(|n|+r_0|l|)^2,
  \end{align*}
 we get
 \begin{align*}
    \|W\|_{L^{2}(I)}
    &\leq (4n^{2})^{-1}\left\|{n^2+r^2l^2}\right\|_{L^{\infty}(I)}\left\|\dfrac{4n^2{W}}{n^2+r^2l^2}\right\|_{L^{2}(I)}\\&\leq Cn^{-2}(|n|+r_0|l|)^2\big(r_0\delta\|W_1\|_{L^{2}(I)}+|r_0\delta|^{-\f12}\|F_{2,2}\|_{L^{2}}\big).
  \end{align*}

Summing up and using Young's inequality,  we conclude that
\begin{align*}
 E\leq&
C\left\|\f{W}{n^2+r^2l^2}\right\|_{L^2(I)}\|W_1\|_{L^2(I)}+ \f{C\|W\|_{L^2(I)}^2}{r_0\delta(|n|+r_0|l|)} +C\frac{\|F_{2,2}\|_{L^{2}}^2}{r_0\delta(|n|+r_0|l|)^{3}}
\\ \leq& \frac{C}{n^2r_0\delta}\bigg(\left\|\f{n^2W}{n^2+r^2l^2}\right\|_{L^2(I)}^2 +\big((r_0\delta)\|W_1\|_{L^2(I)} \big)^2\bigg)\\& \quad+ \f{C\|W\|_{L^2(I)}^2}{r_0\delta(|n|+r_0|l|)} +C\frac{\|F_{2,2}\|_{L^{2}}^2}{r_0\delta n^2}\\
\leq &C{\|F_{2,2}\|_{L^{2}}^2}/{(r_0\delta n^2)}+Cn^{-2}(r_0\delta)^{-1}\big((r_0\delta)^2\|W_1\|_{L^{2}(I)}^2+\|F_{2,2}\|_{L^{2}}^2\big) \\& \quad+Cn^{-4}(r_0\delta)^{-1}(|n|+r_0|l|)^3\big((r_0\delta)^2\|W_1\|_{L^{2}(I)}^2+\|F_{2,2}\|_{L^{2}}^2\big) \\
\leq& Cn^{-4}(r_0\delta)^{-1}(|n|+r_0|l|)^3\big((r_0\delta)^2\|W_1\|_{L^{2}(I)}^2+\|F_{2,2}\|_{L^{2}}^2\big).
\end{align*}
Here we used $(|n|+r_0|l|)^3\geq n^2$ and $n^{-2}\leq n^{-4}(|n|+r_0|l|)^3.$ This gives the lemma due to 
$\|W_1\|_{L^{2}(I)}^2\leq \|1\|_{L^{1}(I)}\|W_1\|_{L^{\infty}(I)}^2\leq Cr_0\delta \|W_1\|_{L^{\infty}(I)}^2$.
\end{proof}

\begin{Lemma}\label{lem:EW}
Under the same assumptions as in Lemma \ref{lem:W-L2}, it holds that
\begin{align*}
&r_0\delta\|W_1\|_{L^2}+\delta\|rW_1\|_{L^2}\leq  C\big({r_0\delta^2\|W_1\|_{1}+\|F_{2,2}\|_{L^{2}}}\big),\\
&E\leq
Cn^{-4}(r_0\delta)^{-1}(|n|+r_0|l|)^3\big(r_0^2\delta^4\|W_1\|_{1}^2+\|F_{2,2}\|_{L^{2}}^2\big).
\end{align*}
\end{Lemma}

\begin{proof}
By Lemma \ref{lem:W-L2} and Lemma \ref{lem:E}, we get
  \begin{align*}
     \|W_1\|^2_{L^2}&\leq C\big(r_0\delta\|W_1\|_{L^\infty(I)}^2+(r_0\delta)^{-2}\|F_{2,2}\|_{L^2}^2+ n^4(r_0\delta)^{-1}(|n|+r_0|l|)^{-3}E\big)\\
     &\leq C\big(r_0\delta\|W_1\|_{L^\infty(I)}^2+(r_0\delta)^{-2}\|F_{2,2}\|_{L^2}^2\big)\\
     &\leq C\big(\delta\|W_1\|_{L^2}\|W\|_{1}+(r_0\delta)^{-2}\|F_{2,2}\|_{L^2}^2\big).
  \end{align*}
Then  Young's inequality gives
  \begin{align*}
     & r_0\delta\|W_1\|_{L^2}\leq  C\big({r_0\delta^2\|W_1\|_{1}+\|F_{2,2}\|_{L^{2}}}\big).
  \end{align*}
Then by Lemma \ref{lem:W-L2} and $r_0^{-1}\delta\leq C$, we have
  \begin{align*}
     &\delta\|rW_1\|_{L^2}\leq C\delta\big(r_0\|W_1\|_{L^2}+r_0^{-1}\|F_{2,2}\|_{L^2}\big)\leq  C\big({r_0\delta^2\|W_1\|_{1}+\|F_{2,2}\|_{L^{2}}}\big).
  \end{align*}
We get by Lemma \ref{lem:E} that 
  \begin{align*}
     E&\leq Cn^{-4}(r_0\delta)^{-1}(|n|+r_0|l|)^3\big((r_0\delta)^3\|W_1\|_{L^{\infty}(I)}^2+\|F_{2,2}\|_{L^{2}}^2\big)\\
     &\leq Cn^{-4}(r_0\delta)^{-1}(|n|+r_0|l|)^3\big(r_0^2\delta^3\|W_1\|_{1}\|W_1\|_{L^2}+\|F_{2,2}\|_{L^{2}}^2\big)\\
     &\leq
Cn^{-4}(r_0\delta)^{-1}(|n|+r_0|l|)^3\big(r_0^2\delta^4\|W_1\|_{1}^2+\|F_{2,2}\|_{L^{2}}^2\big).
  \end{align*}

This proves the lemma.
\end{proof}\smallskip

Now we are in a position to prove Proposition \ref{prop:res-real-key}. 

\begin{proof}
We denote
\beno
G=\|F_{1,1}\|_{L^2}+|\nu /\delta|\|W_1\|_{1}+\|F_2\|_{L^2}=\|F_{1,1}\|_{L^2}+lr_0\delta^2\|W_1\|_{1}+\|F_2\|_{L^2}.
\eeno
Thanks to
\begin{align*}
&\|\widehat{\Delta}_1{U}\|_{L^2}\leq\| \widehat{\Delta} {U}\|_{L^2}+\| (\widehat{\Delta}-\widehat{\Delta}_1) {U}\|_{L^2}=\| \widehat{\Delta} {U}\|_{L^2}+(l/n)\|{U}\|_{1}\leq 2\| \widehat{\Delta} {U}\|_{L^2},
\end{align*}
and $F_{2,2}=({F}_2+\nu \widehat{\Delta}_1{U})/({\ir}l),$
we get by Lemma \ref{lem:U-H2} that 
\begin{align*}
&|l|\|F_{2,2}\|_{L^2}\leq \|{F}_2\|_{L^2}+\nu \|\widehat{\Delta}_1{U}\|_{L^2}\leq \|{F}_2\|_{L^2}+2\nu \|\widehat{\Delta}{U}\|_{L^2}\leq \|F_2\|_{L^2}+C\|F_{1,1}\|_{L^2}.
\end{align*}
Due to $\nu/\delta=|lr_0|\delta^2$, we get
\begin{align}\label{eq:kp-est11}
  r_{0}\delta^2\|W_1\|_{1}+\|F_{2,2}\|_{L^2}\leq CG/|l|,
\end{align}
from which and Lemma \ref{lem:EW}, we infer that 
 \begin{align}\label{eq:kp-est12}
   E^{\f12}\leq Cn^{-2}|l|^{-1}(r_0\delta)^{-\f12}(|n|+r_0|l|)^{\f32}G.
 \end{align}

 By Lemma \ref{lem:W-L2}, Lemma \ref{lem:U-H2}, \eqref{eq:kp-est11} and \eqref{eq:kp-est12}, we deduce that
\begin{align*}
    &|l|r_0\delta\|(W_1,U_1)\|_{L^2}+|l|\delta\|rU_1\|_{L^2}\leq C\big(|l|r_0\delta\|(W_1,U)\|_{L^2}+|l|\delta\|(rW_1,rU)\|_{L^2}\big)\\
   &\leq C\big(|l|\|F_{2,2}\|_{L^2}+|l|r_0\delta^2\|W_1\|_{1}+n^2|l|\delta^{\f12}r_0^{\f12}(|n|+r_0|l|)^{-\f32}E^{\f12}+\|F_{1,1}\|_{L^2}\big)\\
   &\leq  C\big(G+n^2|l|(r_0\delta)^{\f12}(|n|+r_0|l|)^{-\f32} E^{\f12}\big)\leq CG,
\end{align*}
which gives
\begin{align}\label{eq:kp-est13}
& |l|r_0\delta\|(W_1,U_1)\|_{L^2}+|l|\delta\|rU_1\|_{L^2}\leq CG.
\end{align}

Now  we  estimate $G$.
By Lemma \ref{lem:UF11} and \eqref{eq:kp-est13} and $(lr_0\delta)^2=\nu|l|r_0/\delta $, we get
\begin{align*}
   G^2\leq& C\big(\|F_{1,1}\|_{L^2}^2+|\nu/\delta|^2\|W_1\|_{1}^2+\|F_{2}\|_{L^2}^2\big)\\
   \leq &  C\big((|l\delta|\|rU_1\|_{L^2})(|\nu/\delta|\|U_1\|_{1})+\| (F_1,F_2)\|_{L^2}^2+ |\nu /\delta|^2\|(U_1,W_1)\|_{1}^2+\nu |l|\|(W_1,U_1)\|_{L^2}^2\big)\\
   \leq &  C\big(|\nu/\delta|\|U_1\|_{1}G+\| (F_1,F_2)\|_{L^2}^2+ |\nu/\delta|^2\|(U_1,W_1)\|_{1}^2+(\delta/r_0)G^2\big)\\
    \leq &  C\big(|\nu/\delta|\|(U,W_1)\|_{1}G+\| (F_1,F_2)\|_{L^2}^2+ |\nu/\delta|^2\|(U,W_1)\|_{1}^2+c_2G^2\big).
\end{align*}
Taking $c_2$ sufficiently small,  we get by Young's inequality that
\begin{align}\label{eq:kp-G2}
   & G^2\leq C\big(\| (F_1,F_2)\|_{L^2}^2+ |\nu/\delta|^2\|(U,W_1)\|_{1}^2\big).
\end{align}

Next we estimate  $\|W_1\|_{1}.$ By \eqref{eq:LNS-WU-a}, we have
\begin{align*}
   &-\nu\langle W_1,\widehat{\Delta}^{*}_1W_1\rangle+\nu\langle \widehat{\Delta}_1U,U\rangle+2\nu\langle W_1,\widehat{\Delta}U\rangle-4{\ir}n^2l\left\langle \f{W}{n^2+r^2l^2},U\right\rangle \\
   &=-\langle W_1,F_1\rangle-\langle F_2,U\rangle.
\end{align*}
Thanks to $\|W_1\|_{1}^2\leq -2\mathbf{Re}\langle W_1,\widehat{\Delta}^{*}_1W_1\rangle$ and $|\langle \widehat{\Delta}_1U,U\rangle|\leq C\|U\|_{1}^2$, we have
\begin{align*}
   \nu\|W_1\|_{1}^2&\leq C\big(\nu\|U\|_{1}^2+ \nu\|W_1\|_{1}\|U\|_{1}+ \|(W_1,U)\|_{L^2}\|(F_1,F_2)\|_{L^2}\big)+\bigg|\left\langle \f{Cn^2lW}{n^2+r^2l^2},U\right\rangle \bigg|,
\end{align*}
which gives
\begin{align}\label{eq:kp-W1}
   \nu\|W_1\|_{1}^2\leq C\big(\nu\|U\|_{1}^2+ \|(W_1,U)\|_{L^2}\|(F_1,F_2)\|_{L^2}\big)+\bigg|\left\langle \f{Cn^2lW}{n^2+r^2l^2},U\right\rangle \bigg|.
\end{align}
By Lemma \ref{lem:W-L2}, we have
\begin{align*}
   &\left|\bigg\langle\f{4n^2lW}{n^2+r^2l^2},U\bigg\rangle\right|\leq \left\|\f{4n^2lW}{\sqrt{n+rl}}\right\|_{L^\infty}\left\|\dfrac{U}{(n+rl)^{\f32}}\right\|_{L^1(r_0/2,1)} \\&\qquad+\left\|\f{4n^2lW}{n^2+r^2l^2}\right\|_{L^2(0,r_0/2)}\left\|U\right\|_{L^2(0,r_0/2)}\\&\leq (|n|+r_0|l|/2)^{-\f32}\left\|\f{4n^2lW}{\sqrt{n+rl}}\right\|_{L^\infty}\|U\|_{L^1}+\frac{8|nl|r_0}{|n|+r_0|l|}\|W/r\|_{L^{2}}\|U\|_{L^{2}(0,r_0/2)}\\
   &\leq CE^{\f12}\Big((|n|+r_0|l|)^{-\f32}n^2|l|\|U\|_{L^1}+|nl|r_0(|n|+r_0|l|)^{-1}\|U\|_{L^{2}(0,r_0/2)}\Big).
   \end{align*}
   Here we used the fact that for $0<r<r_0/2<r_0$, 
   \begin{align*}
   &\f{|4n^2lW|}{n^2+r^2l^2}\leq \f{|8nlW|}{|n|+r|l|}=\f{|8nl|r}{|n|+r|l|}|W/r|\leq\f{|8nl|r_0}{|n|+r_0|l|}|W/r|.
\end{align*}
We denote
\begin{align}\label{def:G1}
   &G_1=\|U\|_{L^1}+|n|^{-1}r_0(|n|+r_0|l|)^{\f12}\|U\|_{L^{2}(0,r_0/2)}.
\end{align}
Then  we conclude
\begin{align}\label{eq:kp-nonlocal}
   &\left|\bigg\langle\f{4n^2lW}{n^2+r^2l^2}, U\bigg\rangle\right|\leq CE^{\f12}(|n|+r_0|l|)^{-\f32}n^2|l|G_1.
\end{align}

Thus, we infer from \eqref{eq:kp-nonlocal},  \eqref{eq:kp-W1} and \eqref{eq:kp-est12}  that 
 \begin{align*}
    \nu\|W_1\|_{1}^2\leq& C\big(\nu\|U\|_{1}^2+ \|(W_1,U)\|_{L^2}\|(F_1,F_2)\|_{L^2}\big)+C(r_0\delta)^{-\f12}GG_1,
 \end{align*}
 which implies (using $\nu/\delta^2=|lr_0\delta|$) that 
 \begin{align}\label{eq:kp-W1-2}
   |\nu/\delta|^2\|W_1\|_{1}^2&\leq C|\nu/\delta|^2\|U\|_{1}^2\\ \nonumber&+ C|lr_0\delta|\big(\|(W_1,U)\|_{L^2}\|(F_1,F_2)\|_{L^2}+C(r_0\delta)^{-\f12}GG_1\big).
 \end{align}
 
By \eqref{eq:kp-W1-2}, \eqref{eq:kp-est13} and \eqref{eq:kp-G2},  we get
\begin{align*}
 G^2\leq& C\big(\|(F_1,F_2)\|_{L^2}^2+|\nu/\delta|^2\|U\|_{1}^2\big)\\&+ C|lr_0\delta|\big(\|(W_1,U)\|_{L^2}\|(F_1,F_2)\|_{L^2}+C(r_0\delta)^{-\f12}GG_1\big)\\ \leq &C\big(\|(F_1,F_2)\|_{L^2}^2+|\nu/\delta|^2\|U\|_{1}^2\big)+ C\big(G\|(F_1,F_2)\|_{L^2}+C|lr_0\delta|(r_0\delta)^{-\f12}GG_1\big),
\end{align*}
from which and Young's inequality, we infer  that
\begin{align}\label{eq:G2-2}
   & G^2\leq C\big(\| (F_1,F_1)\|_{L^2}^2+ |\nu/\delta|^2\|U\|_{1}^2+l^2r_0\delta G_1^2\big).
\end{align}

Next we estimate $G_1.$ By Lemma \ref{lem:interp} and  Lemma \ref{lem:U-H2}, we have
\begin{align*}&\|U\|_{L^{1}}\leq C\|U\|_{L^2}^{\f12}\|(\lambda-r^2)U\|_{L^2}^{\f12}\leq Cr_0^{-\f14}\|U\|_{1}^{\f14}\|(\lambda-r^2)U\|_{L^2}^{\f34}\leq Cr_0^{-\f14}|l|^{-\f34}\|U\|_{1}^{\f14}\|F_{1,1}\|_{L^2}^{\f34},\\
   &r_0\|U\|_{L^{2}(0,r_0/2)}\leq Cr_0^{-1}\|(\lambda-r^2)U\|_{L^2}\leq C|r_0l|^{-1}\|F_{1,1}\|_{L^2},
\end{align*}
which along with \eqref{def:G1} gives
\begin{align*}G_1&\leq  Cr_0^{-\f14}|l|^{-\f34}\|U\|_{1}^{\f14}\|F_{1,1}\|_{L^2}^{\f34}+ C|nr_0l|^{-1}(|n|+r_0|l|)^{\f12}\|F_{1,1}\|_{L^2}\\&\leq  Cr_0^{-\f14}|l|^{-\f34}\|U\|_{1}^{\f14}G^{\f34}+ C|nr_0l|^{-1}(|n|+r_0|l|)^{\f12}G,
\end{align*}
Plugging it into \eqref{eq:G2-2}, we get
\begin{align*}
G^2&\leq C\big(\| (F_1,F_2)\|_{L^2}^2+ |\nu/\delta|^2\|U\|_{1}^2+|lr_0\delta^2|^{\f12}\|U\|_{1}^{\f12}G^{\f32}+|n^2r_0|^{-1}\delta(|n|+r_0|l|)G^2\big).
\end{align*}
Since $|n^2r_0|^{-1}\delta(|n|+r_0|l|)\leq\delta/r_0+|\delta l|\leq 2c_2$, taking $c_2$ sufficiently small and noting $\nu/\delta=|lr_0\delta^2|$, we conclude  that
\begin{align*}G^2&\leq C\big(\| (F_1,F_2)\|_{L^2}^2+ |\nu/\delta|^2\|U\|_{1}^2\big),
\end{align*}
which along with \eqref{eq:U-H1} and \eqref{eq:kp-est13} gives
\begin{align*}
    G^2&\leq C\big(\|(F_1,F_2)\|_{L^2}+|lr_0\delta|\|(F_1,F_2)\|_{L^2}\|(U,W_1)\|_{L^2}\big)\\
   &\leq C\big(\|(F_1,F_2)\|_{L^2}^2+\|(F_1,F_2)\|_{L^2}G\big).
\end{align*}
This shows that $G\leq C\|(F_1,F_2)\|_{L^2}.$ By \eqref{eq:kp-est13}  again, we get
\begin{align*}
   lr_0\delta\|(W_1,U)\|_{L^2}&\leq Clr_0\delta\|(W_1,U_1)\|_{L^2}\leq CG\leq C\|(F_1,F_1)\|_{L^2}.
\end{align*}
Then we get by \eqref{eq:kp-est12} that
\begin{align*}
   & E^{\f12}\leq Cn^{-2}|l|^{-1}(r_0\delta)^{-\f12}(|n|+r_0|l|)^{\f32}\|(F_1,F_2)\|_{L^2}.
\end{align*}

This completes the proof of Proposition \ref{prop:res-real-key}.
\end{proof}

\section{The approximate elliptic system}

In this section, we study the following approximate elliptic system
\begin{align}\label{eq:WU-app}
\left\{
\begin{aligned}
&-\nu \widehat{\Delta}_1{U}+{\ir}l(\lambda-r^2)U=0,\quad U(1)=1,\\
&\widehat{\Delta}_1{W}=U,\quad W(1)=0,
\end{aligned}\right.
\end{align}
here $\la\in \C$. We define 
\begin{align}\label{def:A}
& A(s)=|l(\lambda-s^2)/\nu|^{\f12}+|ls/\nu|^{\f13}+|l|+|n/s|,\quad A=A(1),\\
 &A_1(s)=|l(\lambda-s^2)/\nu|^{\f12}+|l|+|n/s|,\quad A_1=A_1(1).
\end{align}

\begin{Proposition}\label{prop:Wa-lower}
There exists a constant $c_4>0$ so that  if $l\lambda_i\leq c_4|\nu n l|^{\f12}$ and $0<\nu<c_4'\min(|l|,1)$, then it holds that
\begin{align*}
     |\partial_r{W}(1)|& \geq C^{-1}A^{-1}.
  \end{align*}
\end{Proposition}

\begin{Proposition}\label{prop:Ua}
There exists a constant $c_5>0$ so that  if $l\lambda_i\leq c_5|\nu n l|^{\f12}$ and $0<\nu<c_5\min(|l|,1)$, then it holds that for $s\in(0,1]$,
\beno
\|U\|^2_{L^2(0,s)}\le Cs|U(s)|^2/A(s).
\eeno
\end{Proposition}
\subsection{Homogeneous elliptic problem}
In this subsection, we solve the following homogeneous elliptic equation
 \begin{align}\label{eq:WU-app-toy}
\left\{
\begin{aligned}
&-\nu \widehat{\Delta}_1{U}+{\ir}l(\lambda-1)U=0,\quad U(1)=1,\\
&\widehat{\Delta}_1{W}=U,\quad W(1)=0.
\end{aligned}\right.
\end{align}

\begin{Proposition}\label{prop:UW-toy}
Assume that $ 2l\lambda_i\leq |l(\lambda-1)|$ or $3l\lambda_i\leq \nu (n^2+l^2)$.  It holds that 
\begin{align*}
&|\partial_rU(1)|\leq CA_1,\quad |{\partial_rW(1)}|\geq C^{-1}A_1^{-1},\\
& \|U\|_{L^2}^2\leq CA_1^{-1},\quad \|(1-r^2)U\|_{L^2}^2\leq CA_1^{-3}.
\end{align*}
\end{Proposition}

We need the following simple lemma.

 \begin{Lemma}\label{lem:UW-U}
   For $s\in(0,1]$, it holds that
   \begin{align*}
      & (n^2+l^2)\int_0^s\frac{r|U|^2}{n^2+r^2l^2}\mathrm{d}r\leq  \int_0^s\left(\frac{r|\partial_rU|^2}{n^2+r^2l^2}+\frac{|U|^2}{r}\right)\mathrm{d}r,\\
      &(n^2+l^2)\int_0^1\left(\frac{r|\partial_r W|^2}{n^2+r^2l^2}+\frac{|W|^2}{r}\right)\mathrm{d}r \leq \int_{0}^{1} \frac{r|U|^2}{n^2+r^2l^2}\mathrm{d}r.
   \end{align*}
 \end{Lemma}
 
\begin{proof}
 Thanks to $r^2(n^2+l^2)\leq n^2+r^2l^2$, we have $\dfrac{r(n^2+l^2)}{n^2+r^2l^2}\leq \dfrac{1}{r}$  and  then
   \begin{align*}
      &(n^2+l^2)\int_0^s\frac{r|U|^2}{n^2+r^2l^2}\mathrm{d}r\leq  \int_0^s\dfrac{|U|^2}{r}\mathrm{d}r\leq    \int_0^s\left(\frac{r|\partial_rU|^2}{n^2+r^2l^2}+\frac{|U|^2}{r}\right)\mathrm{d}r.
   \end{align*}
Thanks to $\int_0^1\left(\frac{r|\partial_rW|^2}{n^2+r^2l^2}+\frac{|W|^2}{r}\right)\mathrm{d}r =-\big\langle U,(n^2+r^2l^2)^{-1}W\big\rangle$, we have
 \begin{align*}
     \int_0^1\left(\frac{r|\partial_rW|^2}{n^2+r^2l^2}+\frac{|W|^2}{r}\right)\mathrm{d}r &\leq \left\|W/r\right\|_{L^2}
     \left\|\dfrac{rU}{n^2+r^2l^2}\right\|_{L^2}\\
     &\leq \dfrac{1}{\sqrt{n^2+l^2}}\left\|\dfrac{U}{\sqrt{n^2+r^2l^2}}\right\|_{L^2} \|W/r\|_{L^2}\\ 
     &\leq \dfrac{1}{2(n^2+l^2)}\left\|\dfrac{U}{\sqrt{n^2+r^2l^2}}\right\|_{L^2}^2 +\|W/r\|_{L^2}^2/2,
   \end{align*}
   here we used $\dfrac{r}{n^2+r^2l^2}\leq \dfrac{1}{\sqrt{n^2+l^2}}\dfrac{1}{\sqrt{n^2+r^2l^2}}$. This shows that 
   \begin{align*}
      &\int_0^1\left(\frac{r|\partial_rW|^2}{n^2+r^2l^2}+\frac{|W|^2}{r}\right)\mathrm{d}r \leq \dfrac{1}{n^2+l^2}\left\|\dfrac{U}{\sqrt{n^2+r^2l^2}}\right\|_{L^2}^2 =\dfrac{1}{n^2+l^2}\int_{0}^{1} \frac{r|U|^2}{n^2+r^2l^2}\mathrm{d}r.
   \end{align*}  

The lemma is proved. 
\end{proof}\smallskip

Now we prove Proposition \ref{prop:UW-toy}.

\begin{proof}
By integration by parts, we get
\begin{align*}
   -\big\langle \widehat{\Delta}_1U,(n^2+r^2l^2)^{-1}U\big\rangle_s&= -   \left\langle\dfrac{n^2+r^2l^2}{r}\bigg(\partial_r\dfrac{r\partial_rU}{n^2+r^2l^2}-\dfrac{U}{r}\bigg), (n^2+r^2l^2)^{-1}U \right\rangle_s\\
   &=\int_0^s\left(\frac{r|\partial_rU|^2}{n^2+r^2l^2}+\frac{|U|^2}{r}\right)\mathrm{d}r -\frac{s\partial_rU(s)\overline{U(s)}}{n^2+s^2l^2},
\end{align*}
which gives
\begin{align*}
0&=\big\langle-\nu \widehat{\Delta}_1{U}+{\ir}l(\lambda-1)U,(n^2+r^2l^2)^{-1}{U}\big\rangle_s\\&
=\int_0^s\nu\left(\frac{r|\partial_rU|^2}{n^2+r^2l^2} +\frac{|U|^2}{r}\right)\mathrm{d}r+
{\ir}l(\lambda-1)\int_0^s\frac{r|U|^2}{n^2+r^2l^2}\mathrm{d}r -\frac{s\nu\partial_rU(s)\overline{U}(s)}{n^2+s^2l^2},
\end{align*}
By Lemma \ref{lem:UW-U} and Lemma \ref{lem:la-l}, we get
\begin{align*}
   &\int_0^s\nu\left(\frac{r|\partial_rU1|^2}{n^2+r^2l^2}+\frac{|U|^2}{r}\right)\mathrm{d}r+
|l(\lambda-1)|\int_0^s\frac{r|U|^2}{n^2+r^2l^2}\mathrm{d}r \\
   &\leq  2\bigg| \int_0^s\nu\left(\frac{r|\partial_rU|^2}{n^2+r^2l^2}+\frac{|U|^2}{r}\right)\mathrm{d}r+
{\ir}l(\lambda-1)\int_0^s\frac{r|U|^2}{n^2+r^2l^2}\mathrm{d}r\bigg|.
\end{align*}
Thus, we deduce that
\begin{align}\label{eq:toy-est1}
&\int_0^s\nu\left(\frac{r|\partial_rU|^2}{n^2+r^2l^2}+\frac{|U|^2}{r}\right)\mathrm{d}r+
|l(\lambda-1)|\int_0^s\frac{r|U|^2}{n^2+r^2l^2}\mathrm{d}r \nonumber\\ &\leq\frac{2s\nu|\partial_rU(s)||U(s)|}{n^2+s^2l^2},
\end{align}
which gives by taking $s=1$ that 
\begin{align}
   2|\partial_rU(1)|\geq & (n^2+l^2)\int_0^1\left(\frac{r|\partial_rU|^2}{n^2+r^2l^2}+\frac{|U|^2}{r}\right)+
|l(\lambda-1)/\nu|\int_0^1\frac{r(n^2+l^2)|U|^2}{n^2+r^2l^2}\mathrm{d}r\nonumber\\
\geq & (n^2+l^2)\|U/r\|_{L^2}^2+
|l(\lambda-1)/\nu|\|U\|_{L^2}^2+(n^2+l^2) \left\|\dfrac{\partial_rU}{\sqrt{n^2+r^2l^2}}\right\|_{L^2}^2\nonumber\\
\geq & A_1^2\|U\|_{L^2}^2 +(n^2+l^2) \left\|\dfrac{\partial_rU}{\sqrt{n^2+r^2l^2}}\right\|_{L^2}^2,\label{eq:toy-est2}
\end{align}
and hence,
\begin{align}\label{eq:toy-est3}
&\|2\partial_rU+U/r\|_{L^2}^2\leq 2(\|2\partial_rU\|_{L^2}^2+\|U/r\|_{L^2}^2)\nonumber\\ &\leq 8(n^2+l^2)
\int_0^1\left(\frac{r|\partial_rU|^2}{n^2+r^2l^2}+\frac{|U|^2}{r}\right)\mathrm{d}r\leq 16|\partial_rU(1)|.
\end{align}

By integration by parts, we have
\begin{align*}
   &-\big\langle\widehat{\Delta}_1{U},(n^2+r^2l^2)^{-1}W\big\rangle =-\left\langle\dfrac{1}{r}\bigg(\partial_r\dfrac{r\partial_rU}{n^2+r^2l^2}-\dfrac{U}{r}\bigg), W \right\rangle\\
   &=\big\langle (n^2+r^2l^2)^{-1}\partial_rU,\partial_rW\big\rangle+\left\langle U, \dfrac{W}{r^2} \right\rangle\\
   &=-   \left\langle U,\dfrac{1}{r}\bigg(\partial_r\dfrac{r\partial_r W}{n^2+r^2l^2}\bigg)\right\rangle +\bigg(\dfrac{rU(r)\overline{\partial_rW(r)}}{n^2+r^2l^2}\bigg)\bigg|^{1}_{0}+\left\langle U, \dfrac{W}{r^2} \right\rangle\\
   &=-\big\langle (n^2+r^2l^2)^{-1}U,\widehat{\Delta}_1W\big\rangle+ (n^2+l^2)^{-1}\overline{\partial_rW(1)}\\ 
   &=-\int_{0}^{1}\dfrac{r|U|^2}{n^2+r^2l^2}\mathrm{d}r+ (n^2+l^2)^{-1}\overline{\partial_rW(1)},
\end{align*}
which gives 
\begin{align*}0=&\big\langle-\nu \widehat{\Delta}_1{U}+{\ir}l(\lambda-1)U,-(n^2+r^2l^2)^{-1}{W}\big\rangle\\
=&\int_0^1\frac{\nu r|U|^2}{n^2+r^2l^2}\mathrm{d}r-\nu(n^2+l^2)^{-1}\overline{\partial_rW(1)}-{\ir}l(\lambda-1)\langle\widehat{\Delta}_1{W},(n^2+r^2l^2)^{-1}{W}\rangle\\
=&\int_0^1\frac{\nu r|U|^2}{n^2+r^2l^2}\mathrm{d}r-\frac{\nu\overline{\partial_rW(1)}}{n^2+l^2} +{\ir}l(\lambda-1)\int_0^1\left(\frac{r|\partial_rW|^2}{n^2+r^2l^2}+\frac{|W|^2}{r}\right)\mathrm{d}r.
\end{align*}
By Lemma \ref{lem:UW-U} and Lemma \ref{lem:la-l} again, we get
\begin{align*}
   &\int_0^1\frac{\nu r|U|^2}{n^2+r^2l^2}\mathrm{d}r+
|l(\lambda-1)|\int_0^1\left(\frac{r|\partial_rW|^2}{n^2+r^2l^2}+\frac{|W|^2}{r}\right)\mathrm{d}r\\
  &\leq 2\bigg| \int_0^1\frac{\nu r|U|^2}{n^2+r^2l^2}\mathrm{d}r +{\ir}l(\lambda-1)\int_0^1\left(\frac{r|\partial_rW|^2}{n^2+r^2l^2}+\frac{|W|^2}{r}\right)\mathrm{d}r\bigg|.
\end{align*}
Then we infer that 
\begin{align*}
  &\int_0^1\frac{\nu r|U|^2}{n^2+r^2l^2}\mathrm{d}r+
|l(\lambda-1)|\int_0^1\left(\frac{r|\partial_rW|^2}{n^2+r^2l^2}+\frac{|W|^2}{r}\right)\mathrm{d}r\leq \frac{2\nu|{\partial_rW(1)}|}{n^2+l^2}.
\end{align*}
This gives
\begin{align}\label{eq:toy-est4}
 2|\partial_rW(1)|\geq (n^2+l^2) \int_0^1\frac{ r|U|^2}{n^2+r^2l^2}\mathrm{d}r \geq \|U\|_{L^2}^2,
\end{align}

Using the fact that 
\begin{align*}
&1=|U(1)|^2\leq \|r^{-1}\partial_r(rU^2)\|_{L^1}=\|(2\partial_rU+U/r)U\|_{L^1}\leq\|2\partial_rU+U/r\|_{L^2}\|U\|_{L^2},
\end{align*}
we deduce from \eqref{eq:toy-est3} and \eqref{eq:toy-est4}  that
\begin{align}\label{eq:toy-est5}
&1\leq\|2\partial_rU+U/r\|_{L^2}^2\|U\|_{L^2}^2\leq 32|{\partial_rW(1)}||\partial_rU(1)|.
\end{align} 

We get by integration by parts that 
  \begin{align*}
     & \mathbf{Re}\big(\langle \widehat{\Delta}_1U,r\partial_rU\rangle\big)=\mathbf{Re}\left\langle \dfrac{1}{r}\partial_r(r\partial_rU)-\dfrac{n^2+r^2l^2}{r^2}U -\dfrac{2rl^2\partial_rU}{n^2+r^2l^2},r\partial_rU\right\rangle\\
     &=\dfrac{1}{2}\left\langle \dfrac{1}{r}\partial_r(|r\partial_rU|^2)-\dfrac{n^2+r^2l^2}{r}\partial_r(|U|^2) -\dfrac{4r^2l^2|\partial_rU|^2}{n^2+r^2l^2},1\right\rangle\\
     &=\dfrac{1}{2}\Big(|\partial_rU(1)|^2-(n^2+l^2)|U(1)|^2+\big\langle |U|^2,\dfrac{1}{r}\partial_r(n^2+r^2l^2)\big\rangle\\
     &\qquad-4\left\|\dfrac{rl\partial_rU}{\sqrt{n^2+r^2l^2}}\right\|_{L^2}^2 \Big)\\
     &=|\partial_rU(1)|^2/2-(n^2+l^2)/2+l^2\|U\|_{L^2}^2-2 \left\|\dfrac{rl\partial_rU}{\sqrt{n^2+r^2l^2}}\right\|_{L^2}^2\\
     &\geq |\partial_rU(1)|^2/2-(n^2+l^2)/2-2l^2\left\|\dfrac{\partial_rU}{\sqrt{n^2+r^2l^2}}\right\|_{L^2}^2,
  \end{align*}
which gives
  \begin{align*}
     &|\partial_rU(1)|^2\leq C\Big(\mathbf{Re}\big(\langle \widehat{\Delta}_1U,r\partial_rU\rangle\big)+A_1^2+l^2\left\|\dfrac{\partial_rU}{\sqrt{n^2+r^2l^2}}\right\|_{L^2}^2\Big).
  \end{align*}

Using the equation \eqref{eq:WU-app-toy}, we have
\begin{align*}
 \mathbf{Re}\big\langle \widehat{\Delta}_1U,r\partial_rU\big\rangle&=\mathbf{Re}\big(\langle {\ir}l(\lambda-1)U/\nu,r\partial_rU\rangle\big)\\
     &\leq |l(\lambda-1)/\nu|\left\|r\sqrt{n^2+r^2l^2}U\right\|_{L^2}
     \left\|\dfrac{\partial_rU}{\sqrt{n^2+r^2l^2}}\right\|_{L^2}\\ &\leq A_1^{2}(n^2+l^2)^{\f12}\|U\|_{L^2} \left\|\dfrac{\partial_rU}{\sqrt{n^2+r^2l^2}}\right\|_{L^2}.
  \end{align*}
Thus, we conclude that 
  \begin{align*}
  |\partial_rU(1)|^2&\leq C\Big( A_1^2(n^2+l^2)^{\f12}\|U\|_{L^2} \left\|\dfrac{\partial_rU}{\sqrt{n^2+r^2l^2}}\right\|_{L^2}+A_1^2+l^2\left\|\dfrac{\partial_rU}{\sqrt{n^2+r^2l^2}}\right\|_{L^2}^2\Big),
  \end{align*}
which along with \eqref{eq:toy-est2} gives
 \begin{align*}
     &|\partial_rU(1)|^2\leq C \big(A_1|\partial_rU(1)|+A_1^2+|\partial_rU(1)|\big).
  \end{align*}
This shows that
\beno  
|\partial_rU(1)|\leq C(A_1+1)\leq CA_1.
\eeno
This along  with \eqref{eq:toy-est5} and \eqref{eq:toy-est2} shows  that 
\begin{align*}
       &|\partial_rW(1)|\geq C^{-1}|\partial_rU(1)|^{-1}\geq C^{-1}A_1^{-1},\\
       &\|U\|_{L^2}^2\leq 2A_1^{-2}|\partial_rU(1)|\leq CA_1^{-1}.
\end{align*}

Next let us estimate $\|(1-r^2)U\|_{L^2}$. We denote
\begin{align*}
&\Psi_1(s)=\int_0^s\nu\left(\frac{r|\partial_rU|^2}{n^2+r^2l^2}+\frac{|U|^2}{r}\right)\mathrm{d}r+
|l(\lambda-1)|\int_0^s\frac{r|U|^2}{n^2+r^2l^2}\mathrm{d}r.
\end{align*}
It follows from  \eqref{eq:toy-est1} that 
\begin{align*}
   \Psi_1(r)&\leq\frac{2\nu r|\partial_rU||U|}{n^2+r^2l^2}\leq A_1^{-1}\Big(\dfrac{\nu r|\partial_rU|^2}{n^2+r^2l^2}+A_1^2\dfrac{\nu r| U|^2}{n^2+r^2l^2}\Big).
\end{align*}
Notice that 
\begin{align*}
   &\dfrac{\nu r|\partial_rU|^2}{n^2+r^2l^2}+A_1^2\dfrac{\nu r|U|^2}{n^2+r^2l^2}\\
   &= \dfrac{\nu r|\partial_rU|^2}{n^2+r^2l^2}+(n^2+l^2)\dfrac{\nu r|U|^2}{n^2+r^2l^2}+|l(\lambda-1)| \dfrac{r|U|^2}{n^2+r^2l^2}\\
   &\leq \dfrac{\nu r|\partial_rU|^2}{n^2+r^2l^2}+\dfrac{\nu |U|^2}{r}+|l(\lambda-1)| \dfrac{r|U|^2}{n^2+r^2l^2}=\partial_r\Psi_1(r).
\end{align*}
Then we infer that
\beno
\Psi_1(r)\leq A_1^{-1}\partial_r\Psi_1(r),
\eeno
which implies that 
\beno
\Psi_1(r)\leq\Psi_r(1)\mathrm{e}^{-(1-r)A_1}.
\eeno
We have by \eqref{eq:toy-est1} that 
\begin{align*}
& \Psi_1(1)\leq 2\nu|\partial_rU(1)|/(n^2+l^2)\leq C\nu A_1/(n^2+l^2).
\end{align*}
Thus, we obtain
\beno
\Psi_1(r)\leq \Psi_1(1)\mathrm{e}^{-(1-r)A_1}\leq C\nu A_1\mathrm{e}^{-(1-r)A_1}/(n^2+l^2).
\eeno
On the other hand, we have
\begin{align*}
   \Psi_1(s)&\geq \int_0^s\left(\frac{\nu|U|^2}{r}+
|l(\lambda-1)|\frac{r|U|^2}{n^2+r^2l^2}\right)\mathrm{d}r\\
&\geq \int_0^s\left(\nu r|U|^2+
|l(\lambda-1)|\frac{r|U|^2}{n^2+l^2}\right)\mathrm{d}r= \nu A_1^2 (n^2+l^2)^{-1}\|U\|_{L^2(0,s)}^2.
\end{align*}
Then we obtain
\begin{align*}
   &\|U\|_{L^2(0,s)}^2\leq \nu^{-1}A_1^{-2}(n^2+l^2)\Psi_1(s)\leq CA_1^{-1}\mathrm{e}^{-(1-s)A_1},
\end{align*}
which implies that 
\begin{align*}
\|(1-r^2)U\|_{L^2}^2\leq& C\|(1-r)U\|_{L^2}^2 =C \int_{0}^{1}2(1-s)\|U\|_{L^2(0,s)}^2\mathrm{d}s\\
\leq & C\int_{0}^{1}(1-s)A_1^{-1}\mathrm{e}^{-(1-s)A_1} \leq CA_1^{-3}.
\end{align*}

This completes the proof of the proposition.
\end{proof}

The proof of the following proposition is almost the same. Here we omit the details. 

\begin{Proposition}\label{prop:UW-toy-v2}
If ${\lambda}\in\C$ satisfies $ l{\lambda}_i\leq \max(|l({\lambda}-s^2)|/2,\nu (n^2/s^2+l^2)/3)$ and $|\nu /(sl)|^{1/3}A_1(s)\geq 1.$ 
Let $U$ solve
\begin{align*}
-\nu \widehat{\Delta}_1{U}+{\ir}l(\widetilde{\lambda}-s^2)U=0\ \text{in}\ (0,s),\quad U(s)=1.
\end{align*}
Then it holds that 
\begin{align*}&
 \|U\|_{L^2(0,s)}^2\leq C(A_1(s)/s)^{-1},\quad \|(s^2-r^2)U\|_{L^2}^2\leq C(A_1(s)/s)^{-3}.
\end{align*}
\end{Proposition}

\subsection{Inhomogeneous elliptic problem}

\begin{Proposition}\label{prop:ellip-inhom-v1}
Let $\lambda\in\R$, $s\in(0,1],$ and $U\in H^2(0,s)$ solve
\begin{align*}&-\nu \widehat{\Delta}{U}+{\ir}l(\lambda-r^2)U=F,\quad U(s)=0.
\end{align*}
It holds that 
\begin{align*}
&\big(\nu l^2+|\nu nl|^{\f12}+|\nu\lambda l^2|^{\f13}\big)\|U\|_{L^2(0,s)}+|l|\|(\lambda-r^2)U\|_{L^2}
\\&\qquad+\nu^{\f16}|l|^{\f56}\|rU\|_{L^1(0,s)}\leq C\|F\|_{L^2}.
\end{align*}\
\end{Proposition}

\begin{proof}
For simplicity, here we write $L^2$ as $L^2(0,s)$, $\langle f,g\rangle=\int_{0}^{s}f\overline{g}r\mathrm{d}r$ and $\|f\|_{1}^2=\|\partial_rf\|_{L^2(0,s)}^2+n^2\|f/r\|_{L^2(0,s)}^2+l^2\|f\|_{L^2(0,s)}^2$. First of all, we have
\begin{align*}
   &\nu\|U\|_{1}^2+{\ir}l\big\langle(\lambda-r^2)U,U\big\rangle =\langle F,U\rangle,
\end{align*}
which gives
\begin{align}\label{eq:eih-est1}
  &\nu\|U\|_{1}^2+|l|\big|\langle(\lambda-r^2)U,U\rangle\big| \leq C\|U\|_{L^2}\|F\|_{L^2}.
\end{align}
Since $\|U\|_{1}^2\geq l^2\|U\|_{L^2}^2,$ we deduce that $\nu l^2\|U\|_{L^2}\leq C\|F\|_{L^2}.$

On the other hand, we have
\begin{align*}
   &\nu^2\|\widehat{\Delta}U\|_{L^2}^2+2l\nu\mathbf{Im}\big\langle (\lambda-r^2)U, \widehat{\Delta}U\big\rangle +l^2\|(\lambda-r^2)U\|_{L^2}^2 =\|F\|_{L^2}^2,
\end{align*}
which gives
\begin{align}\label{eq:eih-est2}
\nu^2\|\widehat{\Delta}U\|_{L^2}^2 +l^2\|(\lambda-r^2)U\|_{L^2}^2 =4|l|\nu\|r U\|_{L^2}\|\partial_rU\|_{L^2}+\|F\|_{L^2}^2.
\end{align}
Here we used  
\beno
\mathbf{Im}\big\langle(\lambda-r^2)U,\widehat{\Delta}U\big\rangle=-\mathbf{Im}\big\langle\partial_r[(\lambda-r^2)U],\partial_rU\big\rangle=\mathbf{Im}\big\langle 2rU,\partial_rU\big\rangle.
\eeno

Next we consider two cases. \smallskip

\no{\bf Case 1.} $\lambda\le 0$. 

By \eqref{eq:eih-est1}, we have
\beno
|l|\|\sqrt{r^2-\lambda}U\|_{L^2}^2\leq C\|F\|_{L^2}\|U\|_{L^2}.
\eeno
Thanks to $r\leq \sqrt{|\lambda|+r^2}(|\lambda|+1)^{-\f12}$,  we get by \eqref{eq:eih-est1} that 
  \begin{align*}
     \|U\|_{L^2}^2&\leq \|rU\|_{L^2}\|U/r\|_{L^2}\leq (|\lambda|+1)^{-\f12} \left\|\sqrt{r^2-\lambda}U\right\|_{L^2}|n|^{-1}\|U\|_{1}\\
     &\leq (|\lambda|+1)^{-\f12}|l|^{-\f12}|n|^{-1}\nu^{-\f12}\|F\|_{L^2}\|U\|_{L^2},
  \end{align*}
which gives
\begin{align*}&
|\nu nl|^{\f12}\|U\|_{L^2}\leq |\nu n^2l(|\lambda|+1)|^{\f12}\|U\|_{L^2}\leq C\|F\|_{L^2}.
\end{align*}
This along with \eqref{eq:eih-est1} and \eqref{eq:eih-est2}  gives
 \begin{align*}
     l^2\|(\lambda-r^2)U\|_{L^2}^2\leq & 4|l|\nu\|rU\|_{L^2}\|\partial_rU\|_{L^2}+\|F\|_{L^2}\leq 4|l|\nu\|\sqrt{r^2-\lambda}U\|_{L^2}\|U\|_{1}+\|F\|_{L^2}\\
     \leq &C|l\nu|^{\f12}\|F\|\|U\|_{L^2}+ \|F\|_{L^2}^2\\
     \leq & C|\nu l|^{\f12}\|F\|_{L^2}\big(|\nu nl|^{-\f12}\|F\|_{L^2}\big)+ C\|F\|_{L^2}^2\leq C\|F\|_{L^2}^2.
  \end{align*} 
Thus, we obtain
\beno
|l\lambda|\|U\|_{L^2}+|l|\|(\lambda-r^2)U\|_{L^2}\leq C\|F\|_{L^2},
\eeno
and then
  \begin{align*}
     &|\nu\lambda l^2|^{\f13}\|U\|_{L^2}\leq \big(|\nu n^2l|^{\f12}\big)^{\f23}|l\lambda|^{\f13}\|U\|_{L^2}\leq C( |\nu n^2l|^{\f12}+|l\lambda|)\|U\|_{L^2}\leq C\|F\|_{L^2}.
  \end{align*}
  
In summary,  we conclude that 
\beno
(\nu l^2+|\nu nl|^{\f12}+|\nu\lambda l^2|^{\f13})\|U\|_{L^2}+|l|\|(\lambda-r^2)U\|_{L^2}\leq C\|F\|_{L^2}.
\eeno
  
  \no{\bf Case 2.} $\lambda>0$ and $\delta/r_0\leq 1$, where  $r_0=\lambda^\f12,  \delta=|\nu/(r_0l)|^{\f13}$.\smallskip
  
Using the fact that $r\leq(r^2+r^2_0)/(2r_0)= r_0+(\lambda-r^2)/(2r_0)\leq r_0+|\lambda-r^2|/(2r_0)$ for $r>0$, we get by Lemma \ref{lem:interp}  that 
\begin{align*}
  \|r{U}\|_{L^2}\leq& r_0\|{U}\|_{L^2} +\|(\lambda-r^2){U}\|_{L^2}/(2r_0)\\ \leq& C\big(r_0\|(\lambda-r^2){U}\|_{L^2}\|{U}\|_{1}\big)^{\f12} +(C/r_0)\|(\lambda-r^2){U}\|_{L^2}.
\end{align*}
By \eqref{eq:eih-est2}, we have
\begin{align*}
  &\nu^2\|\widehat{\Delta} {U}\|_{L^2}^2+l^2\|(\lambda-r^2){U}\|_{L^2}^2\\ &\leq Cl\nu\Big[(r_0\|(\lambda-r^2){U}\|_{L^2}\|{U}\|_{1})^{\f12}+\|(\lambda-r^2){U}\|_{L^2}/r_0\Big]
  \|\partial_r {U}\|_{L^2}+\|{F}\|_{L^2}^2.
\end{align*}
Then Young's inequality gives
\begin{align*}
  &\nu^2\|\widehat{\Delta} {U}\|_{L^2}^2+l^2\|(\lambda-r^2){U}\|_{L^2}^2 \leq C\big(|r_0 l\nu^2|^{\f23}+|\nu/r_0|^2\big)\|U\|_{1}^2+C\|{F}\|_{L^2}^2.
\end{align*}
Due to $\delta/r_0\leq 1$, we have $|\nu/r_0|\leq |\nu/\delta|=|r_0l\nu^2|^{\f13}$, and then
\begin{align*}
  &\nu^2\|\widehat{\Delta} {U}\|_{L^2}^2+l^2\|(\lambda-r^2){U}\|_{L^2}^2 \leq C|r_0 l\nu^2|^{\f23}\|{U}\|_{1}^2+C\|{F}\|_{L^2}^2.
\end{align*}
By \eqref{eq:eih-est1}, we have $\nu\|U\|_{1}^2\leq C\|F\|_{L^2}\|U\|_{L^2}$, hence, 
\begin{align}\label{eq:eih-est3}
   &\nu^2\|\widehat{\Delta} {U}\|_{L^2}^2+|\nu/\delta|^2\|U\|_{1}^2 +l^2\|(\lambda-r^2){U}\|_{L^2}^2\leq C|lr_0\delta|\|U\|_{L^2}\|F\|_{L^2}+C\|{F}\|_{L^2}^2.
\end{align}
Here we used $|r_0 l\nu^2|^{\f23}= |\nu/\delta|^2=\nu|lr_0\delta|.$

By Lemma \ref{lem:interp}, \eqref{eq:eih-est1} and \eqref{eq:eih-est3}, we get
\begin{align*}
  |lr_0\delta|^2\|U\|_{L^2}^2&\leq C|l\delta|^2r_0\|U\|_{1}\|(\lambda-r^2)U\|_{L^2}\\
  &\leq C|l|\delta^2r_0\nu^{-\f12}\|F\|_{L^2}^{\f12}\|U\|_{L^2}^{\f12} \big(|lr_0\delta|\|U\|_{L^2}\|F\|_{L^2}+\|{F}\|_{L^2}^2\big)^{\f12}\\
  &\leq C\|F\|_{L^2}^{\f12}\big(|lr_0\delta|\|U\|_{L^2}\big)^{\f12} \big(|lr_0\delta|\|U\|_{L^2}\|F\|_{L^2}+\|{F}\|_{L^2}^2\big)^{\f12}.
\end{align*}
Then Young's inequality gives
\begin{align*}
  |lr_0\delta|\|U\|_{L^2}\leq C\|F\|_{L^2},
\end{align*}
which along with \eqref{eq:eih-est3} gives
\begin{align*}
    &\nu\|\widehat{\Delta} {U}\|_{L^{2}}+|l|\|(\lambda-r^2) {U}\|_{L^{2}}+|\nu/\delta|\|U\|_{1}+|lr_0\delta|\|U\|_{L^2}\leq C\|F\|_{L^2}.
  \end{align*}
Due to  $\delta/r_0\leq 1$, we have
\beno
\|r{U}\|_{L^2}\leq r_0\|{U}\|_{L^2} +(C/r_0)\|(\lambda-r^2){U}\|_{L^2} \leq C\big(r_0\|{U}\|_{L^2} +\delta^{-1}\|(\lambda-r^2)U\|_{L^2}\big)\leq C|l\delta|^{-1}\|F\|_{L^2}.
\eeno
 Thus, we obtain
\begin{align*}
   \|U\|_{L^2}&\leq \|rU\|_{L^2}^{\f12}\|U/r\|_{L^2}^{\f12}\leq C|n|^{-\f12}\|rU\|_{L^2}^{\f12}\|U\|_{1}^{\f12}\\ &\leq C|n|^{-\f12}\big(|l\delta|^{-1}\|F\|_{L^2}\big)^{\f12}\big(|\nu/\delta|^{-1}\|F\|_{L^2}\big)^{\f12} \leq C|\nu nl|^{-\f12}\|F\|_{L^2}.
\end{align*}

Thanks to $|lr_0\delta|=|\nu\lambda l^2|^{\f13}$,  we finally conclude that
\begin{align*}
   &(\nu l^2+|\nu nl|^{\f12}+|\nu \lambda l^2|^{\f13})\|U\|_{L^2}+|l|\|(\lambda-r^2)U\|_{L^2}+ |l\delta|\|rU\|_{L^2}\leq C\|F\|_{L^2}.
\end{align*}

\no{\bf Case 3.}  $\lambda>0$ and $\delta/r_0\ge 1$.

We rewrite the equation of $U$ as
\begin{align*}
   &-\nu\widehat{\Delta}U+{\ir}l(0-r^2)U=-{\ir}l\lambda U+F,\quad U(s)=0.
\end{align*}
Then by {\bf Case 1} with $\la=0$, we get
\begin{align}\label{eq:eih-est4}
   &\|r^2U\|_{L^2}\leq C|l|^{-1}\|-{\ir}l\lambda U+F\|_{L^2}\leq C|\lambda|\|U\|_{L^2}+C|l|^{-1}\|F\|_{L^2}.
\end{align}
By \eqref{eq:eih-est1}, we have $\nu\|U\|_{1}^2\leq C\|F\|_{L^2}\|U\|_{L^2}$, and then
\begin{align*}
   \|U\|_{L^2}&\leq \|r^2U\|_{L^2}^{\f13}\|U/r\|_{L^2}^{\f23}\leq |n|^{-\f23}\|r^2U\|_{L^2}^{\f13}\|U\|_{1}^{\f23}\\
   &\leq C|n|^{-\f23}\big(|\lambda|\|U\|_{L^2}+|l|^{-1}\|F\|_{L^2}\big)^{\f13}\big( \nu^{-\f12}\|F\|_{L^2}^{\f12}\|U\|_{L^2}^{\f12}\big)^{\f23}\\
   &= C\big(\nu^{-\f13}|n|^{-\f23}|\lambda|^{\f13}\|U\|_{L^2}^{\f23}\|F\|_{L^2}^{\f13} +\nu^{-\f13}|n|^{-\f23}|l|^{-\f13}\|F\|_{L^2}^{\f23}\|U\|_{L^2}^{\f13}\big),
\end{align*}
which gives
\begin{align*}
   &\|U\|_{L^2}\leq C\big(\nu^{-1}\lambda|n|^{-2}+ |\nu n^2l|^{-\f12}\big)\|F\|_{L^2}.
\end{align*}
Thanks to $\delta/r_0\leq 1$, we have $\lambda\leq |\nu/l|^{\f12}$, and then $\nu^{-1}\lambda n^{-2}\leq |\nu n^4l|^{-\f12}\leq |\nu nl|^{-\f12}$. Thus,
\begin{align*}
   &\|U\|_{L^2}\leq C\big(\nu^{-1}\lambda|n|^{-2}+ |\nu n^2l|^{-\f12}\big)\|F\|_{L^2}\leq C|\nu nl|^{-\f12}\|F\|_{L^2}.
\end{align*}
Thanks to $|\nu\lambda l^2|^{\f13}=|\nu l|^{\f12}|\lambda^2 l/\nu|^{\f16}\leq |\nu l|^{\f12}$, we get
\begin{align*}
   &(|\nu nl|^{\f12}+|\nu\lambda l^2|^{\f13})\|U\|_{L^2}\leq (|\nu nl|^{\f12}+|\nu l|^{\f12})\|U\|_{L^2}\leq  C\|F\|_{L^2},
\end{align*}
and due to \eqref{eq:eih-est4}, we have
\begin{align*}
   |l|\|(\lambda-r^2)U\|_{L^2}&\leq |l\lambda|\|U\|_{L^2}+|l|\|r^2U\|_{L^2}\leq C\big(|l\lambda|\|U\|_{L^2}+\|F\|_{L^2}\big)\\
   &\leq C\big(|\nu l|^{\f12}\|U\|_{L^2}+\|F\|_{L^2}\big)\leq C\|F\|_{L^2}.
\end{align*}
Thus, we conclude that 
\begin{align*}
   &(\nu l^2+|\nu nl|^{\f12}+|\nu \lambda l^2|^{\f13})\|U\|_{L^2}+|l|\|(\lambda-r^2)U\|_{L^2}\leq C\|F\|_{L^2}.
\end{align*}

Finally, it remains to estimate $\|rU\|_{L^1}$. Let $ \delta_1=|\nu/l|^{\f12}+|\nu \lambda/l|^{\f13}$. Then we have
\begin{align*}
   \|rU\|_{L^1}&\leq \|r/(\delta_1+|\lambda-r^2|)\|_{L^2}\big(\delta_1\|U\|_{L^2}+\|(\lambda-r^2)U\|_{L^2}\big).
\end{align*}
We have proven that 
\begin{align*}
   |l|\big(\delta_1\|U\|_{L^2}+\|(\lambda-r^2)U\|_{L^2}\big)=(|\nu l|^{\f12}+|\nu \lambda l^2|^{\f13})\|U\|_{L^2}+|l|\|(\lambda-r^2)U\|_{L^2}\leq C\|F\|_{L^2}.
\end{align*}
A direct calculation gives
 \begin{align*}\|r/(\delta_1+|\lambda-r^2|)\|_{L^2}^2=&\int_{0}^s\frac{r^3}{(\delta_1+|\lambda-r^2|)^2}\mathrm{d}r
=\int_{0}^{s^2}\frac{\tau}{(\delta_1+|\lambda-\tau|)^2}\frac{\mathrm{d}\tau}{2}\\
\leq& \int_{0}^{s^2}\frac{|\lambda|}{(\delta_1+|\lambda-\tau|)^2}\frac{\mathrm{d}\tau}{2}+
\int_{0}^{s^2}\frac{|\lambda-\tau|}{(\delta_1+|\lambda-\tau|)^2}\frac{\mathrm{d}\tau}{2}\\
 \leq&\frac{|\lambda|}{\delta_1}+\ln\f{|\lambda|+s^2+\delta_1}{\delta_1}.
 \end{align*}
 
 If $ |\lambda|\leq 2$, then $|\lambda||\nu/l|^{\f13}\leq2|\nu \lambda/l|^{\f13}\leq2\delta_1$, that is, $|\lambda|/\delta_1\leq 2|\nu/l|^{-\f13}$,  and
 \begin{align*}\ln[{(|\lambda|+s^2+\delta_1)}/{\delta_1}]\leq\ln[{(3+\delta_1)}/{\delta_1}]\leq C\delta_1^{-\f23}\leq C|\nu/l|^{-\f13}, \end{align*}
which gives
\beno
  \|r/(\delta_1+|\lambda-r^2|)\|_{L^2}^2\leq C|\nu/l|^{-\f13}.
\eeno
If $ |\lambda|\geq 2$, then we have
\begin{align*}
\|r/(\delta_1+|\lambda-r^2|)\|_{L^2}\leq& \|1/(\delta_1+|\lambda|-1)\|_{L^2}\leq1/(\delta_1+|\lambda|-1)\\ \leq& \delta_1^{-\f13}(|\lambda|-1)^{-\f23}\leq \delta_1^{-\f13}\leq |\nu/l|^{-\f16}.
\end{align*}
This shows that $ \|r/(\delta_1+|\lambda-r^2|)\|_{L^2}\leq C|\nu/l|^{-\f16}$.  Thus, we have
\begin{align*}
   & \|rU\|_{L^1}\leq C|\nu/l|^{-\f16}|l|^{-1}\|F\|_{L^2}\leq C\nu^{-\f16}|l|^{-\f56}\|F\|_{L^2}.
\end{align*} 

This completes the proof of the proposition.
\end{proof}

\begin{Proposition}\label{prop:ellip-inhom-v2}
Let  $U\in H^2(0,s)$ solve 
\begin{align*}
&-\nu \widehat{\Delta}_1{U}+{\ir}l(\lambda-r^2)U=F,\quad  U(s)=0.
\end{align*}
There exists a constant $c_3$ so that if $ l\lambda_i\leq c_3|\nu nl|^{\f12},$  $s\in(0,1],$ then we have
\begin{align*}
(|\nu nl|^{\f12}+|\nu\lambda l^2|^{\f13})\|U\|_{L^2(0,s)}+\nu^{\f16}|l|^{\f56}\|rU\|_{L^1(0,s)}\leq C\|F\|_{L^2}.
\end{align*}
\end{Proposition}

\begin{proof}\def\q{q}
Thanks to 
 \begin{align*}
    & \nu\big\langle-\widehat{\Delta}_1U,U\big\rangle- l\lambda_i\|U\|_{L^2}^2+{\ir}l\big(\lambda_r\|U\|_{L^2}^2-\|rU\|_{L^2}^2\big)
 =\langle F,U\rangle
 \end{align*}
 and $\|U\|_{1}^2/2\leq \mathbf{Re}\big(-\langle\widehat{\Delta}_1U,U\rangle\big)$(due to Lemma \ref{lem:W1-H1}), we deduce that
 \begin{align*}
   & \nu\|U\|_{1}^2/2-l\lambda_i\|U\|_{L^2}^2\leq \|F\|_{L^2}\|U\|_{L^2}.
 \end{align*}
  If $l\lambda_i\leq 0$, then we have $|l\lambda_i|\|U\|_{L^2}\leq \|F\|_{L^2}$, and  if $l\lambda_i\geq 0$, we have  $|l\lambda_i|\|U\|_{L^2}\leq c_3|\nu nl|^{\f12}\|U\|_{L^2}$. Thus,
 \begin{align}\label{eq:eih-est5}
    &|l\lambda_i|\|U\|_{L^2}\leq \|F\|_{L^2}+c_{3}|\nu nl|^{\f12}\|U\|_{L^2},\\
    &\nu\|U\|_{1}^2\leq C\big(\|F\|_{L^2}+c_3|\nu nl|^{\f12}\|U\|_{L^2}\big)\|U\|_{L^2}.\label{eq:eih-est6}
 \end{align}
 Thanks to $-\nu\widehat{\Delta}U+{\ir}l(\lambda_r-r^2)U=F-\nu\dfrac{2rl^2\partial_rU}{n^2+r^2l^2} +l\lambda_iU,$
 we deduce from Proposition \ref{prop:ellip-inhom-v1}  that 
 \begin{align*}
    &(\nu l^2+|\nu nl|^{\f12}+|\nu\lambda_r l^2|^{\f13})\|U\|_{L^2}+\nu^{\f16}|l|^{\f56}\|rU\|_{L^1} \\
    &\leq C\Big(\|F\|_{L^2}+\nu\left\|\dfrac{2rl^2\partial_rU}{n^2+r^2l^2}\right\|_{L^2}+ |l\lambda_i|\|U\|_{L^2}\Big).
 \end{align*}
 Using the fact that $\nu\left\|\dfrac{2rl^2\partial_rU}{n^2+r^2l^2}\right\|_{L^2}\leq \nu|l/n|\|\partial_rU\|_{L^2}\leq \nu |l|\|U\|_{1}$,  we deduce from \eqref{eq:eih-est5} and \eqref{eq:eih-est6} that
  \begin{align*}
     &(\nu l^2+|\nu nl|^{\f12}+|\nu\lambda_r l^2|^{\f13})\|U\|_{L^2}+ \nu^{\f16}|l|^{\f56}\|rU\|_{L^1}\\
     &\leq  C\big(\|F\|_{L^2}+\nu|l|\|U\|_{1}+c_3|\nu nl|^{\f12}\|U\|_{L^2}\big)\\
     &\leq C\big(\|F\|_{L^2}+\nu^{\f12}|l/n|(\|F\|_{L^2}+c_3|\nu nl|^{\f12}\|U\|_{L^2})^{\f12}\|U\|_{L^2}^{\f12}+c_3|\nu nl|^{\f12}\|U\|_{L^2}\big)\\ &\leq C\big(\|F\|_{L^2}+c_3|\nu nl|^{\f12}\|U\|_{L^2}\big)+\dfrac{\nu l^2}{2}\|U\|_{L^2},
      \end{align*}
which gives
 \begin{align*}
     &(\nu l^2+|\nu nl|^{\f12}+|\nu\lambda_r l^2|^{\f13})\|U\|_{L^2}+\nu^{\f16}|l|^{\f56}\|rU\|_{L^1}
     \leq C\big(\|F\|_{L^2}+c_3|\nu nl|^{\f12}\|U\|_{L^2}\big),
  \end{align*}
 Taking $c_3$ sufficiently small so that $Cc_3\leq 1/2$, we obtain
\begin{align*}
   & (\nu l^2+|\nu nl|^{\f12}+|\nu\lambda_r l^2|^{\f13})\|U\|_{L^2} +\nu^{\f16}|l|^{\f56}\|rU\|_{L^1}\leq C\|F\|_{L^2}.
\end{align*} 
This along with \eqref{eq:eih-est5} gives
\begin{align*}
   & |l\lambda_i|\|U\|_{L^2}^2\leq \|F\|_{L^2}+c_3|\nu nl|^{\f12}\|\q\|_{L^2}\leq C\|F\|_{L^2}.
\end{align*}
Thanks to
\begin{align*}
   |\nu \lambda l^2|^{\f13}&\leq C(|\nu \lambda_r l^2|^{\f13}+|\nu \lambda_i l^2|^{\f13})= C\big(|\nu \lambda_r l^2|^{\f13}+(|\nu l|^{\f12})^{\f23}|l\lambda_i|^{\f13}\big)\\
   &\leq C\big(|\nu \lambda_r l^2|^{\f13}+|\nu l|^{\f12}+|l\lambda_i|\big)\leq C\big(|\nu \lambda_r l^2|^{\f13}+|\nu nl|^{\f12}+|l\lambda_i|\big),
\end{align*}
 we finally conclude that 
 \begin{align*}
    & (|\nu nl|^{\f12}+|\nu\lambda l^2|^{\f13})\|U\|_{L^2}+\nu^{\f16}|l|^{\f56}\|rU\|_{L^1}\\
    &\leq C\big(|\nu nl|^{\f12}+|\nu\lambda_r l^2|^{\f13}+|l\lambda_i|\big)\|U\|_{L^2}+ C\nu^{\f16}|l|^{\f56}\|rU\|_{L^1}\leq C\|F\|_{L^2}.
 \end{align*}
 
 This completes the proof of the proposition. 
\end{proof}

\subsection{Proof of Proposition \ref{prop:Wa-lower}}

It is enough to consider the following two cases: (1)$l\lambda_i\leq \max(|l(\lambda-1)|/2,\nu (n^2+l^2)/3)$ and $|\nu/l|^{1/3}A_1\geq 1/c_4$;
(2) $l\lambda_i\ge \max(|l(\lambda-1)|/2,\nu (n^2+l^2)/3)$. Indeed, for the second case, we have
\begin{align*}&{n^2+l^2}+|l(\lambda-1)|/\nu\leq 5c_4| nl/\nu|^{\f12}\Longrightarrow |n|^{\f32}\leq 5c_4| l/\nu|^{\f12}\Longrightarrow | nl/\nu|^{\f12}\leq | 5c_4|^{\f13}| l/\nu|^{\f23}\\
&\Longrightarrow n^2+l^2+|l(\lambda-1)|/\nu\leq |5c_4|^{\f43}| l/\nu|^{\f23},
\end{align*}
which implies $|\nu/l|^{1/3}A_1\le 1/c_4$ if $c_4\le \f15$.

\subsubsection{Case of  $l\lambda_i\leq \max(|l(\lambda-1)|/2,\nu (n^2+l^2)/3)$ and $|\nu/l|^{1/3}A_1\geq 1/c_4$}

Let $(U_a,W_a)$ solve 
 \begin{align*}
\left\{
\begin{aligned}
&-\nu \widehat{\Delta}_1{U}_a+{\ir}l(\lambda-1)U_a=0,\quad U_a(1)=1,\\
&\widehat{\Delta}_1{W}_a=U_a,\quad W_a(1)=0.
\end{aligned}\right.
\end{align*}
We denote $U_e=U-U_a$ and $W_e=W-W_a$, which satisfy 
\begin{align*}
\left\{\begin{aligned}
&{-\nu \widehat{\Delta}_1{U_e}+{\ir}l(\lambda-r^2)U_e=-{\ir}l(1-r^2)U_a,\quad U_e(1)=0,}\\ &{U_e=\widehat{\Delta}_1{W_e},\quad {W_e}|_{r=1}=0.}
\end{aligned}\right.
\end{align*}
It follows from Proposition \ref{prop:ellip-inhom-v2} that 
\begin{align*}
\nu^{\f16}|l|^{\f56}\|rU_e\|_{L^1}\leq C\|{\ir}l(1-r^2)U_a\|_{L^2}.
\end{align*}
Then we infer from Lemma \ref{lem:sob} and Proposition \ref{prop:UW-toy} that 
\begin{align*}
|\pa_rW_e(1)|\le& C\|r\widehat{\Delta}_1W_e\|_{L^1}= C\|rU_e\|_{L^1}\\
\le&C|\nu/l|^{-\f16}\|(1-r^2)U_a\|_{L^2}\leq C|\nu/l|^{-\f16}A_1^{-\f32}.
\end{align*}
from which and  Proposition \ref{prop:UW-toy}, we infer that 
\begin{align*}
&|{\partial_r W(1)}|\geq|{\partial_rW_a(1)}|-|{\partial_rW_e(1)}|\geq C^{-1}A_1^{-1}-C|\nu/l|^{-\f16}A_1^{-\f32}.
\end{align*}
Now we take $c_4$ small enough so that 
\begin{align*}
 2C|\nu/l|^{-\f16}A_1^{-\f32}= 2C(|\nu/l|^{\f13}A_1)^{-\f12}A_1^{-1}\leq2C(c_4)^{\f12}A_1^{-1}\leq C^{-1}A_1^{-1},
\end{align*}
which gives 
 \begin{align*}
 &|{\partial_rW(1)}|\geq C^{-1}A_1^{-1}-C|\nu/l|^{-\f16}A_1^{-\f32}\geq C^{-1}A_1^{-1}/2\geq C^{-1}A^{-1}.
\end{align*}

\subsubsection{Case of  $l\lambda_i\ge \max(|l(\lambda-1)|/2,\nu (n^2+l^2)/3)$} 

Let $w(r)$ be defined in subsection \ref{sec:airy}, which satisfies 
\begin{align*}
&-\nu\partial_r^2w+\nu(n^2+l^2)w+{\ir}l(\lambda-2r+1)w=0,\quad w(1)=1.
\end{align*}
A direct calculation gives 
\begin{align*}\widehat{\Delta}_1(r^2w)&=\frac{\partial_r(r\partial_r(r^2w))}{r}-(n^2+r^2l^2)w-\dfrac{2rl^2}{n^2+r^2l^2}\partial_r(r^2w)\\&
=4w+5r\partial_rw+r^2\partial_r^2w-(n^2+r^2l^2)w-\dfrac{4r^2l^2w+2r^3l^2\partial_rw}{n^2+r^2l^2}\\&
=\dfrac{4n^2w+(5n^2+3r^2l^2)r\partial_rw}{n^2+r^2l^2}+r^2\partial_r^2w-(n^2+r^2l^2)w.
\end{align*}
Thus, there holds that
\begin{align*}
&-\nu \widehat{\Delta}_1(r^2w)+{\ir}l(\lambda-r^2)(r^2w) =-\nu\dfrac{4n^2w+(5n^2+3r^2l^2)r\partial_rw}{n^2+r^2l^2}
\\&\qquad+r^2(-\nu\partial_r^2w+\nu(n^2+l^2)w+{\ir}l(\lambda-r^2)w) +\nu(1-r^2)n^2w\\&=-\nu\dfrac{4n^2w+(5n^2+3r^2l^2)r
\partial_rw}{n^2+r^2l^2}
-{\ir}l(1-r)^2r^2w+\nu(1-r^2)n^2w.
\end{align*}
Let 
\begin{align*}
&w_1=\dfrac{4n^2w+(5n^2+3r^2l^2)r\partial_rw}{n^2+r^2l^2},\quad
F=-\nu w_1-{\ir}l(1-r)^2r^2w+\nu(1-r^2)n^2w.
\end{align*}
Then we have $|w_1|\leq 5(|w|+|\partial_rw|)$ and
\begin{align*}
&-\nu \widehat{\Delta}_1(r^2w)+{\ir}l(\lambda-r^2)(r^2w)=F.
\end{align*}
Let $W_a$ solve  $\widehat{\Delta}_1W_a=r^2w,\  W_a(1)=0$.  We denote $(U_e,W_e)=(U-r^2w, W-W_a)$, which satisfies 
\begin{align*}
\left\{\begin{aligned}
&{-\nu \widehat{\Delta}_1{U_e}+{\ir}l(\lambda-r^2)U_e=-F,\quad U_e(1)=0,}\\ 
&{U_e=\widehat{\Delta}_1{W_e},\quad {W_e}|_{r=1}=0.}
\end{aligned}\right.
\end{align*}
We infer from Lemma \ref{lem:sob} and Proposition \ref{prop:UW-toy} that 
\begin{align}
|{\partial_rW_e(1)}|\leq& C\|rU_e\|_{L^1}\leq C\nu^{-\f16}|l|^{-\f56}\|F\|_{L^2}\nonumber\\ 
\leq& C\nu^{-\f16}|l|^{-\f56}\big(\nu\|w_1\|_{L^2}+|l|\|(1-r)^2w\|_{L^2}+\nu n^2\|(1-r^2)w\|_{L^2}\big)\nonumber\\ 
\leq& C\nu^{-\f16}|l|^{-\f56}\big(\nu\|w\|_{L^2}+\nu\|\partial_rw\|_{L^2}+|l|\|(1-r)^2w\|_{L^2}+\nu n^2\|(1-r^2)w\|_{L^2}\big).
\label{eq:lower-est1}
\end{align}

Now we need the estimates of $w$.  Due to $|\nu/l|^{\f13}A_1\leq c_4^{-1}$, we have
\begin{align*}
&|d|\leq({n^2+l^2}+|l(\lambda-1)|/\nu)\nu/(2l)\leq c_4^{-2}|\nu/l|^{-2/3}\nu/(2l)=c_4^{-2}|\nu/l|^{1/3}/2\leq c_4^{-2}L^{-1}.
\end{align*}
If $l>0$, then we have
\begin{align*}
\textbf{Im}(Ld)&=L\lambda_i/2-L(n^2+l^2)\nu/(2l)\leq L\lambda_i/2=l\lambda_i/|4\nu l^2|^{\f13}\leq c_4|\nu nl|^{\f12}/|4\nu l^2|^{\f13}\\
&\leq c_4|\nu/l|^{\f16}|n|^{\f12}\leq c_4\big(|\nu/l|^{\f13}A_1\big)^{\f12}\leq c_4(c_4)^{-\f12}=c_4^{\f12}.
\end{align*}
If $c_4\leq \delta^2_0$, then we have $\textbf{Im}(Ld)\le \delta_0$ (now fix this $c_4$). Thus, by Lemma \ref{lem:Airy-w} and \eqref{eq:lower-est1}, we get
\begin{align*}
|{\partial_rW_e(1)}|\leq& C\nu^{-\f16}|l|^{-\f56}\big(\nu L^{-\f12}+\nu L^{\f12}+|l|L^{-\f52}+\nu n^2L^{-\f32}\big)\\
\leq& C\nu^{-\f16}|l|^{-\f56}\nu L^{\f12}=CL^{-2}.
\end{align*}Here we used $L=|2l/\nu|^{\f13},\ |n|\leq A\leq CL$.

Let $J(r)$ be given by Lemma \ref{lem:har-bound} and $J^*(r)$ given by Remark \ref{rem:J-star}. We know that 
$J(r)$ is increasing in $(0,1]$, so is $r^3J^*(r)=r(n^2+l^2)r^2/(n^2+r^2l^2)\cdot J(r).$ Thanks to $\widehat{\Delta}_{1}^{*}J^*=0,\ J^*(1)=1$, then if $l>0$, we have
\begin{align*}
\partial_rW_a(1)=&\partial_rW_a(1)+ \langle W_a,\widehat{\Delta}_{1}^{*}J^*\rangle =\langle r^2w,J^*\rangle=\int_0^1r^3w(r)J^*(r)\mathrm{d}r\\=&
\int_0^1r^3\frac{Ai\big(\mathrm{e}^{{\ir}\f{\pi}{6}}L(r-1+d)\big)}
{Ai\big(\mathrm{e}^{{\ir}\f{\pi}{6}}Ld\big)}J^*(r)\mathrm{d}r=
\int_0^1r^3\frac{A_0'\big(L(d+1-r)\big)}
{A_0'\big(Ld\big)}J^*(r)\mathrm{d}r\\
=&-\frac{A_0\big(Ld\big)}
{LA_0'\big(Ld\big)}+\int_0^1\frac{A_0\big(L(d+1-r)\big)}
{LA_0'\big(Ld\big)}\partial_r(r^3J^*)(r)\mathrm{d}r,
\end{align*}
which along with Lemma \ref{lem:Airy-w} gives 
\begin{align}
L|A_0'\big(Ld\big)||\partial_rW_a(1)|\geq& |A_0\big(Ld\big)|-\int_0^1\left|A_0\big(L(d+1-r)\big)\right|
\partial_r(r^3J^*)(r)\mathrm{d}r\nonumber\\ \geq&|A_0\big(Ld\big)|-\int_0^1|A_0\big(Ld\big)|\mathrm{e}^{-L(1-r)/3}
\partial_r(r^3J^*)(r)\mathrm{d}r\nonumber\\
=&|A_0\big(Ld\big)|\frac{L}{3}\int_0^1\mathrm{e}^{-L(1-r)/3}
r^3J^*(r)\mathrm{d}r.\label{eq:lower-est2}
\end{align}
By Lemma \ref{lem:har-bound}, we have
\begin{align*}
&r^3J^*(r)=r^3(n^2+l^2)J(r)/(n^2+r^2l^2)\geq r^3J(r)\geq r^{3+|n|+|l|},
\end{align*}
which gives
\begin{align*}
\int_0^1\mathrm{e}^{-L(1-r)/3}
r^3J^*(r)\mathrm{d}r\geq \int_0^1r^{L/3}r^{3+|n|+|l|}\mathrm{d}r
\geq \frac{1}{4+|n|+|l|+L/3}.
\end{align*}
Since $|Ld|\leq C,$ we have $\dfrac{|A_0\big(Ld\big)|}{3|A_0'\big(Ld\big)|}\geq C^{-1}$, and $4+|n|+|l|\leq C(n^2+l^2)^{1/2}\leq CA_1\leq CL$. Then by \eqref{eq:lower-est2}, we conclude that
\begin{align*}
 |\partial_rW_a(1)|&\geq\frac{|A_0\big(Ld\big)|}{3|A_0'\big(Ld\big)|}\int_0^1\mathrm{e}^{-L(1-r)/3}
r^3J^*(r)\mathrm{d}r\\
&\geq\frac{|A_0\big(Ld\big)|}{|A_0'\big(Ld\big)|(3(4+|n|+|l|)+L)}\geq \frac{1}{CL}.
\end{align*}
Thus, we obtain
\begin{align*}
&|{\partial_rW(1)}|\geq|{\partial_rW_a(1)}|-|{\partial_rW_e(1)}|\geq C^{-1}L^{-1}-CL^{-2}.
\end{align*}
We take $c_4'$ small enough so that $L=|2l/\nu|^{\f13}\geq 2^{\f13}c_4^{-\f13}\geq 2C^{2}$. Then we have
\begin{align*}
&|{\partial_rW(1)}|\geq C^{-1}L^{-1}-CL^{-2}\geq C^{-1}L^{-1}/2.
\end{align*}

The proof of  the case $l<0$ is similar. This completes the proof of  Proposition \ref{prop:Wa-lower}.

\subsection{Proof of Proposition \ref{prop:Ua}}
\begin{Lemma}\label{lem:A-v2}
For $s\in(0,1],\ \lambda\in\C,\ l\lambda_i\leq c_5|\nu nl|^{\f12},$ $c_5\leq 1/5,$ there exists $ \widetilde{\lambda}\in\C$ such that 
\beno
&& l\widetilde{\lambda}_i\leq \max\big(|l(\widetilde{\lambda}-s^2)|/2,\nu (n^2/s^2+l^2)/3\big),\\
&&|\nu /(sl)|^{1/3}({n^2/s^2+l^2}+|l(\widetilde{\lambda}-s^2)|/\nu)^{1/2}\geq 1,\\
&&(\widetilde{A}_1(s)/s)^{-\f32}+|\lambda-\widetilde{\lambda}|(\widetilde{A}_1(s)/s)^{-\f12}\leq C|l|^{-1}(A(s)/s)^{-\f12}\big(|\nu nl|^{\f12}+|\nu\lambda l^2|^{\f13}\big).
\eeno
Here $\widetilde{A}_1(s)=({n^2/s^2+l^2}+|l(\widetilde{\lambda}-s^2)|/\nu)^{1/2}.$
\end{Lemma}

\begin{proof}
Similar to the arguments above section 5.3.1, it is enough to consider the following two cases.\smallskip

\no {\bf Case 1}. $l\lambda_i\leq \max(|l(\lambda-s^2)|/2,\nu(n^2/s^2+l^2)/3)$ and $|\nu/(sl)|^{\f13}A_1(s)\geq 1$. 

Let $\widetilde{\lambda}=\lambda$. Then we have
  \begin{align*}
& l\widetilde{\lambda}_i\leq \max(|l(\lambda-s^2)|/2,\nu(n^2/s^2+l^2)/3)=\max(|l(\widetilde{\lambda}-s^2)|/2,\nu(n^2/s^2+l^2)/3),  \\
& |\nu/(sl)|^{1/3}({n^2/s^2+l^2}+|l(\widetilde{\lambda}-s^2)|/\nu)^{1/2}= |\nu/(sl)|^{\f13}A_1(s)^{\f12}\geq 1.
  \end{align*}
Thanks to $A_1(s)\geq |sl/\nu|^{\f13}$, we have $A_1(s)=\widetilde{A}_1(s)\geq C^{-1}A(s)$. Then we get
  \begin{align*}
     &(A_1(s)/s)^{-\f32}\leq C(A(s)/s)^{-\f32}\leq C(A(s)/s)^{-\f12}s\big(|ls/\nu|^{\f13}+|l(\lambda-s^2)/\nu|^{\f12}\big)^{-1}.
  \end{align*}  
 If $|\lambda-s^2|\geq s^2/2$, then
  \begin{align*}
     & s(|ls/\nu|^{\f13}+|l(\lambda-s^2)/\nu|^{\f12})^{-1} \leq s|l(\lambda-s^2)/\nu|^{-\f12}\leq Cs|ls^2/\nu|^{-\f12}\leq C|l|^{-1}|\nu nl|^{\f12}.
  \end{align*}
  If $|\lambda-s^2|\leq s^2/2$, then $|\lambda|\geq C^{-1}s^2$ and 
  \begin{align*}
     & s(|ls/\nu|^{\f13}+|l(\lambda-s^2)/\nu|^{\f12})^{-1} \leq s|ls/\nu|^{-\f13}\leq C|l|^{-1}|\nu s^2l^2|^{\f13}\leq C|l|^{-1}|\nu \lambda l^2|^{\f13}.
  \end{align*}  
 Then $s(|ls/\nu|^{\f13}+|l(\lambda-s^2)/\nu|^{\f12})^{-1} \leq C|l|^{-1}(|\nu nl|^{\f12}+|\nu\lambda l^2|^{\f13})$ and  
  \begin{align*}
     &(A_1(s)/s)^{-\f32}\leq C(A(s)/s)^{-\f32}\leq  C|l|^{-1}\big(A(s)/s)^{-\f12}(|\nu nl|^{\f12}+|\nu\lambda l^2|^{\f13}\big).
  \end{align*}  
 As $\widetilde{A}_1=A_1$ and  $\widetilde{\lambda}=\lambda$, we obtain
  \begin{align*}
  (\widetilde{A}_1(s)/s)^{-\f32}+|\lambda-\widetilde{\lambda}|(\widetilde{A}_1(s)/s)^{-\f12}\leq C|l|^{-1}(A(s)/s)^{-\f12}\big(|\nu nl|^{\f12}+|\nu\lambda l^2|^{\f13}\big).
\end{align*}

\no {\bf Case 2}.  $|\nu/(sl)|^{\f13}A_1(s)^{\f12}\leq 1$.

In this case, we have $A(s)\leq C|sl/\nu|^{\f13}$ and
\begin{align*}
|l(\lambda-s^2)|/\nu\leq|sl/\nu|^{2/3}\Rightarrow |\lambda-s^2|\leq|\nu s^2/l|^{1/3}.
\end{align*}
If $ |\lambda|\leq s^2/2 $, then we have
\begin{align*}
&s^2/2\leq|\lambda-s^2|\leq|\nu s^2/l|^{1/3}\Rightarrow s^{4/3}\leq2|\nu/l|^{1/3}
\\&\Rightarrow s^4\leq 8\nu/l,\Rightarrow\nu s^2 l^2\leq\nu  l^2|8\nu/l|^{\f12}\leq|2\nu l|^{\f32}\Rightarrow|\nu s^2 l^2|^{\f13}\leq|2\nu l|^{\f12}.
\end{align*}
If $ |\lambda|\geq s^2/2 $ then we have
$|\nu s^2 l^2|^{\f13}\leq|2\nu\lambda l^2|^{\f13}.$
Thus, we have
\begin{align}
|\nu s^2 l^2|^{\f13}\leq\max(|2\nu l|^{\f12},|2\nu\lambda l^2|^{\f13})\leq 2(|\nu nl|^{\f12}+|\nu\lambda l^2|^{\f13}).\label{est:nus2l2}
\end{align}
Let $ \widetilde{\lambda}=s^2-{\ir}|\nu s^2 l^2|^{\f13}/l$. Then we have $|l(\widetilde{\lambda}-s^2)|/\nu=|sl/\nu|^{2/3}$, and 
\begin{align*}
   & l\widetilde{\lambda}_i = -|\nu s^2l^2|^{\f13}\leq 0\leq \max(|l(\widetilde{\lambda}-s^2)|/2,\nu (n^2/s^2+l^2)/3),\\
   &|\nu/(sl)|^{\f13}\widetilde{A}_1(s)\geq |\nu/(sl)|^{\f13}|l(\widetilde{\lambda}-s^2)/\nu|^{\f12}=1,\\
&|\lambda-\widetilde{\lambda}|\leq |\lambda-s^2|+|s^2-\widetilde{\lambda}|\leq |\nu s^2/l|^{1/3}+|\nu s^2/l|^{1/3}=2|\nu s^2/l|^{1/3}.\end{align*}
Thanks to $\widetilde{A}_1(s)/s\geq |l(\widetilde{\lambda}-s^2)/(\nu s^2)|^{\f12}=|l/(\nu s^2)|^{\f13}$, we deduce that
\begin{align*}
   &\big(\widetilde{A}_1(s)/s\big)^{-\f32}+|\lambda-\widetilde{\lambda}|\big(\widetilde{A}_1(s)/s\big)^{-\f12} \leq 3|\nu s^2/l|^{\f12}=3|l|^{-1}|\nu s^2/l|^{\f16}|\nu s^2l^2|^{\f13}.
\end{align*}
Since $|\nu s^2/l|^{\f16}=\big(|sl/\nu|^{\f13}/s\big)^{-\f12}\leq C(A(s)/s)^{-\f12},$ we get by \eqref{est:nus2l2} that
\begin{align*}
   \big(\widetilde{A}_1(s)/s\big)^{-\f32}+|\lambda-\widetilde{\lambda}|\big(\widetilde{A}_1(s)/s\big)^{-\f12}&\leq 3|l|^{-1}|\nu s^2/l|^{\f16}|\nu s^2l^2|^{\f13}\\ &\leq C|l|^{-1}(A(s)/s)^{-\f12} (|\nu nl|^{\f12}+|\nu\lambda l^2|^{\f13}).
\end{align*}

This proves the lemma.
\end{proof}\smallskip

Now we prove Proposition \ref{prop:Ua}.

\begin{proof}
We choose $\widetilde{\lambda}$ to be as in the Lemma \ref{lem:A-v2}.
Let $U^s$ solve
\begin{align*}
-\nu \widehat{\Delta}_1{U^s}+{\ir}l(\widetilde{\lambda}-s^2)U^s=0\ \text{in}\ (0,s),\quad U^s(s)=1.
\end{align*}
Then we infer from Proposition \ref{prop:UW-toy-v2} that
\begin{align}
\label{est:qd1}& 
\|U^s\|_{L^2(0,s)}^2\leq C(\widetilde{A}_1(s)/s)^{-1},\quad \|(s^2-r^2)U^s\|_{L^2}^2\leq C(\widetilde{A}_1(s)/s)^{-3}.
\end{align}
Let $U_e=U-U(s)U^s$, which satisfies
\begin{align*}
-\nu \widehat{\Delta}_1{U_e}+{\ir}l(\lambda-r^2)U_e={\ir}l\big(\widetilde{\lambda}-\lambda-(s^2-r^2)\big)U(s)U^s\ \text{in}\ (0,s),\ U_e(s)=0.
\end{align*}
By Proposition \ref{prop:ellip-inhom-v2}, Lemma \ref{lem:A-v2} and \eqref{est:qd1}, we have
\begin{align*}
(|\nu nl|^{\f12}+|\nu\lambda l^2|^{\f13})\|U_e\|_{L^2(0,s)} \leq& C\|{\ir}l(\widetilde{\lambda}-\la-(s^2-r^2))U(s)U^s\|_{L^2}\\ \leq&C|lU(s)|\big(\|(s^2-r^2)U_s\|_{L^2}+|\lambda-\widetilde{\lambda}|\|U^s\|_{L^2}\big)
\\ \leq& C|lU(s)|\big((\widetilde{A}_1(s)/s)^{-\f32}+|\lambda-\widetilde{\lambda}|(\widetilde{A}_1(s)/s)^{-\f12}\big) \\ 
\leq& C(|\nu nl|^{\f12}+|\nu\lambda l^2|^{\f13})(A(s)/s)^{-\f12}|U(s)|,
\end{align*}
which shows that
\begin{align}\label{est:qe1}
  &\|U_e\|_{L^2(0,s)}\leq C(A(s)/s)^{-\f12}|U(s)|.
\end{align}
By \eqref{est:qd1} and \eqref{est:qe1}, we obtain
\begin{align*}
\|U\|_{L^2(0,s)}\leq\|U_e\|_{L^2(0,s)}+\|U^s\|_{L^2(0,s)}|U(s)|\leq C(A(s)/s)^{-\f12}|U(s)|.
\end{align*}

This completes the proof of the proposition. 
\end{proof}

\section{The homogeneous linearized system}

In this section, we consider the homogeneous linearized system:
\begin{align}\label{eq:LNS-hom}
\left\{\begin{array}{l}
-\nu (2\widehat{\Delta} {U}-\widehat{\Delta}_1^*W)+{\ir}l(\lambda-r^2){U}=0,\\
-\nu \widehat{\Delta}_1{U}+{\ir}l(\lambda-r^2){W}+\dfrac{4{\ir}n^2l{W}}{n^2+r^2l^2}
=0,\\ 
W_1=\widehat{\Delta}_1{W},\quad {W}|_{r=1}=0,\quad {W_1}|_{r=1}={U}|_{r=1}=1.
\end{array}\right.
\end{align}

\begin{Proposition}\label{prop:lower}
Let $A=A(1)$ with $A(s)$ given by \eqref{def:A}. There exists a constant $c_6>0$ such that if $\lambda_i\leq c_6|\nu n l|^{\f12},$ and $0<\nu<c_6\min(|l|,1)$, the solution of \eqref{eq:LNS-hom} satisfies
\begin{align*}
&|\partial_r{W}(1)|\geq C^{-1}A^{-1}.
\end{align*}
\end{Proposition}

We define 
\begin{align}\label{eq:WU-a}
\left\{
\begin{aligned}
&-\nu \widehat{\Delta}_1{U_a}+{\ir}l(\lambda-r^2)U_a=0,\quad U_a(1)=1,\\
&\widehat{\Delta}_1{W}_a=U_a,\quad W_a(1)=0.
\end{aligned}\right.
\end{align}
Then we introduce
\beno
U_e=U-U_a,\quad W_{1,e}=W_1-U_a,\quad W_e=W-W_a.
\eeno
Then $(U_e, W_{e})$ satisfies 
\begin{align}\label{eq:WE-e}
\left\{\begin{array}{l}-\nu (2\widehat{\Delta} {U_e}-\widehat{\Delta}_1^*{W_{1,e}})+{\ir}l(\lambda-r^2){U_e}=\nu(2\widehat{\Delta}-\widehat{\Delta}_1^*-\widehat{\Delta}_1)U_a,\\-\nu \widehat{\Delta}_1{U_e}+{\ir}l(\lambda-r^2){W_{1,e}}+\dfrac{4{\ir}n^2l{W_e}}{n^2+r^2l^2}
=-\dfrac{4{\ir}n^2l{W_a}}{n^2+r^2l^2},\\  W_{1,e}=\widehat{\Delta}_1{W_e},\ {W_e}|_{r=1}=0,\ {W_{1,e}}|_{r=1}={U_e}|_{r=1}=0.
\end{array}\right.
\end{align}

We need the following lemmas.

\begin{Lemma}\label{lem:error}
   It holds that
   \begin{align*}
    & \left\|(2\widehat{\Delta}-\widehat{\Delta}_{1}-\widehat{\Delta}_{1}^{*})U_a\right\|_{L^2}\leq C l^2n^2(n^2+l^2)^{-2}A^{-\f12},\\
     &\left\|\dfrac{4{\ir}n^2l{W_a}}{n^2+r^2l^2}\right\|_{L^2}\leq Cn^2|l|(n^2+l^2)^{-\f54}A^{-2}.
  \end{align*}
\end{Lemma}

\begin{proof}
 Thanks to $\partial_s\big(\|U_a\|_{L^2(0,s)}^2\big)=s|U_a(s)|^2$,  we get by Proposition \ref{prop:Ua} and Lemma \ref{lem:A}
 that
 \begin{align*}
  \|U_a\|_{L^2(0,1)}^2\leq CA^{-1},\quad \|U_a\|_{L^2(0,s)}^2\leq \|U_a\|_{L^2(0,1)}^2e^{-c\int_s^1A(r)\mathrm{d}r}\leq CA^{-1}e^{-c(1-s)A}.
  \end{align*}
  Then we deduce that
  \begin{align*}
     \left\|\dfrac{U_a}{(n^2+r^2l^2)^2}\right\|_{L^2}^2= &\left\|\dfrac{U_a}{(n^2+r^2l^2)^2}\right\|_{L^2(1/2,1)}^2+ \left\|\dfrac{U_a}{(n^2+r^2l^2)^2}\right\|_{L^2(0,1/2)}^2\\ \leq& \frac{\|U_a\|_{L^2(0,1)}^2}{(n^2+|l/2|^2)^4}+\frac{\|U_a\|_{L^2(0,1/2)}^2}{n^8} \leq\frac{CA^{-1}}{(n^2+|l/2|^2)^4}+\frac{CA^{-1}\mathrm{e}^{-c\cdot(1/2)\cdot A}}{n^8}\\ \leq& \frac{CA^{-1}}{(n^2+l^2)^4}+\frac{CA^{-9}}{n^8} \leq C (n^2+l^2)^{-4}A^{-1}+C n^{-8}A^{-1}(n^2+l^2)^{-4}\\ \leq& C (n^2+l^2)^{-4}A^{-1},
  \end{align*}
 which gives the first inequality of the lemma due to the fact that
 \beno
  \left\|(2\widehat{\Delta}-\widehat{\Delta}_{1}-\widehat{\Delta}_{1}^{*})U_a\right\|_{L^2}= \left\|\dfrac{4l^2n^2}{(n^2+r^2l^2)^2}U_a\right\|_{L^2}.
  \eeno
  
Thanks to $\partial_s\big(\|U_a\|_{L^2(0,s)}^2\big)=s|U_a(s)|^2$, we get by integration by parts that 
  \begin{align*}
     \left\|{(1-r)^2}{U_a}\right\|_{L^2}^2&=\int_0^14(1-s)^3\|U_a\|_{L^2(0,s)}^2\mathrm{d}s\\
     &\leq\int_0^1CA^{-1}(1-s)^3\mathrm{e}^{-c(1-s)A}\mathrm{d}s\leq CA^{-5}.
  \end{align*}
  
 Let $\ph$ solve 
 \begin{align*}
     \widehat{\Delta}_{1}^{*}\ph=\dfrac{{W_a}}{(n^2+r^2l^2)^2},\quad \ph(1)=0.
  \end{align*}
Notice that $(n^2+r^2l^2)\widehat{\Delta}_{1}^{*}\ph=\widehat{\Delta}_{1}\big((n^2+r^2l^2)\ph\big).$ Let $ \ph_1=(n^2+r^2l^2)\ph$ and
$\widetilde{W}_a=\dfrac{{W_a}}{(n^2+r^2l^2)}$.  Then we have
\begin{align*}
     \widehat{\Delta}_{1}\ph_1=\widetilde{W}_a,\quad \ph_1(1)=0.
  \end{align*}
Moreover, we have
\begin{align*}
     &\left\|\dfrac{{W_a}}{n^2+r^2l^2}\right\|_{L^2}^2=\langle W_a,\widehat{\Delta}_{1}^{*}\ph\rangle=\langle\widehat{\Delta}_{1}W_a,\ph\rangle=\langle U_a,\ph\rangle.
  \end{align*}  
  
Now let us first give some estimates of  $\ph_1$. Thanks to $\|\ph_1\|_1^2\leq -2\mathbf{Re}\langle \ph_1,\widehat{\Delta}_1\ph_1\rangle$ and $ (n^2+l^2)\|\ph_1\|_{L^2}^2\leq \|\ph_1\|_1^2,$ we obtain
\begin{align*}
     (n^2+l^2)\|\ph_1\|_{L^2}^2&\leq\|\ph_1\|_1^2\leq -2\mathbf{Re}\langle \ph_1,\widehat{\Delta}_1\ph_1\rangle=-2\mathbf{Re}\langle \ph_1, \widetilde{W}_a\rangle\\&\leq 2\|\ph_1\|_{L^2}\|\widetilde{W}_a\|_{L^2},
     \end{align*}
which shows  that    
 \begin{align*} 
 \|\ph_1\|_{L^2}&\leq2(n^2+l^2)^{-1}\|\widetilde{W}_a\|_{L^2},\quad \|\ph_1\|_{1}\leq2(n^2+l^2)^{-\f12}\|\widetilde{W}_a\|_{L^2}.
  \end{align*}
 Since $\widehat{\Delta}_1\ph_1=
  \partial_r^2\ph_1+\dfrac{1}{r}\partial_r\ph_1-\left(\dfrac{n^2}{r^2}+l^2\right)\ph_1-\dfrac{2rl^2}{n^2+r^2l^2}\partial_r\ph_1$, we have
  \begin{align*}
      \|\partial_r^2\ph_1\|_{L^{2}(1/2,1)}&\leq C\big(\|\widehat{\Delta}_1\ph_1\|_{L^{2}(1/2,1)}+\|\partial_r\ph_1\|_{L^{2}(1/2,1)}+(n^2+l^2)\|\ph_1\|_{L^{2}(1/2,1)}\big)\\&\leq C\big(\|\widetilde{W}_a\|_{L^2}+\|\ph_1\|_{1}+(n^2+l^2)\|\ph_1\|_{L^{2}}\big)\leq C\|\widetilde{W}_a\|_{L^2},
  \end{align*}
and by the interpolation, we have
\begin{align*}
\|\partial_r\ph_1\|_{L^{\infty}(1/2,1)}&\leq C\big(\|\partial_r^2\ph_1\|_{L^{2}(1/2,1)}^{\f12}\|\partial_r\ph_1\|_{L^{2}(1/2,1)}^{\f12}+\|\partial_r\ph_1\|_{L^{2}(1/2,1)}\big)\\&\leq C(n^2+l^2)^{-\f14}\|\widetilde{W}_a\|_{L^2}.
\end{align*}
Using the fact that for $r\in (1/2,1)$, 
\beno
|(1-r)^{-1}\ph_1(r)|=\left|(1-r)^{-1}\int_{r}^{1}\partial_r\ph_1(s)\mathrm{d}s\right|\leq \|\partial_r\ph_1\|_{L^\infty(1/2,1)},
\eeno
we  have
\beno
\|(1-r)^{-1}\ph_1\|_{L^\infty(1/2,1)}\leq \|\partial_r\ph_1\|_{L^\infty(1/2,1)}.
\eeno 
Then we deduce that
\begin{align*}
      \|(1-r)^{-2}\ph_1\|_{L^{2}(1/2,1-1/(2A))}&\leq \|(1-r)^{-1}\|_{L^{2}(1/2,1-1/(2A))}\|(1-r)^{-1}\ph_1\|_{L^{\infty}(1/2,1)}\\ &\leq CA^{\f12}\|\partial_r\ph_1\|_{L^{\infty}(1/2,1)}\leq CA^{\f12}(n^2+l^2)^{-\f14}\|\widetilde{W}_a\|_{L^2},  
\end{align*}
and
\begin{align*}
 \|\ph_1\|_{L^{2}(1-1/(2A),1)}&\leq \|(1-r)\|_{L^{2}(1-1/(2A),1)}\|(1-r)^{-1}\ph_1\|_{L^{\infty}(1/2,1)}\\ &\leq CA^{-\f32}\|\partial_r\ph_1\|_{L^{\infty}(1/2,1)}\leq CA^{-\f32}(n^2+l^2)^{-\f14}\|\widetilde{W}_a\|_{L^2}.
  \end{align*}
  
Now we estimate $\|\widetilde{W}_a\|_{L^2}$. We have
\begin{align*}
 \|\widetilde{W}_a\|_{L^2}^2&=\langle U_a,\ph\rangle= \langle U_a,(n^2+r^2l^2)^{-1}\ph_1\rangle\leq n^{-2}\|U_a\|_{L^2(0,1/2)}\|\ph_1\|_{L^2}\\ &\quad+(n^2+|l/2|^2)^{-1}\|(1-r)^{2}U_a\|_{L^2}\|(1-r)^{-2}\ph_1\|_{L^{2}(1/2,1-1/(2A))}\\ &\quad+(n^2+|l/2|^2)^{-1}
     \|U_a\|_{L^2}\|\ph_1\|_{L^{2}(1-1/(2A),1)}\\ &\leq 2n^{-2}\|U_a\|_{L^2(0,1/2)}(n^2+l^2)^{-1}\|\widetilde{W}_a\|_{L^2}\\&\quad+C\big(A^{\f12}\|(1-r)^{2}U_a\|_{L^2}+A^{-\f32}
     \|U_a\|_{L^2}\big)(n^2+l^2)^{-\f54}\|\widetilde{W}_a\|_{L^2},
  \end{align*}
which shows  that
\begin{align*}
     (n^2+l^2)^{\f54}\|\widetilde{W}_a\|_{L^2}\leq&2n^{-2}\|U_a\|_{L^2(0,1/2)}(n^2+l^2)^{\f14}+C\big(A^{\f12}\|(1-r)^{2}U_a\|_{L^2}+A^{-\f32}
     \|U_a\|_{L^2}\big)\\ \leq&C\big(A^{\f12}+(n^2+l^2)^{\f14}\big)\|(1-r)^{2}U_a\|_{L^2}+CA^{-\f32}
     \|U_a\|_{L^2}\\ \leq&CA^{\f12}\|(1-r)^{2}U_a\|_{L^2}+CA^{-\f32}
     \|U_a\|_{L^2(0,1)}\\\leq& CA^{\f12}A^{-\f52}+CA^{\f32}
     A^{-\f12}\leq CA^{-2},
  \end{align*}
here we used 
\beno
\|(1-r)^2U_a\|_{L^2}\leq CA^{-\f52},\quad \|U_a\|_{L^2}\leq CA^{-\f12}.
\eeno 
Then we conclude that 
\begin{align*}
     \left\|\dfrac{4{\ir}n^2l{W_a}}{n^2+r^2l^2}\right\|_{L^2}&= 4n^2|l|\left\|\widetilde{W}_a\right\|_{L^2} \leq Cn^2|l|(l^2+n^2)^{-\f54}A^{-2},
  \end{align*}  
which gives the second inequality of the lemma.
  \end{proof}

\begin{Lemma}\label{lem:con}
Assume that $n,l,\nu\in\R,\ |n|\geq 1,\ 0<\nu<\min(|l|,1)$. Then we have
\begin{align*}
     &(|\nu nl|^{\f38}+|\nu\lambda l^2|^{\f14})A\geq C^{-1}|l/\nu|^{\f{1}{12}}|l|^{\f34},\\&\nu |l|^{\f74}|n|A^{\f12}\leq C\nu^{\f38}(|\nu nl|^{\f38}+|\nu\lambda l^2|^{\f14})(|n|+|l|)^{\f{11}{4}}.
  \end{align*}\end{Lemma}
  \begin{proof}
  If $|\lambda-1|\leq 1/2$, then $|\lambda|\geq 1/2$ and 
  \begin{align*}
     &|l/\nu|^{\f{1}{12}}|l|^{\f34}=|\nu l^2|^{\f14}|l/\nu|^{\f13}\leq |\nu (2\lambda) l^2|^{\f14}A\leq 2^{\f14}(|\nu n l|^{\f38}+|\nu\lambda l^2|^{\f14})A.
  \end{align*}
  If $|\lambda-1|\geq 1/2$, then $A\geq |l(\lambda-1)/\nu|^{\f12}\geq |l/\nu|^{\f12}/\sqrt{2}$ and 
  \begin{align*}
     & |l/\nu|^{\f{1}{12}}|l|^{\f34}=|\nu/l|^{\f{1}{24}}\nu^{-\f18}|l|^{\f78}\leq \nu^{-\f18}|l|^{\f78}= |\nu l|^{\f38}|l/\nu|^{\f12}\leq\sqrt{2}(|\nu n l|^{\f38}+|\nu\lambda l^2|^{\f14})A.
  \end{align*}
This shows the first inequality of the lemma.

It is easy to see that
  \begin{align*}
     \nu|l|^{\f74}|n||l/\nu|^{\f16}&=\nu^{\f56}|l|^{\f{23}{12}}|n|= |\nu/l|^{\f{1}{12}}\nu^{\f34}|l|^2|n|\leq \nu^{\f34}|l|^2|n|\\
     &= \nu^{\f38}|\nu nl|^{\f38}|n|^{\f58}|l|^{\f{13}{8}} \leq C\nu^{\f38}(|\nu nl|^{\f38}+|\nu\lambda l^2|^{\f14})(|n|+|l|)^{\f94}\\
     &\leq C\nu^{\f38}(|\nu nl|^{\f38}+|\nu\lambda l^2|^{\f14})(|n|+|l|)^{\f{11}{4}},
  \end{align*}
 and 
  \begin{align*}
      \nu|l|^{\f74}|n|(|n|+|l|)^{\f12}&= |\nu/l|^{\f14}\nu^{\f34}|l|^2|n|(|n|+|l|)^{\f12}\leq \nu^{\f34}|l|^2|n|(|n|+|l|)^{\f12}\\
     &= \nu^{\f38}|\nu nl|^{\f38}|n|^{\f58}|l|^{\f{13}{8}}(|n|+|l|)^{\f12} \leq C\nu^{\f38}(|\nu nl|^{\f38}+|\nu\lambda l^2|^{\f14})(|n|+|l|)^{\f{11}{4}}.
  \end{align*}
 Then we conclude 
 \beno
 \nu|l|^{\f74}|n|(|l/\nu|^{\f13}+|n|+|l|)^{\f12}\leq C\nu^{\f38}(|\nu nl|^{\f38}+|\nu\lambda l^2|^{\f14})(|n|+|l|)^{\f{11}{4}}.
 \eeno
 We also have
  \begin{align*}
     \nu |l|^{\f74}|n||l/\nu|^{\f14}&=\nu^{\f34}|l|^2|n|= \nu^{\f38}|\nu nl|^{\f38}|n|^{\f58}|l|^{\f{13}{8}}\\& \leq \nu^{\f38}(|\nu nl|^{\f38}+|\nu\lambda l^2|^{\f14})(|n|+|l|)^{\f94}
     \leq \nu^{\f38}(|\nu nl|^{\f38}+|\nu\lambda l^2|^{\f14})(|n|+|l|)^{\f{11}{4}},
  \end{align*}
and
  \begin{align*}
      \nu|l|^{\f74}|n||l\lambda/\nu|^{\f14}&=\nu^{\f34}|l|^2|\lambda|^{\f14}|n| =|\nu/l|^{\f18}\nu^{\f58}|l|^{\f{17}{8}}|\lambda|^{\f14}|n|\\
     &\leq \nu^{\f58}|l|^{\f{17}{8}}|\lambda|^{\f14}|n|= \nu^{\f38}|\nu\lambda l^2|^{\f14}|l|^{\f{13}{8}}|n|\leq \nu^{\f38}(|\nu nl|^{\f38}+|\nu\lambda l^2|^{\f14})(|n|+|l|)^{\f{21}{8}}\\
     &\leq \nu^{\f38}(|\nu nl|^{\f38}+|\nu\lambda l^2|^{\f14})(|n|+|l|)^{\f{11}{4}}.
  \end{align*}
Then we conclude 
  \begin{align*}
     \nu|l|^{\f74}|n|(|l(\lambda-1)/\nu|^{\f12})^{\f12}&\leq C\big(\nu|l|^{\f74}|n||l\lambda/\nu|^{\f14} +\nu|l|^{\f74}|n||l/\nu|^{\f14} \big)\\
     &\leq C\nu^{\f38}(|\nu nl|^{\f38}+|\nu\lambda l^2|^{\f14})(|n|+|l|)^{\f{11}{4}}.
  \end{align*}
  
  Summing up, we obtain
  \begin{align*}
     \nu|l|^{\f74}|n|A^{\f12}&\leq C\nu|l|^{\f74}|n|(|l/\nu|^{\f13}+|n|+|l|)^{\f12}+ C\nu|l|^{\f74}|n|(|l(\lambda-1)/\nu|^{\f12})^{\f12}\\
     &\leq C\nu^{\f38}(|\nu nl|^{\f38}+|\nu\lambda l^2|^{\f14})(|n|+|l|)^{\f{11}{4}},
  \end{align*}
 which gives the second inequality of the lemma.
\end{proof}

Now we are in a position to prove Proposition \ref{prop:lower}.

\begin{proof}
By Proposition \ref{prop:res-com} and Lemma \ref{lem:error}, we have
\begin{align*}&|\partial_rW_e(1)|\\
&\leq  C(|n|+|l|)^{\f54}n^{-1}|l|^{-\f14}(|\nu nl|^{\f12}+|\nu\lambda l^2|^{\f13})^{-\f34}\left(\nu\left\|(2\widehat{\Delta}-\widehat{\Delta}_{1}-\widehat{\Delta}_{1}^{*})U_a\right\|_{L^2}
  +\left\|\dfrac{4{\ir}n^2l{W_a}}{n^2+r^2l^2}\right\|_{L^2}\right)\\ &\leq  C(|n|+|l|)^{\f54}n^{-1}|l|^{-\f14}(|\nu nl|^{\f12}+|\nu\lambda l^2|^{\f13})^{-\f34}\big(\nu l^2n^2(l^2+n^2)^{-2}A^{-\f12}+n^2|l|(l^2+n^2)^{-\f54}A^{-2}\big)\\ &\leq  C\big(\nu |l|^{\f74}|n|(|n|+|l|)^{-\f{11}{4}}A^{-\f12}+|l|^{\f34}|n|(|n|+|l|)^{-\f{5}{4}}A^{-2}\big)\big(|\nu nl|^{\f38}+|\nu\lambda l^2|^{\f14}\big)^{-1}.
  \end{align*}
  from which and Lemma \ref{lem:con}, we deduce that 
  \begin{align*}&|\partial_rW_e(1)|\leq  C\big(\nu^{\f38}+|l|^{\f34}|n|(|n|+|l|)^{-\f{5}{4}}|l/\nu|^{-\f{1}{12}}|l|^{-\f34}\big)A^{-1}\leq  C\big(\nu^{\f38}+|\nu/l|^{\f{1}{12}}\big)A^{-1}.
  \end{align*}
Then by Proposition \ref{prop:Wa-lower},  we obtain
\begin{align*}
|\partial_r{W}(1)|\geq |\partial_r{W_a}(1)|-|\partial_rW_e(1)| \geq \big(C^{-1}-C(\nu^{\f38}+|\nu/l|^{\f{1}{12}})\big)A^{-1},
\end{align*}
which gives our result by taking $\nu$ small enough. 
 \end{proof}
  
\section{Linear stability for axisymmetric perturbations} 

In the axisymmetric case, the linearized resolvent system takes the form
 \begin{align}\label{eq:LNS-WU-s}
  \left\{\begin{aligned}
     &-\nu\widehat{\Delta}_{(1)}\Omega_1 +{\ir}l(\lambda-r^2)\Omega_1=0,\\
     &-\nu\widehat{\Delta}_{(1)}J_1+{\ir}l(\lambda-r^2)J_1=0,\\
     &\Omega_1=\widehat{\Delta}_{(1)}W_2,\ W_2|_{r=1}=\partial_rW_2|_{r=1}=0,\ J_1|_{r=1}=0,
  \end{aligned}\right.
  \end{align}
 where $\widehat{\Delta}_{(1)}=\dfrac{1}{r}\partial_r(r\partial_r)-\dfrac{1+r^2l^2}{r^2}$ and $\la\in \C$.
 
\begin{Proposition}\label{prop:J1}
Let $J_1$ solve  $-\nu\widehat{\Delta}_{(1)}J_1+{\ir}l(\lambda-r^2)J_1=F$ and $J_1(1)=0$.
 There exists a constant $c_7\in(0,1]$ such that if $l\lambda_i\leq c_7|\nu l|^{\f12}$, then it holds that
  \begin{align*}
   & (|\nu l|^{\f12}+|\nu \lambda l^2|^{\f13})\|J_1\|_{L^2}+\nu^{\f16}|l|^{\f56}\|rJ_1\|_{L^1}\leq C\|F\|_{L^2}.
\end{align*}
In particular, if $F=0$, then $J_1=0$.
\end{Proposition}

\begin{proof}
For $\la\in \R$, we apply Proposition \ref{prop:ellip-inhom-v1} with $n=1, s=1$ to obtain
 \begin{align*}
   & (|\nu l|^{\f12}+|\nu \lambda l^2|^{\f13})\|J_1\|_{L^2}+\nu^{\f16}|l|^{\f56}\|rJ_1\|_{L^1}\leq C\|F\|_{L^2}.
\end{align*}
We write the equation of $J_1$ as
  \begin{align*}
   &-\nu\widehat{\Delta}_{(1)}J_1+{\ir}l(\lambda_r-r^2)J_1 ={\ir}l\lambda_iJ_1 +F,\quad J_1(1)=0.
  \end{align*}
Then we have
\begin{align}
   & (|\nu l|^{\f12}+|\nu \lambda l^2|^{\f13})\|J_1\|_{L^2}+\nu^{\f16}|l|^{\f56}\|rJ_1\|_{L^1}\leq C\big(|l\lambda_i|\|J_1\|_{L^2}+\|F\|_{L^2}\big).\label{eq:J1-est1-ax}
\end{align}

Let $\|f\|_{(1)}^2= \|\partial_rf\|_{L^2}^2+\|f/r\|_{L^2}^2+l^2\|f\|_{L^2}^2$. Then we have $-\big\langle \widehat{\Delta}_{(1)}f,f\big\rangle=\|f\|_{(1)}^2$ and 
$\mathbf{Re}\big({\ir}l\langle(\lambda-r^2)J_1,J_1\rangle \big)=-l\lambda_i\|J_1\|_{L^2}^2$.  Taking the real part of 
  \begin{align*}
     &-\nu\big\langle \widehat{\Delta}_{(1)}J_1,J_1\big\rangle+{\ir}l\big\langle(\lambda-r^2)J_1,J_1\big\rangle =\langle F,J_1\rangle,
  \end{align*}
we obtain
\begin{align*}
     &\nu\|J_1\|_{(1)}^2-l\lambda_i\|J_1\|_{L^2}^2\leq |\langle F_2,J_1\rangle|\leq \|J_1\|_{L^2}\|F\|_{L^2}.
  \end{align*}
If $l\lambda_i\leq 0$, then we have $-l\lambda_i\|J_1\|_{L^2}\leq \|F\|_{L^2}$. Thus, we deduce that for $l\lambda_i\leq c_7|\nu l|^{\f12}$, 
  \begin{align*}
     &|l\lambda_i|\|J_1\|_{L^2}\leq c_7|\nu l|^{\f12}\|J_1\|_{L^2}+\|F\|_{L^2}.
  \end{align*}
Taking $c_7$ sufficiently small such that $Cc_7\leq 1/2$, we infer from \eqref{eq:J1-est1-ax} that
\begin{align*}
   & (|\nu l|^{\f12}+|\nu \lambda l^2|^{\f13})\|J_1\|_{L^2}+\nu^{\f16}|l|^{\f56}\|rJ_1\|_{L^1}\leq C\|F\|_{L^2}.
\end{align*}

This completes the proof of the proposition.
\end{proof}\smallskip

We denote
\begin{align*}
   &\tilde{A}= |l(\lambda-1)/\nu|^{\f12}+|l/\nu|^{\f13}+1+|l|,\quad \tilde{A}_1=(|l(\lambda-1)/\nu|+1+l^2)^{\f12}.
\end{align*}

\begin{Proposition}\label{prop:OM1}
Let $(\Om_1,W_1)$ solve 
 \begin{align}\label{eq:LNS-OM}
  \left\{\begin{aligned}
     &-\nu\widehat{\Delta}_{(1)}\Omega_1+{\ir}l(\lambda-r^2)\Omega_1=0,\\
     &\Omega_1=\widehat{\Delta}_{(1)}W_1, \quad W_1(1)=0,\quad \Om_1(1)=1.
  \end{aligned}\right.
  \end{align}
There exist constants $c_8, c_9\in(0,1)$ such that if  $l\lambda_i\leq c_8|\nu n l|^{\f12}$ and $0<\nu<c_9\min(|l|,1)$, then we have
\begin{align*}
|\partial_rW_1(1)|& \geq C^{-1}\tilde{A}^{-1}.
\end{align*}
\end{Proposition}

Assume $0<\nu<c_9\min(|l|,1)$ and take $c=\min(c_7,c_8)$. First of all, if $\Om_1(1)=0$, Proposition \ref{prop:J1} and Proposition \ref{prop:OM1} ensure that  if $l\la_i\le c|\nu nl|^\f12$ or  $\mathbf{Re} s\ge -c|\nu/L_z|^\f12$,  then the solution $\Om_1=J_1=0$ of \eqref{eq:LNS-WU-s}.  
Secondly, if $\Om_1(1)\neq 0$, we can normalize $\Om_1(1)=1$. Now it follows from Proposition \ref{prop:OM1} that  if $l\la_i\le c|\nu nl|^\f12$ or  $\mathbf{Re} s\ge -c|\nu/L_z|^\f12$,  we have $|\partial_rW_2(1)|\geq C^{-1}\tilde{A}^{-1}$, which contradicts 
with $\pa_rW_2(1)=0$. Thus, $\Om_1(1)=0$ and $\Om_1=0$ when $\mathbf{Re} s\ge -c|\nu/L_z|^\f12$.
This shows that for axisymmetric perturbations,  $m_1(\nu)\le -c|\nu/L_z|^{\f12}$.\smallskip


The proof of Proposition \ref{prop:OM1} is similar to Proposition \ref{prop:Wa-lower}.  We need the following lemma.

\begin{Lemma}\label{lem:UW-app-axs}
Let $(U,W)$ solve 
\ben\label{eq:LNS-OM-a}
 -\nu\widehat{\Delta}_{(1)}U+\ir l(\lambda-1)U=0,\quad U=\widehat{\Delta}_{(1)}W,\quad W(1)=0, \,\, U(1)=1.
 \een
If $2l\lambda_i\leq |l(\lambda-1)|$ or $3l\lambda_i\leq \nu(1+l^2)$, then we have
  \begin{align*}
     &\|U\|_{L^2}^2\leq C\tilde{A}_1^{-1},\quad \|(1-r^2)U\|_{L^2}^2\leq C\tilde{A}_1^{-3},\\     
     & |\partial_rU(1)|\leq C\tilde{A}_1,\quad  |\partial_rW(1)|\geq C^{-1}\tilde{A}^{-1}_1.
  \end{align*}
\end{Lemma}

\begin{proof}
We follow the proof of Proposition \ref{prop:UW-toy}.
We get by integration by parts  that 
\begin{align*}
  0 &=\langle -\nu\widehat{\Delta}_{(1)}U+{\ir}l(\lambda-1)U, U\rangle_s \\
  &=\nu\big(\|\partial_rU\|_{L^2(0,s)}^2+\|U/r\|_{L^2(0,s)}^2+l^2\|U\|_{L^2(0,s)}^2\big) -\nu s\partial_rU(s)\overline{U(s)}\\
  &\quad+{\ir}l(\lambda-1)\|U\|_{L^2(0,s)}^2.
\end{align*}
Thanks to $2l\lambda_i\leq |l(\lambda-1)|$ or $3l\lambda_i\leq \nu(1+l^2)$,  we have
\begin{align*}
   &2l\lambda_i\leq |l(\lambda-1)|\ \text{or}\ 3l\lambda_i\|U\|_{L^2(0,s)}^2\leq \nu\big(\|\partial_rU\|_{L^2(0,s)}^2 +\|U/r\|_{L^2(0,s)}^2
   +l^2\|U\|_{L^2(0,s)}^2\big).
\end{align*}
Then we infer from Lemma \ref{lem:la-l}  that 
\begin{align}\label{q2pros}
   &\nu\big(\|\partial_rU\|_{L^2(0,s)}^2+\|U/r\|_{L^2(0,s)}^2+l^2\|U\|_{L^2(0,s)}^2\big)+|l(\lambda-1)|\|U\|_{L^2(0,s)}^2\nonumber\\
   &\leq 2\nu s|\partial_rU(s)||U(s)|,
\end{align}
which gives, by taking $s=1$ using $\|U/r\|_{L^2}\ge\|U\|_{L^2}$, that
\begin{align}\label{q2pro}
   &\tilde{A}_1^{2}\|U\|_{L^2}^2+ \|2\partial_rU+U/r\|_{L^2}^2\leq C|\partial_rU(1)|.
\end{align}

We get by integration by parts again that
\begin{align*}
   0&=\big\langle -\nu\widehat{\Delta}_{(1)}U+{\ir}l(\lambda-1)U,W\big\rangle \\
  &=-\nu\|U\|_{L^2}^2+\nu \overline{\partial_rW(1)} -il(\lambda-1)\big(\|\partial_rW\|_{L^2}^2+\|W/r\|_{L^2}^2 +l^2\|W\|_{L^2}^2\big).
\end{align*}
Thanks to the fact that 
\begin{align*}
   \|\partial_rW\|_{L^2}^2+\|W/r\|_{L^2}^2 +l^2\|W\|_{L^2}^2&=|\langle U,W\rangle| \leq \|W\|_{L^2}\|U\|_{L^2} \\
   &\leq \sqrt{1+l^2}\|W\|_{L^2}(1+l^2)^{-\f12}\|U\|_{L^2},
\end{align*} 
we have $(1+l^2)\big(\|\partial_rW\|_{L^2}^2+\|W/r\|_{L^2}^2 +l^2\|W\|_{L^2}^2\big) \leq \|U\|_{L^2}^2$. Since $2l\lambda_i\leq |l(\lambda-1)|$ or $3l\lambda_i\leq \nu(1+l^2)$, we have
\begin{align*}
   &2l\lambda_i\leq |l(\lambda-1)|\ \text{or}\ 3l\lambda_i\big(\|\partial_rW\|_{L^2}^2+\|W/r\|_{L^2}^2 +l^2\|W\|_{L^2}^2\big)\leq \nu\|U\|_{L^2}^2.
\end{align*}
Then we get by Lemma \ref{lem:la-l}  that 
\begin{align*}
   &\|U\|_{L^2}^2+|l(\lambda-1)/\nu|\big(\|\partial_rW\|_{L^2}^2+\|W/r\|_{L^2}^2 +l^2\|W\|_{L^2}^2\big) \leq 2|\partial_rW(1)|,
\end{align*}
which along with  \eqref{q2pro}  gives 
\begin{align}\label{W2q2>1}
   1&=|U(1)|^2\leq \|r^{-1}\partial_r(rU^2)\|_{L^1}=\|(2\partial_rU+U/r)U\|_{L^2}\nonumber\\
   &\leq \|2\partial_rU+U/r\|_{L^2}\|U\|_{L^2}\leq C|\partial_rU(1)||\partial_rW(1)|.
\end{align}

We define 
\begin{align*}
   &\Psi_1(r)= -\nu(r\partial_rU)^2+\nu (r^2l^2+1)U^2 +{\ir}l(\lambda-1) r^2U^2,\\
   &\Psi_2(r)=2r\big(\nu l^2+{\ir}l(\lambda-1)\big)U^2.
\end{align*}
Notice that 
\begin{align*}
   0&= 2\big(-\nu\widehat{\Delta}_{(1)}U+{\ir}l(\lambda-1)U\big)r^2\p_rU\\
   &=-2\nu \partial_r(r\partial_rU)r\partial_rU +2\nu(r^2l^2+1)U\partial_rU +2{\ir}l(\lambda-1)r^2U\partial_rU\\
   &= -\nu \partial_r(r\partial_rU)^2 +\nu (r^2l^2+1)\partial_r(U^2)+{\ir}l(\lambda-1)r^2\partial_r(U^2).
\end{align*}
Then we have $\partial_r\Psi_1=\Psi_2$ and 
\begin{align*}
   \big|-\nu\big(\partial_rU(1)\big)^2+\nu (1+l^2)+&{\ir}l(\lambda-1)\big|=|\Psi_1(1)| \leq \int_{0}^{1}|\Psi_2(r)|\mathrm{d}r\\
   \leq &\int_{0}^{1}2\nu r\big(l^2+|l(\lambda-1)/\nu|\big)|U|^2\mathrm{d}r= 2\nu \tilde{A}_1^2\|U\|_{L^2}^2.
\end{align*}
which along with \eqref{q2pro} gives
\begin{align*}
   &|\partial_rU(1)|^2\leq \tilde{A}_1^2+ 2\tilde{A}_1^2\|U\|_{L^2}^2\le C\tilde{A}_1^2+1\leq C\tilde{A}_1^2.
   \end{align*}
Then we infer from \eqref{W2q2>1} and \eqref{q2pro}   that 
\beno
|\partial_rW(1)|\geq C^{-1}|\partial_rU(1)|^{-1}\geq C^{-1}\tilde{A}_1^{-1},\quad \|U\|_{L^2}^2\leq C\tilde{A}_1^{-1}.
\eeno 

It remains to  estimate $\|(1-r^2)U\|_{L^2}$. Let
\begin{align*}
   \Psi_3(s)=\nu\big(\|\partial_rU\|_{L^2(0,s)}^2+\|U/r\|_{L^2(0,s)}^2 +l^2\|U\|_{L^2(0,s)}^2\big)+|l(\lambda-1)|\|U\|_{L^2(0,s)}^2.
\end{align*}
Then $\partial_r\Psi_3(r)=\nu r\big(|\partial_rU|^2+ |U/r|^2+l^2|U|^2\big) +r|l(\lambda-1)| |U|^2\geq \nu r\big(|\partial_rU|^2+\tilde{A}_1^2|U|^2\big).$ By \eqref{q2pros}, we have
\begin{align*}
   & \Psi_3(r)\leq 2\nu r|\partial_rU(r)||U(r)| \leq \tilde{A}_1^{-1}\partial_r\Psi_3,\quad \Psi_3(1)\leq 2\nu|\partial_rU(1)|,
\end{align*}
which implies that 
\begin{align*}
   &\Psi_3(r)\leq \Psi_3(1)\mathrm{e}^{-(1-r)\tilde{A}_1}\leq 2\nu |\partial_rU(1)|\mathrm{e}^{-(1-r)\tilde{A}_1}\le C\nu \tilde{A}_1\mathrm{e}^{-(1-r)\tilde{A}_1},
 \end{align*}
and then
\begin{align*}
   &\|U\|_{L^2(0,s)}^2\leq C\nu^{-1}\Psi_3(s)\leq C\tilde{A}_1e^{-(1-s)\tilde{A}_1}.
\end{align*}
Thus, we conclude that 
\begin{align*}
   \|(1-r^2)U\|_{L^2}&\leq C\|(1-r)U\|_{L^2}=C\int_{0}^{1}2(1-s)\|U\|_{L^2(0,s)}\mathrm{d}s \\
   &\leq C\int_{0}^{s}(1-s)\tilde{A}_1^{-1}\mathrm{e}^{-(1-s)\tilde{A}_1}\mathrm{d}s \leq C\tilde{A}_1^{-3}.
\end{align*}

This completes the proof of the lemma.
\end{proof}\smallskip

Now we prove Proposition \ref{prop:OM1}. 

\begin{proof}
As in Proposition \ref{prop:Wa-lower}, it is enough to consider the following two cases. \smallskip

\no {\bf Case 1.}  $l\lambda_i\leq \max(|l(\lambda-1)|/2,\nu(1+l^2)/3)$ and $|\nu/l|^{1/3}\tilde{A}_1\geq 1/c_8$.\smallskip

 Let $(\Om_{1,a}, W_{1,a})$ solve \eqref{eq:LNS-OM-a} and $\Omega_{1,e}=\Omega_1-\Om_{1,a}$. Then $\Omega_{1,e}$ satisfies
  \begin{align*}
     &-\nu\widehat{\Delta}_{(1)}\Omega_{1,e}+{\ir}l(\lambda_r-r^2)\Omega_{1,e} =l\lambda_i\Omega_{1,e}-{\ir}l(1-r^2)\Om_{1,a},\quad \Omega_{1,e}(1)=0.
  \end{align*}
Let $W_{1,e}$ solve $-\widehat{\Delta}_{(1)}W_{1,e}=\Omega_{1,e},\ W_{1,e}(1)=0.$ Then we have
\begin{align*}
   &\nu\big(\|\partial_r\Omega_{1,e}\|_{L^2}^2+\|\Omega_{1,e}/r\|_{L^2}^2+ l^2\|\Omega_{1,e}\|_{L^2}^2\big)+{\ir}l\int_{0}^{1}r(\lambda_r-r^2)|\Omega_{1,e}|^2 \mathrm{d}r-l\lambda_i\|\Omega_{1,e}\|_{L^2}^2\\
   &=\big\langle -{\ir}l(1-r^2)\Om_{1,a}, \Omega_{1,e}\big\rangle \leq \|{\ir}l(1-r^2)\Om_{1,a}\|_{L^2}\|\Omega_{1,e}\|_{L^2}.
\end{align*}
For $l\lambda_i\leq 0$, we use the above inequality, and for $l\lambda_i\geq0$,  we use $|l\lambda_i|\leq c_8|\nu l|^{\f12}$.
We conclude that 
\begin{align}\label{eq:axs-est1}
   &|l\lambda_i|\|\Omega_{1,e}\|_{L^2}\leq c_8|\nu l|^{\f12}\|\Omega_{1,e}\|_{L^2}+\|{\ir}l(1-r^2)\Om_{1,a}\|_{L^2}.
\end{align} 

Using Proposition \ref{prop:ellip-inhom-v1} with $n=1, s=1$, we deduce that 
  \begin{align*}
     |\nu l|^{\f12}\|\Omega_{1,e}\|_{L^2}+\nu^{\f16}|l|^{\f56}\|r\Omega_{1,e}\|_{L^1}\leq& C\left\|l\lambda_i\Omega_{1,e}-{\ir}l(1-r^2)\Om_{1,a}\right\|_{L^2}\\
     \leq&C\big( c_8|\nu l|^{\f12}\|\Omega_{1,e}\|_{L^2}+\|{\ir}l(1-r^2)\Om_{1,a}\|_{L^2}\big).
  \end{align*}
  By Lemma \ref{lem:UW-app-axs}, we have $\|{\ir}l(1-r^2)U_{1,a}\|_{L^2}\leq C|l|\tilde{A}_1^{-\f32}.$ Taking $c_8$ sufficiently small so that $Cc_8\leq 1/2$, we conclude that 
  \begin{align}\label{rOm2}
     \nu^{\f16}|l|^{\f56}\|r\Omega_{1,e}\|_{L^1}&\leq C|l|\tilde{A}_{1}^{-\f32}.
  \end{align}
Let $J$ be as in Remark \ref{rem:J}. Then we infer that
  \begin{align*}
  |\partial_rW_{1,e}(1)|=|\langle J, \Omega_{1,e}\rangle|\leq \|r\Omega_{1,e}\|_{L^1}\leq C|\nu/l|^{-\f16}\tilde{A}_1^{-\f32}.
  \end{align*}
  Then we have
  \begin{align*}
     |\partial_rW_1(1)|&\geq |\partial_rW_{1,a}(1)|-|\partial_rW_{1,e}(1)| \geq C^{-1}\tilde{A}_1^{-1}- C|\nu/l|^{-\f16}\tilde{A}_1^{-\f32} \\
     &\geq C^{-1}\tilde{A}_{1}^{-1}\big(1-C(|\nu/l|^{\f13}\tilde{A}_{1})^{-\f12}\big) \geq C^{-1}\tilde{A}_{1}^{-1}\big(1-C(c_8)^{\f12}\big),
  \end{align*}
 which gives  $|\partial_rW_1(1)|\ge C^{-1}\tilde{A}_{1}^{-1}$  by taking $c_8$ sufficiently small.\smallskip

 \no{\bf Case 2}.  $l\lambda_i\ge \max(|l(\lambda-1)|/2,\nu (n^2+l^2)/3)$. In this case, we have $|\nu/l|^{1/3}\tilde{A}_1\le 1/c_8$.\smallskip
 
 Let $w(r)$ be defined in subsection \ref{sec:airy} with $n=1$, which solves
\begin{align*}
&-\nu\partial_r^2w+\nu(1+l^2)w+{\ir}l(\lambda-2r+1)w=0,\quad w(1)=1.
\end{align*}
A direct calculation gives
\begin{align*}
   &\widehat{\Delta}_{(1)}(r^2w)=\dfrac{\partial_r(r\partial_r(r^2w))}{r}-(1+l^2r^2)w=4w+5r\partial_rw+r^2\partial_r^2w-(1+l^2r^2)w.
\end{align*}
Then we have
\begin{align*}
   &-\nu\widehat{\Delta}_{(1)}(r^2w)+{\ir}l(\lambda-r^2)(r^2w)=F,
\end{align*}
where $F=-\nu (3w+5r\partial_rw+r^2w)-{\ir}lr^2(1-r)^2w$.
Let $\Om_{1,e}=\Omega_1-r^2w$ and $\widehat{\Delta}_{(1)}W_{1,e}=\Om_{1,e},\ W_{1,e}(1)=0$. Then we have
\begin{align*}
   & -\nu\widehat{\Delta}_{(1)}\Om_{1,e}+{\ir}l(\lambda-r^2)\Om_{1,e}=-F,\quad \Om_{1,e}(1)=0.
\end{align*}
As in  \eqref{eq:axs-est1}, we have
\begin{align*}
   &|l\lambda_i|\|\Om_{1,e}\|_{L^2}\leq c_8|\nu l|^{\f12}\|\Om_{1,e}\|_{L^2}+\nu\|F\|_{L^2},
   \end{align*}
and then by Proposition \ref{prop:ellip-inhom-v1} with $n=1, s=1$, we have
\begin{align*}
   & |\nu l|^{\f12}\|\Om_{1,e}\|_{L^2}+\nu^{\f16}|l|^{\f56}\|r\Om_{1,e}\|_{L^1}\leq C\big(c_8|\nu l|^{\f12}\|\Om_{1,e}\|_{L^2}+\nu\|F\|_{L^2}\big),
\end{align*}
Taking $c_8$ sufficiently small, we obtain
\begin{align}\label{q4L1}
   & \|r\Om_{1,e}\|_{L^1}\leq \nu^{\f56}|l|^{-\f56}\|F\|_{L^2}.
\end{align}
Without loss of generality, we assume $l>0$. In present case, we have $|Ld|\le c_8^{-2}$ and  $\mathbf{Im}(Ld)\leq c_8^{\f12}$. 
Taking $c_8$ sufficiently small so that $c_8^\f12\le \delta_0$(now we fix $c_8$), we can apply Lemma \ref{lem:Airy-w} to obtain  
\begin{align*}
   \|F\|_{L^2} &\leq C\nu\big(\|w\|_{L^2}+\|\partial_rw\|_{L^2}\big)+C|l|\|(1-r)^2w\|_{L^2}\\
   &\leq C\nu\big(L^{-\f12}+L^{\f12}\big) +C|l|L^{-\f52}\leq C(\nu L^{\f12}+|l| L^{-\f52})=C\nu^{\f56}|l|^{\f16}.
\end{align*}
Then we infer from Remark \ref{rem:J} and  \eqref{q4L1} that
\begin{align}\label{paW41}
   &|\partial_rW_{1,e}(1)|=|\langle \Om_{1,e}, J \rangle|\leq \|r\Om_{1,e}\|_{L^1}\leq C(\nu^{-\f16}|l|^{-\f56})(\nu^{\f56}|l|^{\f16})=C|\nu/l|^{\f23}.
\end{align}

Let $W_{1,a}$ solve $\widehat{\Delta}_{(1)}W_{1,a}=r^2w,\ W_{1,a}(1)=0$. Then we find that
\begin{align*}
   \partial_rW_{1,a}(1)&=\langle r^2w,J\rangle=\int_{0}^{1}r^3wJ\mathrm{d}r =\int_{0}^{1}r^3\dfrac{A_0'(Ld+1-r)}{A_0'(Ld)}J\mathrm{d}r\\
   &=-\dfrac{A_0(Ld)}{LA'_0(Ld)}+\int_{0}^{1}\dfrac{A_0(L(d+1-r))}{LA_0'(Ld)} \partial_r\big(r^3J\big)\mathrm{d}r
\end{align*}
Thanks to $\partial_rJ\geq 0$, we get by Lemma \ref{lem:Airy-w}  that
\begin{align*}
   L|A_0'(Ld)||\partial_rW_5(1)|&\geq |A_0(Ld)|-\int_{0}^{1}|A_0(L(d+1-r))|\partial_r(r^3J)\mathrm{d}r\\
   &\geq  |A_0(Ld)|-\int_{0}^{1}|A_0(Ld)|\mathrm{e}^{-L(1-r)/3}\partial_r(r^3J)\mathrm{d}r\\
   &= |A_0(Ld)|\dfrac{L}{3}\int_{0}^{1}\mathrm{e}^{-L(1-r)/3}r^3J\mathrm{d}r,
\end{align*}
which gives 
\beno
|\partial_rW_{1,a}(1)|\geq \dfrac{|A_0(Ld)|}{3|A'_0(Ld)|}\int_{0}^{1}\mathrm{e}^{-L(1-r)/3}r^3J\mathrm{d}r.
\eeno
Remark \ref{rem:J} gives
\begin{align*}
   & \int_{0}^{1}\mathrm{e}^{-L(1-r)/3}r^3J\mathrm{d}r \geq \int_{0}^{1}\mathrm{e}^{-L(1-r)/3}r^{4+|l|}\mathrm{d}r\geq \int_{0}^{1}r^{L/3}r^{4+|l|}\mathrm{d}r\geq \dfrac{1}{5+|l|+L/3}.
\end{align*}
Since $|Ld|\leq c_8^{-2}$ and $|\nu/l|^{1/3}\tilde{A}_1\le 1/c_8$, we have $\dfrac{|A_0(Ld)|}{3|A'_0(Ld)|}\geq C^{-1}$ and $5+|l|\leq C(1+l^2)^{\f12}\leq CL$.
Then we conclude that 
\begin{align*}
   &|\partial_rW_{1,e}(1)|\geq \dfrac{|A_0(Ld)|}{3|A'_0(Ld)|}\int_{0}^{1}\mathrm{e}^{-L(1-r)/3}r^3J\mathrm{d}r \geq \dfrac{1}{CL}.
\end{align*}
from which and  \eqref{paW41},  we deduce by taking $c_9$ small enough that 
\begin{align*}
   |\partial_rW_1(1)|&\geq |\partial_rW_{1,a}(1)|-|\partial_rW_{1,e}(1)|\geq C^{-1}L^{-1}-C|\nu/l|^{\f23}\\
   &=C^{-1}L^{-1}-CL^{-2}\geq C^{-1}L^{-1}.
\end{align*}

This completes the proof of the proposition.
\end{proof}

\section{Linear stability for general perturbations} 

In this section, we prove Theorem \ref{thm:stability}. 
Let $0\neq u\in D(\mathcal{L}_\nu)$ solve the resolvent equation 
\beno
su=\mathcal{L}_\nu u.
\eeno 
Our goal is to show that 
\ben\label{eq:spectral bound}
m_0(\nu)\leq -c_0\nu,\quad m_1(\nu)\leq -c_0\nu.
\een

In the sequel, we denote by $\|\cdot\|_{L^2}$ and $\langle \cdot, \cdot\rangle$ the $L^2(\Om)$ norm and inner product in $L^2(\Om)$ respectively. 

\subsection{Upper bound of $m_0(\nu)$}

If $0\neq {u}=(u^1,u^2,u^3)\in D(\mathcal{L}_\nu)\cap\mathcal{H}_0 $, then $\mathbb{P}\big((0,0, {u}\cdot\nabla V)+V\partial_z{u}\big)=\mathbb{P}(0,0,{u}\cdot\nabla V).$ If $(u^1,u^2)\neq 0$, then we have 
\begin{align*}
\mathbf{Re}(s)\|(u^1,u^2)\|_{L^2}^2=&\mathbf{Re}\langle s{u},(u^1,u^2,0)\rangle=\mathbf{Re}\big\langle \mathcal{L}_v{u},(u^1,u^2,0)\big\rangle\\=&\mathbf{Re}\big\langle \nu\mathbb{P}\Delta {u}-\mathbb{P}(0,0,{u}\cdot\nabla V),(u^1,u^2,0)\big\rangle\\
=&\mathbf{Re}\big\langle \nu\Delta {u}-(0,0, {u}\cdot\nabla V),(u^1,u^2,0)\big\rangle\\=&\mathbf{Re}\big\langle \nu\Delta {u},(u^1,u^2,0)\big\rangle=-\mathbf{Re}\big\langle \nu\nabla {u},\nabla(u^1,u^2,0)\big\rangle\\=&-\nu\|\nabla(u^1,u^2)\|_{L^2}^2\leq -C^{-1}\nu\|(u^1,u^2)\|_{L^2}^2.
\end{align*}
Here we used Poincare's inequality. Thus, $\mathbf{Re}(s)\leq -C^{-1}\nu.$

If $(u^1,u^2)= 0$ and $u^3\neq 0$, then we have  $\mathcal{L}_\nu{u}=\nu\mathbb{P}\Delta {u}=\mathbb{P}(0,0,\nu\Delta {u}^3)=(0,0,\nu\Delta {u}^3)$. Thus, $su^3=\nu\Delta {u}^3,\ u^3|_{r=1}=0$. Then we deduce that 
\begin{align*}
&\mathbf{Re}(s)\|u^3\|_{L^2}^2=\mathbf{Re}\langle s {u}^3,u^3\rangle=\mathbf{Re}\langle \nu\Delta{u}^3,u^3\rangle=-\nu\|\nabla u^3\|_{L^2}^2\leq -C^{-1}\nu\|u^3\|_{L^2}^2.
\end{align*}
Thus, $\mathbf{Re}(s)\leq -C^{-1}\nu.$ Combining two cases,  we conclude the upper bound of $m_0(\nu).$ 

\subsection{Upper bound of $m_1(\nu)$}
We introduce $W=(x_1,x_2,0)\cdot u$. First of all, we consider the case of $W=0$. In this case, we have
\beno
 {u}\cdot\nabla V=0,\quad \langle \na P, {u}\rangle=0,\quad \langle V\partial_z{u}, {u}\rangle=0,
\eeno
 which show that 
 \begin{align*}
\mathbf{Re}(s)\|{u}\|_{L^2}^2=&\mathbf{Re}\langle s{u},{u}\rangle=\mathbf{Re}\big\langle \nu\Delta {u}-V\partial_z{u}-(0,0, {u}\cdot\nabla V)-\na P, {u}\big\rangle\\
=&\mathbf{Re}\langle \nu\Delta {u}, {u}\rangle=-\nu\|\nabla {u}\|_{L^2}^2\leq -C^{-1}\nu\|{u}\|_{L^2}^2.
\end{align*}
This shows that $\mathbf{Re}(s)\leq -C^{-1}\nu.$


Next we consider the case of $W\neq 0$.  We derive the following system of $\widehat{W}(n,l), \widehat{U}(n,l) $  in section 2(still denote $\widehat{W}, \widehat{U} $ by $W, U$): 
 \begin{align}\label{eq:LNS-WU-proof}
\left\{\begin{aligned}
&-\nu\big(2\widehat{\Delta}{U}-\widehat{\Delta}_1^*{W}_1\big)+{\ir}l(\lambda-r^2){U}=0,\\
&-\nu \widehat{\Delta}_1{U}+{\ir}l(\lambda-r^2){W}_1+\dfrac{4{\ir}n^2l{W}}{n^2+r^2l^2}
=0,\\
&W_1=\widehat{\Delta}_1{W},\quad {{W}}|_{r=1}=\partial_r{W}|_{r=1}=0,\quad {W}_1|_{r=1}={U}|_{r=1},
\end{aligned}\right. 
\end{align}
where $\la=s/({\ir}l)+1$.

Assume  $0<\nu<c_0\min(|l|,1)$. Let us first consider the case of $n\neq 0$ and $l\neq0$. If ${W}_1|_{r=1}={U}|_{r=1}=0$, Proposition \ref{prop:res-com} ensures
that $(W,U)=0$ in the case when $l\lambda_i\leq c_6|\nu n l|^{\f12}$. If ${W}_1|_{r=1}={U}|_{r=1}\neq 0$, we can normalize them so that ${W}_1|_{r=1}={U}|_{r=1}=1$. Then Proposition \ref{prop:lower} ensures that  if $l\lambda_i\leq c_6|\nu n l|^{\f12}$ and $0<\nu<c_6\min(|l|,1)$,  then  $\pa_rW(1)\neq 0$, which contradicts with $\pa_rW(1)=0$. Thus, we also have 
$W=U=0$.


For the case of  $n=0$ and $l\neq 0$, we  have  ${U}={W}_1$ and $(\Omega_1,W_2)=\big({W}_1/r, {W}/r\big)$ satisfies  
\begin{align*}
  \left\{\begin{aligned}
     &-\nu\widehat{\Delta}_{(1)}\Omega_1 +{\ir}l(\lambda-r^2)\Omega_1=0,\\    
     &\Omega_1=\widehat{\Delta}_{(1)}W_2,\quad W_2|_{r=1}=\partial_rW_2|_{r=1}=0.
  \end{aligned}\right.
  \end{align*}
In last section, we have showed that $\Om_1=W_2=0$ in the case when $l\lambda_i\leq c_8|\nu n l|^{\f12}$. 

Thanks to $\la=s/({\ir}l)+1$, we have $W=0$ when $\mathbf{Re s}\ge -c|\nu/L_z|^\f12$.
This in turn shows that $m_1(\nu)\leq -c|\nu/L_z|^{\f12}\le -c\nu$.\smallskip

Finally, let us point out that  our  argument of linear stability does not need to prove the existence of the solution for the system \eqref{eq:LNS-WU-ar-3}  and \eqref{eq:LNS-hom-3}. It is enough to prove the existence of the solution for the following linear elliptic type equations(in fact, ODE)  such as 
\begin{align*}
&-\nu \widehat{\Delta}_1{U}+{\ir}l(\lambda-r^2)U=0,\quad U(0)=0,\,\,U(1)=1,\\
 &-\nu\widehat{\Delta}_{(1)}U+\ir l(\lambda-1)U=0,\quad U(0)=0,\,\,U(1)=1,\\
&\widehat{\Delta}_1{W}=U,\quad W(0)=0,\,\,W(1)=0,\\
&\widehat{\Delta}_{(1)}W=U,\quad W(0)=0,\,\,W(1)=0.
\end{align*}
The proof is not  hard and we will write a proof in a note.

\section{Appendix}

\subsection{Hardy type inequalities}

\begin{Lemma}\label{lem:hardy-1}
  If $\lambda\geq1$ and $f(1)=0$, then we have
  \begin{align*}
    \|f\|_{L^2} &\leq C\lambda^{-\f13}\|f\|_{1}^{\f13}\big\|(\lambda-r^2)^{\f12}f\big\|_{L^2}^{\f23}.
  \end{align*}
\end{Lemma}

\begin{proof}
 By Hardy's inequality under the measure $\mathrm{d}r$, we have
  \begin{align*}
     \left\|\dfrac{f}{1-r^2}\right\|_{L^2(r\mathrm{d}r)}&\leq \left\|\dfrac{f}{1-r^2}\right\|_{L^2((0,1/2);r\mathrm{d}r)}+\left\|\dfrac{f}{1-r^2}\right\|_{L^2([1/2,1];r\mathrm{d}r)}\\
     &\leq \|f\|_{L^2((0,1/2);r\mathrm{d}r)} \left\|\dfrac{1}{1-r^2}\right\|_{L^\infty((0,1/2);r\mathrm{d}r)}+ \sqrt{2}\left\|\dfrac{f}{1-r^2}\right\|_{L^2([1/2,1],\mathrm{d}r)}\\
     &\leq C\big(\|f\|_{L^2(r\mathrm{d}r)}+\|\partial_rf\|_{L^2([1/2,1],\mathrm{d}r)}\big)\\
     &\leq C\big(\|f\|_{L^2(r\mathrm{d}r)}+\|\partial_rf\|_{L^2(r\mathrm{d}r)}\big)\leq C\|f\|_{1} .
  \end{align*}
On the other hand, we have
  \begin{align*}
    \big\|(1-r^2)^{\f12}f\big\|_{L^2}\leq \big\|(\lambda-r^2)^{\f12}f\big\|_{L^2}\left\|\dfrac{1-r^2}{\lambda-r^2}\right\|^{\f12}_{L^\infty}\leq \lambda^{-\f12}\big\|(\lambda-r^2)^{\f12}f\big\|_{L^2}.
  \end{align*}
Thus, we conclude  that
  \begin{align*}
     \|f\|_{L^2}&\leq \left\|\dfrac{f}{1-r^2}\right\|_{L^2}^{\f13}\big\|(1-r^2)^{\f12}f\big\|^{\f23}_{L^2} \leq C\lambda^{-\f13}\|f\|_{1}^{\f13}\big\|(\lambda-r^2)^{\f12}f\big\|_{L^2}^{\f23}.
  \end{align*}
  
  This completes the proof of the lemma.
\end{proof}

\begin{Lemma}\label{lem:hardy-2}
If $\widetilde{r}\in(0,1]$ and  $f(\widetilde{r})=0$, then we have
  \begin{align*}
& \left\|\dfrac{f/(\widetilde{r}^2-r^2)}{n^2+r^2l^2}\right\|_{L^{2}}^2\leq \frac{C}{\widetilde{r}^2(n^2+\widetilde{r}^2l^2)}\int_0^1\left(\dfrac{r|\partial_rf|^2}{n^2+r^2l^2}+\dfrac{|f|^2}{r}\right)\mathrm{d}r.
\end{align*}
\end{Lemma}

\begin{proof}
If $f(\widetilde{r})=0$, then we  have
  \begin{align*}
     \left\|\f{f/(\widetilde{r}^2-r^2)}{n^2+r^2l^2}\right\|_{L^2}^2 &\leq  \left\|\f{f/\big((r-\widetilde{r})(r+\widetilde{r})\big)}{n^2+r^2l^2}\right\|_{L^2}^2\leq \f{1}{\widetilde{r}^2} \left\|\f{f/(r-\widetilde{r})}{n^2+r^2l^2}\right\|_{L^2}^2\\
     &\leq \f{2}{\widetilde{r}^2}\bigg(\left\|\f{f/(r-\widetilde{r})}{n^2+r^2l^2} -\f{f/(r-\widetilde{r})}{(n^2+r^2l^2)^{\f12}(n^2+\widetilde{r}^2l^2)^{\f12}}\right\|_{L^2}^2 \\
     &\qquad\quad+\left\|\f{f/(r-\widetilde{r})}{(n^2+\widetilde{r}^2l^2)^{\f12}(n^2+r^2l^2)^{\f12}}\right\|_{L^2}^2\bigg).
  \end{align*}
Thanks to 
  \begin{align*}
     &\left|\f{1}{r-\widetilde{r}}\bigg(\f{1}{(n^2+r^2l^2)^{\f12}}-\f{1}{(n^2+\widetilde{r}^2l^2)^{\f12}}\bigg)\right| \\&=\left|\f{(r+\widetilde{r})l^2}{((n^2+r^2l^2)^{\f12}+(n^2+\widetilde{r}^2l^2)^{\f12}) (n^2+r^2l^2)^{\f12}(n^2+\widetilde{r}^2l^2)^{\f12}}\right|\\
     &\leq \f{1}{r(n^2+\widetilde{r}^2l^2)^{\f12}},
  \end{align*}
we find that
\begin{align*}
   &\left\|\f{f/(r-\widetilde{r})}{n^2+r^2l^2} -\f{f/(r-\widetilde{r})}{(n^2+r^2l^2)^{\f12}(n^2+r^2_0l^2)^{\f12}}\right\|_{L^2}^2 \leq \f{1}{n^2+\widetilde{r}^2l^2}\int_{0}^{1}\f{|f|^2}{r(n^2+r^2l^2)}dr.
\end{align*}
Then by Hardy's inequality and  $f(\widetilde{r})=0$, we have
\begin{align*}
  \left\|\f{f/(r-\widetilde{r})}{(n^2+\widetilde{r}^2l^2)^{\f12}(n^2+r^2l^2)^{\f12}}\right\|_{L^2}^2 &=\f{1}{n^2+\widetilde{r}^2l^2}\int_{0}^{1}\f{1}{(r-\widetilde{r})^2}\left| \f{r^{\f12}f}{(n^2+r^2l^2)^{\f12}}\right|^2dr \\ &=\f{1}{n^2+\widetilde{r}^2l^2}\int_{0}^{1}\f{1}{(r-\widetilde{r})^2}\left| \f{r^{\f12}f}{(n^2+r^2l^2)^{\f12}}-\f{\widetilde{r}^{\f12}f(\widetilde{r})}{(n^2+\widetilde{r}^2l^2)^{\f12}}\right|^2dr \\&\leq \f{C}{n^2+\widetilde{r}^2l^2}\int_{0}^{1}\left|\partial_r\bigg[ \f{r^{\f12}f}{(n^2+r^2l^2)^{\f12}}\bigg]\right|^2dr\\&\leq \f{C}{n^2+\widetilde{r}^2l^2}\int_{0}^{1}\bigg(\f{|f|^2}{r(n^2+r^2l^2)}+\f{ r|\partial_rf|^2}{n^2+r^2l^2}\bigg)dr.
\end{align*}
Thus, we conclude that
\begin{align*}
   \left\|\f{f/(\widetilde{r}^2-r^2)}{n^2+r^2l^2}\right\|_{L^2}^2&\leq\f{C}{n^2+\widetilde{r}^2l^2} \int_{0}^{1}\bigg(\f{|f|^2}{r(n^2+r^2l^2)}+\f{ r|\partial_rf|^2}{n^2+r^2l^2}\bigg)dr\\
   &\leq \f{C}{n^2+\widetilde{r}^2l^2}\int_{0}^{1}\bigg(\f{|f|^2}{r}+\f{ r|\partial_rf|^2}{n^2+r^2l^2}\bigg)dr.
\end{align*}

This completes the proof of the lemma. 
\end{proof}

\subsection{Interpolation inequality} 

\begin{Lemma}\label{lem:interp}
Assume that $\lambda>0,\ s>0,\ g\in H_0^1(0,s)$ and let $r_0=\lambda^{\f12}$. It holds that
  \begin{align*}
    &\|g\|_{L^{1}}\leq 2\|g\|_{L^2}^{\f12}\|(\lambda-r^2)g\|_{L^2}^{\f12},\\
    &r_0\|{g}\|_{L^2}^2\leq C \|(\lambda-r^2){g}\|_{L^2}\|g\|_{1}.
  \end{align*}
\end{Lemma}

\begin{proof}
Without loss of generality, we may assume that $\|g\|_{L^2}\neq 0$ and $\|(\lambda-r^2)g\|_{L^2}\neq0$. Let $E_1=\{r\in[0,s]:|\lambda-r^2|\leq \tilde{\delta}\}$, $E_1^c=[0,s]\setminus E_1$ for $\tilde{\delta}>0$. Then we have
\begin{align*}
   &\|1\|_{L^2(E_1)}=\bigg(\int_{r\in E_1}r\mathrm{d}r\bigg)^{\f12}\leq \bigg(\int_{|r^2-\lambda|\leq \tilde{\delta},r>0}\mathrm{d}r^2/2\bigg)^{\f12}\leq \tilde{\delta}^{\f12},
\end{align*}
and
\begin{align*}
   \|1/(\lambda-r^2)\|_{L^2(E_1^c)}&=\bigg(\int_{r\in E_1^c}r/(\lambda-r^2)^2\mathrm{d}r\bigg)^{\f12}\\
   &\leq \bigg(\int_{|r^2-\lambda|> \tilde{\delta},r>0}1/(\la-r^2)^2\mathrm{d}r^2/2\bigg)^{\f12}\leq \tilde{\delta}^{-\f12}.
\end{align*}
Thus, we obtain
\begin{align*}
   \|g\|_{L^1}&\leq \|g\|_{L^1(E_1)}+\|g\|_{L^1(E_1^c)}\\
   &\leq \|g\|_{L^2}\|1\|_{L^2(E_1)}+\|(\lambda-r^2)g\|_{L^2}\|1/(\lambda-r^2)\|_{L^2(E_1^c)}
   \\&\leq \tilde{\delta}^{\f12}\|g\|_{L^2}+\tilde{\delta}^{-\f12}\|(\lambda-r^2)g\|_{L^2},
\end{align*}
which gives the first inequality by taking $\tilde{\delta}=\|(\lambda-r^2)g\|_{L^2}^{\f12}\|g\|_{L^2}^{-\f12}$.

We get by integration by parts that
  \begin{align*}
  2\mathbf{Re}\big\langle (\lambda-r^2){g},r\partial_r {g}/(r+r_0)\big\rangle&=\big\langle r(\lambda-r^2)/(r+r_0),\partial_r |{g}|^2\big\rangle=\big\langle r(r_0-r),\partial_r |{g}|^2\big\rangle\\&=-\big\langle r^{-1}\partial_r[r^2(r_0-r)], |{g}|^2\big\rangle=-\big\langle 2r_0-3r, |{g}|^2\big\rangle,
\end{align*}
which gives
 \begin{align*}
 \mathbf{Re}\big\langle (\lambda-r^2){g},(2r\partial_r {g}+3g)/(r+r_0)\big\rangle&=-\big\langle 2r_0-3r, |{g}|^2\big\rangle+3\big\langle (\lambda-r^2)/(r+r_0), |{g}|^2\big\rangle\\&=\big\langle 3r-2r_0, |{g}|^2\rangle+3\langle (r_0-r), |{g}|^2\big\rangle=r_0\|{g}\|_{L^2}^2,
\end{align*}
from which, we infer that
\begin{align*}
  r_0\|{g}\|_{L^2}^2\leq& \|(\lambda-r^2){g}\|_{L^2}\|(2r\partial_r {g}+3g)/(r+r_0)\|_{L^2}\\ \leq& \|(\lambda-r^2){g}\|_{L^2}\big(2\|\partial_r {g}\|_{L^2}+3\| {g}/r\|_{L^2}\big)\leq C \|(\lambda-r^2){g}\|_{L^2}\|g\|_{1}.
\end{align*}
This proves the second inequality of the lemma.
\end{proof} 

\subsection{Bounds on  harmonic functions} 

\begin{Lemma}\label{lem:har-bound}
  Let $J(r)$ solve $\widehat{\Delta}_1J=0,\ J(0)=0, J (1)=1$. Then $J(r)$ is increasing in $(0,1]$ and there holds
  \begin{align*}
   r^{|n|}\mathrm{e}^{-|l|(1-r)}\le r^{|n|+|l|}\le J(r)\le r^{|n|}.
\end{align*}
\end{Lemma}

\begin{proof}
We first prove that $J\geq 0$. Let $r_1\in[0,1]$ so that $J(r_1)=\min\{J\}$. If $J(r_1)<0$, then $r_1\neq 0,1$, and then $\partial_r^2J(r_1)\geq 0$, $\partial_rJ(r_1)=0$. On the other hand, 
\begin{align*}
   0&=\widehat{\Delta}_1J(r_1) =\partial_r^2 J(r_1)+ \frac{\partial_r J(r_1)}{r_1}-\frac{n^2+r_1^2l^2}{r_1^2}J(r_1) -\dfrac{2r_1l^2}{n^2+r^2_1l^2}\partial_r J(r_1)\\
   &=\partial_r^2 J(r_1)-\frac{n^2+r_1^2l^2}{r_1^2} J(r_1)>0,
\end{align*}
which is a contradiction. So,  $J\geq 0$.

Let $ \ph(r)=r^{|n|}\mathrm{e}^{-|l|(1-r)}$. Then $r\partial_r\ph=(|n|+r|l|)\ph $ and
\begin{align*}\widehat{\Delta}_1\ph&=\frac{\partial_r(r\partial_r\ph)}{r}-\frac{n^2+r^2l^2}{r^2}\ph-\dfrac{2rl^2}{n^2+r^2l^2}\partial_r\ph\\&
=\frac{\partial_r((|n|+r|l|)\ph)}{r}-\frac{n^2+r^2l^2}{r^2}\ph-\dfrac{2l^2(|n|+r|l|)\ph}{n^2+r^2l^2}\\&
=\frac{(|n|+r|l|)^2\ph}{r^2}+\frac{|l|\ph}{r}-\frac{n^2+r^2l^2}{r^2}\ph-\dfrac{2l^2(|n|+r|l|)\ph}{n^2+r^2l^2}\\&
=\frac{2|n|r|l|\ph}{r^2}+\frac{|l|\ph}{r}-\dfrac{2l^2(|n|+r|l|)\ph}{n^2+r^2l^2}\\&
=\frac{(2|n|+1)|l|\ph}{r}-\dfrac{2l^2(|n|+r|l|)\ph}{n^2+r^2l^2}\\&
=\frac{2(|n|-1)|l|\ph}{r}+\dfrac{|l|(3(n^2+r^2l^2)-2r|l|(|n|+r|l|))\ph}{r(n^2+r^2l^2)}\\&
=\frac{2(|n|-1)|l|\ph}{r}+\dfrac{|l|(2n^2+(|n|-r|l|)^2)\ph}{r(n^2+r^2l^2)}\geq0,
\end{align*}
which implies
\begin{align*}
0&\leq\dfrac{r(J\widehat{\Delta}_1\ph-\ph\widehat{\Delta}_1 J)}{n^2+r^2l^2} =J\bigg(\partial_r\dfrac{r\partial_r\ph}{n^2+r^2l^2}- \dfrac{\ph}{r} \bigg)-\ph\bigg(\partial_r\dfrac{r\partial_rJ}{n^2+r^2l^2}- \dfrac{J}{r} \bigg) \\
 &=J\partial_r\dfrac{r\partial_r\ph}{n^2+r^2l^2}-\ph\partial_r\dfrac{r\partial_rJ}{n^2+r^2l^2} =\partial_r\dfrac{r(J\partial_r\ph-\ph\partial_rJ)}{n^2+r^2l^2}.
\end{align*}
So, $\dfrac{r(J\partial_r\ph-\ph\partial_rJ)}{n^2+r^2l^2}$ is increasing, then $\dfrac{r(J\partial_r\ph-\ph\partial_rJ)}{n^2+r^2l^2}\geq 0$, hence \begin{align*}
   & J\partial_r\ph-\ph\partial_rJ\geq0.
\end{align*}
Then we infer that
\begin{align*}0\leq r(J\partial_r\ph-\ph\partial_rJ)=-r\ph^2\partial_r(J/\ph).
\end{align*}
Thus, $J/\ph$ is decreasing in $(0,1]$, $J(r)/\ph(r)\geq J(1)/\ph(1)=1,\ J(r)\geq\ph(r)=r^{|n|}\mathrm{e}^{-|l|(1-r)}$ for $r\in(0,1]$.

Let $\varphi_1=r^{|n|}$. Then we have
\begin{align*}
   \widehat{\Delta}_1\varphi_1&=-l^2r^{|n|}-\dfrac{2|n|l^2r^{|n|}}{n^2+r^2l^2}\leq 0,
\end{align*}
which implies
\begin{align*}
0&\geq\dfrac{r(J\widehat{\Delta}_1\varphi_1-\varphi_1\widehat{\Delta}_1J)}{n^2+r^2l^2} =\partial_r\dfrac{r(J\partial_r\varphi_1-\varphi_1\partial_rJ)}{n^2+r^2l^2}.
\end{align*}
So, $\dfrac{r(J\partial_r\varphi_1-\varphi_1\partial_rJ)}{n^2+r^2l^2}$ is decreasing, then $\dfrac{r(J\partial_r\varphi_1-\varphi_1\partial_rJ)}{n^2+r^2l^2}\leq 0$, hence,
 \begin{align*}
 J\partial_r\varphi_1-\varphi_1\partial_rJ\leq 0.
\end{align*}
Thus, we infer that $0\leq r(J\partial_r\varphi_1-\varphi_1\partial_rJ)=-r\varphi_1^2\partial_r(J/\ph).$ Then $J/\ph$ is decreasing in $(0,1]$, $J(r)/\varphi_1(r)\geq J(1)/\varphi_1(1)=1,\ J(r)\leq\varphi_1(r)=r^{|n|}$ for $r\in(0,1]$.

Finally, we prove that $J$ is increasing. By the continuity of $J$, we have that $\lim_{r\rightarrow 0^+}{J}'(r)={J}'(0)= \lim_{r\rightarrow 0^+}J(r)/r\geq \lim_{r\rightarrow 0^+}r^{|n|+|l|}/r=0$. Thanks to 
\begin{align*}
   0=& \widehat{\Delta}_{1}J=\dfrac{n^2+r^2l^2}{r}\left(\partial_r\dfrac{r\partial_rJ}{n^2+r^2l^2} -\dfrac{J}{r}\right),
\end{align*}
we have
\begin{align*}
   & \partial_r\dfrac{r\partial_rJ}{n^2+r^2l^2}=\dfrac{J}{r}\geq r^{|n|+|l|-1}> 0,\quad r\in(0,1].
\end{align*}
Then $\dfrac{r\partial_rJ}{n^2+r^2l^2}$ is increasing and  $\dfrac{r\partial_rJ}{n^2+r^2l^2}>\lim_{r\rightarrow 0^+}\dfrac{r\partial_rJ}{n^2+r^2l^2} \geq 0,\ r\in(0,1]$, which means that $\partial_rJ>0, r\in(0,1]$. That is, $J$ is increasing. 
\end{proof}

\begin{remark} 
The estimate of $J$ can be proved via the explicit formula. 
Let  $I=r\partial_r J/(n^2+r^2l^2)$. Thanks to 
\begin{align*}
   & \dfrac{r\widehat{\Delta}_1J}{n^2+r^2l^2}= \partial_r\left(\dfrac{r\partial_r J}{n^2+r^2l^2}\right)-\dfrac{ J}{r}=0,
\end{align*}
we find that $ J=r\partial_rI$ and
\begin{align*}
   &r\widehat{\Delta} I= \partial_r(r\partial_rI)-\dfrac{n^2+r^2l^2}{r}I= \partial_r J-\partial_r J=0\Longrightarrow \widehat{\Delta}I=0.
 \end{align*}
Let $t=irl$. Then we have
\begin{align*}
   &\big(t\partial_t(t\partial_t)+(t^2-n^2)\big)I=0.
\end{align*}
This is exactly the Bessel equation, whose solution is given by\begin{align*}
   &I(r)=c_1\sum_{k=0}^{+\infty}(-1)^{k}\dfrac{t^{|n|+2k}|n|!}{2^{ 2k}k!(|n|+k)!}=c\sum_{k=0}^{+\infty}\dfrac{(rl)^{|n|+2k}|n|!}{2^{ 2k}k!(|n|+k)!}=cI_n(rl).
\end{align*}
Then $J(r)=crlI_n'(rl)=c J_n(rl)$ with
\begin{align*}
J_n(x)=\sum_{k=0}^{+\infty}\dfrac{(|n|+2k)x^{|n|+2k}|n|!}{2^{ 2k}k!(|n|+k)!}.
\end{align*}
Thus, $J(r)= J_n(rl)/ J_n(l)$ due to $cJ_n(rl)=J(1)=1$.
\end{remark}

\begin{remark}\label{rem:J}
Let $J$ solve $\widehat{\Delta}_{(1)}J=0,\ J(0)=0, J(1)=1$. Then we have
\begin{align*}
   & Cr\geq J(r) \geq  r\mathrm{e}^{-|l|(1-r)} \geq r^{1+|l|},\quad \partial_rJ\geq 0.
\end{align*}

Indeed, it is easy to show that  if $\widehat{\Delta}_{(1)}g\leq 0,\ g|_{r=0,1}\geq 0$, then $g\geq 0$.
Let $ \phi(r)=r\mathrm{e}^{-|l|(1-r)}$.  We have
\begin{align*}\widehat{\Delta}_{(1)}\phi&=\frac{\partial_r(r\partial_r\phi)}{r} -\frac{1+r^2l^2}{r^2}\phi
=\frac{\partial_r((1+r|l|)\phi)}{r}-\frac{1+r^2l^2}{r^2}\phi\\&
=\frac{(1+r|l|)^2\phi}{r^2}+\frac{|l|\phi}{r}-\frac{1+r^2l^2}{r^2}\phi
=\frac{2r|l|\phi}{r^2}+\frac{|l|\phi}{r}\geq0,
\end{align*}
and  $\widehat{\Delta}_{(1)}(r)=\dfrac{r-1-r^2l^2}{r^2}\leq 0.$ On the other hand, $(J-\phi)|_{r=0,1}=(r-J)|_{r=0,1}=0$. Then we have $J-\phi\geq0$ and $r-J\geq 0$, which gives  
\begin{align*}
   &  & r\geq J(r) \geq  r\mathrm{e}^{-|l|(1-r)} \geq r^{1+|l|}.
\end{align*}
Since $J'(0)=\lim_{r\rightarrow0^{+}}J(r)/r\geq 0$ and $\partial_r(r\partial_rJ)=(1+r^2l^2)J/r\geq 0$, we deduce that $r\partial_rJ\geq \lim_{r\rightarrow0^{+}} (rJ)\geq 0$.
\end{remark}

\begin{remark}\label{rem:J-star}
 Let $J^{*}=(n^2+l^2)J/(n^2+r^2l^2)$. Then we have
  \begin{align*}
     & \widehat{\Delta}_{1}^{*}J^{*}=0,\quad J^*(0)=0, \quad J^*(1)=1.
  \end{align*}
Moreover, we have $|J^*(r)|\leq Cr$. Indeed,  let $\varphi_{1}=r(9n^2+r^2l^2)/(9n^2+l^2)$. Then $\varphi_1(0)=0,\ \varphi_1(1)=1$ and
  \begin{align*}
     &(9n^2+l^2)\widehat{\Delta}_{1}\varphi_1 = 9n^2\widehat{\Delta}_1(r) +l^2 \widehat{\Delta}_1(r^{3})\\
     &= 9n^2\Big(\dfrac{1}{r}-\dfrac{n^2+r^2l^2}{r}-\dfrac{2rl^2}{n^2+r^2l^2}\Big) +l^2\Big(9r-r(n^2+r^2l^2)-\dfrac{6r^3l^2}{n^2+r^2l^2}\Big)\\
     &\leq 9n^2\Big(\dfrac{1}{r}-\dfrac{n^2+r^2l^2}{r}\Big) +9rl^2 =9(1-n^2)(n^2/r+rl^2)\leq 0,
     \end{align*}
which implies that  $J(r)\leq r(9n^2+r^2l^2)/(9n^2+l^2)$. Thus, $J^*(r)\le 9r$.
  
\end{remark}

\subsection{Sobolev type inequality}

\begin{Lemma}\label{lem:sob}
If $f(0)=f(1)=0$, then we have
\begin{align*}
   &|\partial_rf(1)|^2\leq 2\|\widehat{\Delta}f\|_{L^2}\|\partial_rf\|_{L^2}\le C\|\widehat{\Delta}_1f\|_{L^2}\|\partial_rf\|_{L^2},\\
   &|\partial_rf(1)|\leq  C\|r\widehat{\Delta}_1f\|_{L^1}.
\end{align*}
\end{Lemma}
\begin{proof}
Using the facts that 
  \begin{align*}
     |\partial_rf(1)|^2&\leq \|(r\partial_rf)^2\|_{L^\infty}\leq \left\|\partial_r[(r\partial_rf)^2]\right\|_{L^1(\mathrm{d}r)}\leq 2\|r\partial_rf\partial_r(r\partial_rf)\|_{L^1(\mathrm{d}r)}\nonumber\\ &\leq 2\|\partial_r(r\partial_rf)\|_{L^2(rdr)}\|\partial_rf\|_{L^2(r\mathrm{d}r)}
  \end{align*}
 and 
  \begin{align*}
     \|\partial_r(r\partial_rf)-n^2f/r\|_{L^2}^2&=\|\partial_r(r\partial_rf)\|_{L^2} +\|n^2f/r\|_{L^2}^2-2\mathbf{Re}(\langle\partial_r(r\partial_rf),n^2f/r \rangle)\\
     &=\|\partial_r(r\partial_rf)\|_{L^2} +\|n^2f/r\|_{L^2}^2+ 2n^2\|\partial_rf\|_{L^2}^2,
  \end{align*}
we deduce that 
  \begin{align*}
     |\partial_rf(1)|^2&\leq 2\|\partial_r(r\partial_rf)\|_{L^2(r\mathrm{d}r)}\|\partial_rf\|_{L^2(r\mathrm{d}r)}\leq 2\|\partial_r(r\partial_rf)-n^2f/r\|_{L^2(r\mathrm{d}r)}\|\partial_rf\|_{L^2(r\mathrm{d}r)}\\
     &\leq 2\|\partial_r(r\partial_rf)/r-n^2f/(r^2)\|_{L^2(r\mathrm{d}r)}\|\partial_rf\|_{L^2(r\mathrm{d}r)}\leq 2\|\widehat{\Delta}f\|_{L^2}\|\partial_rf\|_{L^2}.
  \end{align*}
Using the fact that  
\beno
\|f\|_{1}^2\leq -2\mathbf{Re}\langle\widehat{\Delta}_{1}f,f\rangle\leq 2\|\widehat{\Delta}_1f\|_{L^2}\|f\|_{L^2}\leq 2|l|^{-1}\|\widehat{\Delta}_1f\|_{L^2}\|f\|_{1},
\eeno
we deduce that  $\|f\|_{1}\leq 2|l|^{-1}\|\widehat{\Delta}_1f\|_{L^2}$  and 
 \begin{align*}
    &\left\|\dfrac{2rl^2\partial_rf}{n^2+r^2l^2} \right\|_{L^2}\leq |l/n|\|\partial_rf\|_{L^2}\leq |l/n|\|f\|_{1}\leq 2\|\widehat{\Delta}_1f\|_{L^2},
 \end{align*}
which gives 
 \begin{align*}
 \|\widehat{\Delta}f\|_{L^2}\leq \|\widehat{\Delta}_1f\|_{L^2}+ \left\|\dfrac{2rl^2\partial_rf}{n^2+r^2l^2} \right\|_{L^2}\leq 3\|\widehat{\Delta}_1f\|_{L^2}.
 \end{align*}
Thus, we obtain 
\begin{align*}
 |\partial_rf(1)|^2\leq 2\|\widehat{\Delta}f\|_{L^2}\|\partial_rf\|_{L^2}\le 6\|\widehat{\Delta}_1f\|_{L^2}\|\partial_rf\|_{L^2}.
\end{align*}

Let $J^*$ solve the equation 
\begin{align*}
   &\widehat{\Delta}_{1}^*J^*=0,\quad J^*(0)=0,\quad J^*(1)=1.
\end{align*}  
By Remark \ref{rem:J-star}, we have $J^*\leq Cr$ and then
\begin{align*}
   &|\partial_rf(1)|=|\langle\widehat{\Delta}_{1}f, J^*\rangle| \leq \|J^*\widehat{\Delta}_1f\|_{L^1}\leq  C\|r\widehat{\Delta}_1f\|_{L^1},
\end{align*}
which gives the second inequality.
\end{proof}

\subsection{Some basic inequalities}

  \begin{Lemma}\label{lem:A}
  Let  $A(s)=|l(\lambda-s^2)/\nu|^{\f12}+|ls/\nu|^{\f13}+|l|+|n/s|$ and $A=A(1)$. Then  it holds that for $s\in(0,1]$,
  \begin{align*}
     &\int_s^1A(r)\mathrm{d}r\geq C^{-1}(1-s)A.
  \end{align*}
\end{Lemma}

\begin{proof}
   Let $\lambda_r=\mathbf{Re}\lambda,\ \lambda_i=\mathbf{Im}\lambda.$ We first claim that
   \begin{align}\label{eq:A-est1}
      &\int_{s}^{1}|\lambda_r-r^2|^{\f12}\mathrm{d}r\geq C^{-1}(1-s)|\lambda_r-1|^{\f12}.
   \end{align}
   By \eqref{eq:A-est1}, we have
   \begin{align*}
      &\int_{s}^{1}|\lambda-r^2|^{\f12}\mathrm{d}r\geq C^{-1}\int_{s}^{1}\big(|\lambda_r-r^2|^{\f12}+|\lambda_i|^{\f12}\big)\mathrm{d}r\\
      &\geq C^{-1}\big(|1-s||\lambda_r-1|^{\f12}+|1-s||\lambda_i|^{\f12}\big)\geq C^{-1}|1-s||\lambda-1|^{\f12}.
   \end{align*}
 It is easy to see that 
   \begin{align*}
      \int_{s}^{1}|r|^{\f13}\mathrm{d}r&\geq \int_{(s+1)/2}^{1}|r|^{\f13}\mathrm{d}r\geq  \int_{(s+1)/2}^{1}\big(\dfrac{s+1}{2}\big)^{\f13}\mathrm{d}r\\
      &\geq \dfrac{1-s}{2}\big(\dfrac{s+1}{2}\big)^{\f13}\geq \dfrac{1-s}{2^{4/3}},
   \end{align*}
   and
   \begin{align*}
      &\int_{s}^{1}|r|^{-1}\mathrm{d}r\geq \int_{s}^{1}1\mathrm{d}r=1-s.
   \end{align*}
 Summing up, we obtain
 \begin{align*}
    \int_{s}^{1}A(r)\mathrm{d}r&=\int_{s}^{1} \big(|l(\lambda-r^2)/\nu|^{\f12}+|lr/\nu|^{\f13}+|l|+|n/r|\big)\mathrm{d}r\\
    &\geq C^{-1}(1-s)\big(l(\lambda-1)/\nu|^{\f12}+|l/\nu|^{\f13}+|l|+|n|\big)=C^{-1}(1-s)A.
 \end{align*}
 
 It remains to prove \eqref{eq:A-est1}.
 Let $s_1=(s+1)/2$, then $s_1\in(1/2,1].$  \smallskip
 
 \no\textbf{Case A}. $\lambda_r\leq s_1^{2}.$ 
 
 In this case, we have
   \begin{align*}
     \int_{s}^{1}|\lambda_r-r^2|^{\f12}\mathrm{d}r&\geq \int_{(s_1+1)/2}^{1}|\lambda_r-r^2|^{\f12}\mathrm{d}r\geq \int_{(s_1+1)/2}^{1}(r^2-\lambda_r)^{\f12}\mathrm{d}r\\
     &\geq \int_{(s_1+1)/2}^{1}\left(\Big(\dfrac{s_1+1}{2}\Big)^2-\lambda_r\right)^{\f12}\mathrm{d}r =\dfrac{1-s_1}{2}\left(\dfrac{(s_1+1)^{2}-4\lambda_r}{4}\right)^{\f12}\\
     &= \dfrac{1-s}{4}\left(\dfrac{(s_1^2-\lambda_r)+2(s_1-\lambda_r)+1-\lambda_r}{4}\right)^{\f12}\\
     &\geq  \dfrac{1-s}{4} \left(\dfrac{1-\lambda_r}{4}\right)^{\f12}= \dfrac{(1-s)|\lambda_r-1|^{\f12}}{8}.
   \end{align*}
   Here we used $s_1^2-\lambda_r\geq 0$ and $2(s_1-\lambda_r)\geq 2(s_1^2-\lambda_r)\geq 0$.

\no \textbf{Case B}. $s_1^{2}\leq \lambda_r\leq 1$. 
   
   In this case, we have
   \begin{align*}
      &\int_{s}^{1}|\lambda_r-r^2|^{\f12}\mathrm{d}r\geq \int_{s}^{s_1}|\lambda_r-r^2|^{\f12}\mathrm{d}r= \int_{s}^{s_1}(\lambda_r-r^2)^{\f12}\mathrm{d}r\\
      &\geq \int_{s}^{s_1}\left(s_1^2-r^2\right)^{\f12}\mathrm{d}r\geq \int_{s}^{s_1}s_1^{\f12}\left(s_1-r\right)^{\f12}\mathrm{d}r=\frac{2}{3}s_1^{\f12}\left(s_1-s\right)^{\f32}\\&\geq \frac{2}{3}\Big(\f12\Big)^{\f12}\left(s_1-s\right)^{\f32} =\dfrac{\sqrt{2}}{3}\Big(\dfrac{1-s}{2}\Big)^{\f32}.
   \end{align*}
Thanks to $\lambda_r\geq s_1^2=\Big(\dfrac{s+1}{2}\Big)^2$, we have $1-\lambda_r\leq 1-\Big(\dfrac{s+1}{2}\Big)^2=\dfrac{1-s}{2}\cdot\dfrac{s+3}{2}\\ \leq (1-s)$. Then we conclude
   \begin{align*}
      &\int_{s}^{1}|\lambda_r-r^2|^{\f12}\mathrm{d}r\geq \dfrac{\sqrt{2}}{3}\big(\dfrac{1-s}{2}\big)^{\f32}\geq \dfrac{\sqrt{2}}{3}\Big(\dfrac{1-s}{2}\Big)\Big(\dfrac{1-\lambda_r}{2}\Big)^{\f12}=\dfrac{(1-s)|1-\lambda_r|^{\f12}}{6}.
   \end{align*}

  \no \textbf{Case C}. $1\leq \lambda_r$.
  
   In this case, we have
   \begin{align*}
      &\int_{s}^{1}|\lambda_r-r^2|^{\f12}\mathrm{d}r\geq \int_{s}^{1}|\lambda_r-1|^{\f12}\mathrm{d}r=(1-s)|\lambda_r-1|^{\f12}.
   \end{align*}
 Then  \eqref{eq:A-est1} follows by combining three cases. 
 \end{proof}
 
 \begin{Lemma} \label{lem:la-l}
 Let $\lambda\in\C,\ \lambda_r=\mathbf{Re}\lambda,\ \lambda_i=\mathbf{Im}\lambda$, and $a,b,l\in\R,\ a\geq 0,\ b\geq 0,\ l\neq 0$
 and $s\in (0,1]$. 
 If $2l\lambda_i\leq |l(\lambda-s^2)|$ or $3bl\lambda_i\leq a$, then it holds that 
 \beno
  |a|+|l(\lambda-s^2)b|\leq 2|a+{\ir}l(\lambda-s^2)b|. 
 \eeno
\end{Lemma} 

\begin{proof}
First of all, we consider the case of $s=1$.  If $2l\lambda_i\leq |l(\lambda-1)|$, then we have
   \begin{align*}|a+{\ir}l(\lambda-1)b|^2&=|a-l\lambda_ib|^2+(l(\lambda_r-1)b)^2 =a^2-2al\lambda_ib+(l\lambda_ib)^2+(l(\lambda_r-1)b)^2\\
      &\geq a^2-ab|l(\lambda-1)|+(l\lambda_ib)^2+(l(\lambda_r-1)b)^2= a^2-ab|l(\lambda-1)|+|l(\lambda-1)b|^2\\
      &=(|a|+|l(\lambda-1)b|)^2/4+3(|a|-|l(\lambda-1)b|)^2/4\geq (|a|+|l(\lambda-1)b|)^2/4.
       \end{align*} If $a\geq 3bl\lambda_i$, then we have
   \begin{align*}
      |a+{\ir}l(\lambda-1)b|^2&=a^2+(l\lambda_ib)^2-2al\lambda_ib+(l(\lambda_r-1)b)^2\\
      &\geq a^2+(l\lambda_ib)^2-2a^2/3+(l(\lambda_r-1)b)^2\geq a^2/3+(l\lambda_ib)^2+(l(\lambda_r-1)b)^2\\
      &=a^2/3+|l(\lambda-1)b|^2\geq \big(|a|+|l(\lambda-1)b|\big)^2/4.
   \end{align*}
   Here we used $f^2/3+g^2\geq (|f|+|g|)^2/4$. This proves the case of $s=1$.
   
   For $s\in (0,1]$,  let $\lambda_1={\lambda}-s^2+1$, then we have 
   \beno
   2l\mathbf{Im}(\lambda_1)\leq |l(\lambda_1-1)|\quad \text{or}\quad 3b\mathbf{Im}(\lambda_1)\le a.
   \eeno 
 This is reduced to the case of $s=1$.  Thus, 
    \begin{align*}
     & |a|+|l(\lambda_1-1)b|\leq 2|a+{\ir}l(\lambda_1-1)|,
  \end{align*}
  which gives $|a|+|l({\lambda}-s^2)b|\leq 2|a+{\ir}l({\lambda}-s^2)|$.  
    \end{proof}
 
 \subsection{Some estimates involving the  Airy function}\label{sec:airy}
 
 Let $Ai(y)$ be the Airy function, which is a nontrivial solution of $f''-yf=0$.  Let
\begin{align*}
f_1(y)=Ai(\mathrm{e}^{{\ir}\f{\pi}{6}}y),\quad f_2(y)=Ai(\mathrm{e}^{{\ir}\f{5\pi}{6}}y).
\end{align*}
Then $f_1$ and $f_2$ are two linearly independent solutions of $f''-{\ir}yf=0$.
 We denote
\begin{align*}
&A_0(z)=\int_{e^{{\ir}\pi/6}z}^{\infty}Ai(t)dt=e^{{i}\pi/6}\int_{z}^{\infty}Ai(e^{{\ir}\pi/6}t)dt.
\end{align*}

\begin{Lemma}\label{lem:Airy-p1}
There exists $c>0$ and $\delta_0>0$ so that for $\textbf{Im}z\le \delta_0$,
\begin{align}
&\left|\f{A_0'(z)}{A_0(z)}\right|\lesssim1+|z|^{\f12},\quad {\rm Re}\f{A_0'(z)}{A_0(z)}\leq\min(-1/3,-c(1+|z|^{\f12})).
\end{align}
Moreover, for ${{\bf Im }z}\le \delta_0$, we have 
\beno
 \Big|\f{A_0''(z)}{A_0(z)}\Big|\le C(1+|z|).
 \eeno
\end{Lemma} 

\begin{proof}
The first part of the lemma comes from \cite{CLWZ}. 
Let $F(z)=\f{A_0'(z)}{A_0(z)}$. Then $F$ is analytic for ${{\bf Im }z}\leq 2\delta_0$. By Cauchy formula, we have
\begin{align*}
F'(z)=\f1{2\pi i}\oint_{|\zeta-z|=r}\f{F(\zeta)d\zeta}{(\zeta-z)^2},\quad \forall\ {{\bf Im }z}\leq \delta_0,\ r=\delta_0.
\end{align*}
Then we have
\begin{align*}
&|F'(z)|\leq\f1{r}\sup_{|\zeta-z|=r}|F(\zeta)|\leq C\f{1+(|z|+r)^{\f12}}{r}\leq C(1+|z|^{\f12}),
\end{align*}
which gives
\begin{align*}
\Big|\f{A_0''(z)}{A_0(z)}\Big|=\big|F'(z)+F(z)^2\big|\leq C(1+|z|^{\f12})+ C(1+|z|^{\f12})^2\leq C(1+|z|).
\end{align*} 
\end{proof}

We denote 
\beno
L=|2l/\nu|^{\f13},\quad  d=(\lambda-1)/2-{\ir}(n^2+l^2)\nu/(2l).
\eeno
We define $w(r)$ in the following way: if $l<0$, let
\begin{align*}
w(r)=Ai\big(\mathrm{e}^{{\ir}\f{5\pi}{6}}L(r-1-d)\big)/Ai\big(-\mathrm{e}^{{\ir}\f{5\pi}{6}}Ld\big),
\end{align*}
and if $l>0,$ let
\begin{align*}
w(r)=Ai\big(-\mathrm{e}^{{\ir}\f{\pi}{6}}L(r-1-d)\big)/Ai\big(\mathrm{e}^{{\ir}\f{\pi}{6}}Ld\big)
=A_0'\big(L(d+1-r)\big)/A_0'\big(Ld\big).
\end{align*}
It is easy to verify that 
\begin{align*}
&-\nu\partial_r^2w+\nu(n^2+l^2)w+{\ir}l(\lambda-2r+1)w=0,\quad w(1)=1.
\end{align*}

\begin{Lemma}\label{lem:Airy-w}
Let $\delta_0$ be  in the Lemma \ref{lem:Airy-p1} and fix $C_0>0$.  If $\mathbf{Im}(Ld)\leq \delta_0$ and $|Ld|\leq C_0$, then 
for $t>0,\ r\in(0,1]$, 
\begin{align*}
|A_0\big(Ld+t)\big)|\leq \mathrm{e}^{-t/3}|A_0\big(Ld\big)|,
\end{align*}
and 
\begin{align*}
&|w(r)|\leq C\mathrm{e}^{-L(1-r)/4},\quad |\partial_rw(r)|\leq CL\mathrm{e}^{-L(1-r)/4},\quad \|w\|_{L^2}\leq CL^{-\f12},\\ 
&\|\partial_rw\|_{L^2}\leq CL^{\f12},\ \|(1-r)^2w\|_{L^2}\leq CL^{-\f52},\ \|(1-r^2)w\|_{L^2}\leq CL^{-\f32}.
\end{align*}
\end{Lemma}

\begin{proof}
 Without loss of generality,  we may assume $l>0$. By Lemma \ref{lem:Airy-p1}, we have
  \begin{align*}
     \left|\dfrac{A_0(Ld+t)}{A_0(t)}\right|&= \left|\exp\big(\ln(A_0(Ld+t))-\ln(A_0(Ld))\big)\right| = \left|\exp\bigg(\int_{0}^{t}\dfrac{A_0'(Ld+s)}{A_0(Ld+s)}\mathrm{d}s\bigg)\right|\\
     &\leq  \exp\bigg(\int_{0}^{t}\mathbf{Re}\dfrac{A_0'(Ld+s)}{A_0(Ld+s)}\mathrm{d}s\bigg) \leq \exp\bigg(-\int_{0}^{t}1/3\mathrm{d}s\bigg)\leq \mathrm{e}^{-t/3}.
  \end{align*}
  This shows that $|A_0\big(Ld+t)\big)|\leq \mathrm{e}^{-t/3}|A_0\big(Ld\big)|$. 
  
  Thanks to $\mathbf{Re}\dfrac{A_0'(z)}{A_0(z)}\leq -1/3<0$, we have $\left|\mathbf{Re}\dfrac{A_0'(z)}{A_0(z)}\right| \geq 1/3$ and 
  \begin{align*}
     & \left|\dfrac{A_0(z)}{A_0'(z)}\right|=\left|\dfrac{A_0'(z)}{A_0(z)}\right|^{-1}\leq \left|\mathbf{Re}\dfrac{A_0'(z)}{A_0(z)}\right|^{-1} \leq 3,
  \end{align*}
which along with Lemma \ref{lem:Airy-p1} gives
\begin{align*}
   |w(r)|&=\left|\dfrac{A_0'(L(d+1-r))}{A_0'(Ld)}\right| =\left|\dfrac{A_0(Ld)}{A_0'(Ld)}\right| \left|\dfrac{A_0(L(d+1-r))}{A_0(Ld)}\right| \left|\dfrac{A_0'(L(d+1-r))}{A_0(L(d+1-r))}\right| \\
   &\leq 3e^{-L(1-r)/3}\left|\dfrac{A_0'(L(d+1-r))}{A_0(L(d+1-r))}\right| \leq C\big(1+|Ld|+L(1-r)\big)^{\f12}e^{-L(1-r)/3}\\
   &\leq Ce^{-L(1-r)/4},
\end{align*} 
and 
\begin{align*}
   |\partial_rw(r)| &=L\left|\dfrac{A_0(Ld)}{A_0'(Ld)}\right| \left|\dfrac{A_0(L(d+1-r))}{A_0(Ld)}\right| \left|\dfrac{A_0''(L(d+1-r))}{A_0(L(d+1-r))}\right| \\
   &\leq 3Le^{-L(1-r)/3}\left|\dfrac{A_0''(L(d+1-r))}{A_0(L(d+1-r))}\right| \leq CL(1+|Ld|+L(1-r))e^{-L(1-r)/3}\\
   &\leq CLe^{-(1-r)/4}.
\end{align*} 
Then it is easy to verify that
\begin{align*}
     &\|w\|_{L^2}\leq \left(\int_{0}^{1}C\mathrm{e}^{-L(1-r)/2}r\mathrm{d}r\right)^{\f12}\leq CL^{-\f12},\\
     &\|\partial_rw\|_{L^2}\leq \left(\int_{0}^{1}CL^2\mathrm{e}^{-L(1-r)/2}r\mathrm{d}r\right)^{\f12}\leq CL^{\f12},\\
     & \|(1-r)^2w\|_{L^2}\leq \left(\int_{0}^{1}C(1-r)^2\mathrm{e}^{-L(1-r)/2}r\mathrm{d}r\right)^{\f12}\leq CL^{-\f52},\\
    & \|(1-r^2)w\|_{L^2}\leq C\left(\int_{0}^{1}(1-r)^2\mathrm{e}^{-L(1-r)/2}r\mathrm{d}r\right)^{\f12}\leq CL^{-\f32}.
    \end{align*}
    
This proves the lemma.
\end{proof}

\section*{Acknowledgement}
Z. Zhang is partially supported by NSF of China under Grant 11425103.

\end{document}